\documentclass[11pt,leqno]{amsart}
\usepackage{hyperref}
\usepackage{epsfig}
\usepackage{enumitem}
\usepackage{amsmath, amssymb}
\usepackage{amscd}
\usepackage[all,cmtip,matrix,arrow]{xy}
\usepackage{mathrsfs}
\usepackage{graphicx}
\usepackage{tikz-cd}
\usepackage{mathabx}
\usepackage{array}
\usepackage{longtable}
\xyoption{rotate}
\xyoption{pdf}
\usepackage{xparse}

\newcommand{\Implies}[2]{$\text{\ref{#1}}\implies\text{\ref{#2}}$}

\makeatletter
\DeclareRobustCommand*{\bfseries}{%
  \not@math@alphabet\bfseries\mathbf
  \fontseries\bfdefault\selectfont
  \boldmath
}
\makeatother

\AtBeginEnvironment{tikzcd}{\setlength{\mathsurround}{0pt}}

\newtheorem{theo}{Theorem}[section]
\newtheorem{lemma}[theo]{Lemma}
\newtheorem{defi}[theo]{Definition}
\newtheorem{prop}[theo]{Proposition}

\newtheorem{cor}[theo]{Corollary}
\newtheorem{remark}[theo]{Remark}

\newtheorem{example}[theo]{Example}
\newtheorem{ques}[theo]{Question}
\numberwithin{equation}{section}

\mathchardef\mhyphen="2D

\def\A{{\mathbb A}}

\def\N{\mathbb{N}}
\def\bL{\mathbb{L}}

\def\R{\mathbb{R}}
\def\bS{\mathbb{S}}
\def\C{\mathbb{C}}
\def\Z{\mathbb{Z}}
\def\Q{\mathbb{Q}}

\def\wt{\widetilde}

\def\bL{{\mathbf L}}

\def\PP{{\mathbb P}}

\def\pre-tr{\operatorname{pre-tr}}
\def\h{\operatorname{h}}

\def\Hom{\operatorname{Hom}}

\def\Map{\operatorname{Map}}
\def\End{\operatorname{End}}

\def\Pic{\operatorname{Pic}}
\def\rad{\operatorname{rad}}
\def\gr{\operatorname{gr}}

\DeclareMathOperator*{\colim}{colim}

\newcommand{\tens}[1]{%
  \mathbin{\mathop{\otimes}\displaylimits_{#1}}%
}

\newcommand{\Ltens}[1]{%
  \mathbin{\mathop{\otimes}\displaylimits^{\bL}_{#1}}%
}

\newcommand{\laxtimes}[1]{\mathop{\times\mkern-13mu\raise1.3ex\hbox{$\scriptscriptstyle\to$}_{#1}}}

\newcommand{\hy}{\mhyphen}

\newcommand{\indlim}[1][]{\mathop{\varinjlim}\limits_{#1}}
\newcommand{\inddlim}[1][]{{``{\indlim[#1]}"}}
\newcommand{\prolim}[1][]{\mathop{\varprojlim}\limits_{#1}}

\newcommand{\biggplus}[1][]{\mathop{\bigoplus}\limits_{#1}}

\newcommand{\prodd}[1][]{\mathop{\prod}\limits_{#1}}

\newcommand{\bbar}{\overline}
\newcommand{\hhat}{\widehat}

\newcommand{\xto}{\xrightarrow}

\newcommand{\hto}{\hookrightarrow}

\newcommand{\onto}{\twoheadrightarrow}

\newcommand{\Ch}{\operatorname{Ch}}

\newcommand{\lex}{\operatorname{lex}}

\newcommand{\rex}{\operatorname{rex}}

\newcommand{\Groth}{\operatorname{Groth}}
\newcommand{\Mot}{\operatorname{Mot}}

\newcommand{\Nuc}{\operatorname{Nuc}}
\newcommand{\Solid}{\operatorname{Solid}}
\newcommand{\Coh}{\operatorname{Coh}}

\newcommand{\Proj}{\operatorname{Proj}}
\newcommand{\Inj}{\operatorname{Inj}}
\newcommand{\Flat}{\operatorname{Flat}}

\newcommand{\Cat}{\operatorname{Cat}}

\newcommand{\Cone}{\operatorname{Cone}}

\newcommand{\Fiber}{\operatorname{Fiber}}

\newcommand{\loc}{\operatorname{loc}}

\newcommand{\cont}{\operatorname{cont}}
\newcommand{\red}{\operatorname{red}}

\newcommand{\eff}{\operatorname{eff}}
\newcommand{\Eff}{\operatorname{Eff}}
\newcommand{\ex}{\operatorname{ex}}

\newcommand{\St}{\operatorname{St}}

\newcommand{\st}{\operatorname{st}}

\newcommand{\dual}{\operatorname{dual}}

\newcommand{\deff}{\operatorname{def}}
\newcommand{\Fil}{\operatorname{Fil}}

\newcommand{\bE}{{\mathbb E}}

\newcommand{\mk}{\mathrm k}
\newcommand{\ev}{\mathrm ev}

\newcommand{\cF}{{\mathcal F}}

\newcommand{\cO}{{\mathcal O}}

\newcommand{\cD}{{\mathcal D}}

\newcommand{\cA}{{\mathcal A}}
\newcommand{\cB}{{\mathcal B}}

\newcommand{\cC}{{\mathcal C}}
\newcommand{\cE}{{\mathcal E}}

\newcommand{\cY}{{\mathcal Y}}
\newcommand{\cU}{{\mathcal U}}

\newcommand{\cS}{{\mathcal S}}
\newcommand{\cT}{{\mathcal T}}

\newcommand{\veps}{\varepsilon}

\newcommand{\un}{\underline}
\newcommand{\la}{\langle}
\newcommand{\ra}{\rangle}

\newcommand{\incl}{\operatorname{incl}}

\newcommand{\Cohass}{\operatorname{CohAss}}

\newcommand{\TR}{\operatorname{TR}}
\newcommand{\TC}{\operatorname{TC}}
\newcommand{\HH}{\operatorname{HH}}
\newcommand{\HC}{\operatorname{HC}}
\newcommand{\HP}{\operatorname{HP}}

\newcommand{\CAlg}{\operatorname{CAlg}}

\newcommand{\Fun}{\operatorname{Fun}}

\newcommand{\Ab}{\operatorname{Ab}}
\newcommand{\ab}{\operatorname{ab}}

\newcommand{\Add}{\operatorname{Add}}
\newcommand{\Exact}{\operatorname{Exact}}

\newcommand{\Perf}{\operatorname{Perf}}

\newcommand{\perf}{\operatorname{perf}}

\newcommand{\Kar}{\operatorname{Kar}}

\newcommand{\PSh}{\operatorname{PSh}}

\newcommand{\Shv}{\operatorname{Shv}}

\newcommand{\coker}{\operatorname{coker}}

\newcommand{\Ext}{\operatorname{Ext}}

\newcommand{\Id}{\operatorname{Id}}

\newcommand{\Sp}{\operatorname{Sp}}

\newcommand{\add}{\operatorname{add}}

\newcommand{\Ind}{\operatorname{Ind}}

\newcommand{\Calk}{\operatorname{Calk}}

\newcommand{\Spec}{\operatorname{Spec}}

\newcommand{\Vect}{{\rm Vect}}

\newcommand{\Alg}{\operatorname{Alg}}

\newcommand{\id}{\operatorname{id}}

\newcommand{\tors}{\operatorname{tors}}

\newcommand{\compl}{\operatorname{compl}}

\newcommand{\Ac}{\operatorname{Ac}}

\newcommand{\iin}{\operatorname{in}}
\newcommand{\pr}{\operatorname{pr}}

\newcommand{\Mod}{\operatorname{Mod}}

\newcommand{\inv}{\operatorname{inv}}

\newcommand{\Nil}{\operatorname{Nil}}

\usepackage{amscd}

\title[Theorem of the heart for Weibel's homotopy $K$-theory]
{Theorem of the heart for Weibel's homotopy $K$-theory}

\author{Alexander I. Efimov}
\address{Steklov Mathematical Institute of RAS, Gubkin St. 8, GSP-1, Moscow 119991, Russia}
\email{efimov@mccme.ru}

\begin{document}

\begin{abstract}In this paper we prove the theorem of the heart for Weibel's homotopy $K$-theory $KH.$ Namely, if $\cC$ is a small stable $\infty$-category with a bounded $t$-structure, then the realization functor $D^b(\cC^{\heartsuit})\to \cC$ induces an equivalence of spectra $KH(\cC^{\heartsuit})\xto{\sim}KH(\cC).$ In a certain sense this result is dual to the Dundas-Goodwillie-McCarthy theorem. We deduce the d\'evissage theorem for $KH$ of abelian categories, also on the level of spectra (in all degrees). More generally, we prove these results for dualizable categories with nice $t$-structures and for the so-called coherently assembled abelian categories.
	
The proof is heavily based on another new result, which is a much stronger version of Barwick's theorem of the heart. Its special case states the following: if $\cC$ is a small stable category with a bounded $t$-structure, such that for some $n\geq 1$ the realization functor induces isomorphisms on $\Ext^{\leq n}$ between the objects of $\cC^{\heartsuit},$ then the map $K_j(\cC^{\heartsuit})\to K_j(\cC)$ is an isomorphism for $j\geq -n-1,$ and a monomorphism for $j=-n-2.$ Moreover, we prove that these estimates are sharp, even for dg categories over a field. In particular the naive $K$-theoretic theorem of the heart fails for $K_{-3}.$
\end{abstract}

%\keywords{Mac Lane homology and cohomology; Wieferich primes; critical points; spectral sequences; cubical construction.}

\maketitle

\tableofcontents

\section{Introduction}

In this paper we study the non-connective $K$-theory and Weibel's homotopy $K$-theory $KH$ of small stable $\infty$-categories with bounded $t$-structures, and more generally for dualizable categories with suitably defined nice $t$-structures. In the latter case we consider the invariants $K^{\cont}$ and $KH^{\cont} $as defined in \cite{E24}. To fix the ideas, in the introduction we give the formulations of our main results in the context of small categories, mentioning the more general versions for dualizable categories and the so-called coherently assembled abelian categories (Definition \ref{def:coh_ass_abelian_cats}). For brevity we will say ``stable category'' instead of ``stable $\infty$-category''.

Originally, Weibel's definition of $KH$ was given in the $\Z$-linear context \cite{Wei89, Wei13, Tab15}. The natural generalization for stable categories is explained for example in \cite[Section 8, Corollary 8.3]{E25c}. For a small stable category $\cC$ the spectrum $KH(\cC)$ is given by the geometric realization \begin{equation}\label{eq:definition_of_KH_intro}
KH(\cC)=\colim\limits_{[n]\in\Delta^{op}} K(\cC\otimes\Perf(\bS[x_1,\dots,x_n])),
\end{equation}
where for $n\geq 0$ we put
\begin{equation*}
\bS[x_1,\dots,x_n]=\bS[x_1]\otimes\dots\otimes\bS[x_n],\quad \bS[x]=\Sigma^{\infty}(\N_+),
\end{equation*}
$\N_+=\N\sqcup\{\ast\}.$ For a small abelian category $\cA$ we put $KH(\cA)=KH(D^b(\cA)).$

Our main result is the following theorem, which is a special case of Theorem \ref{th:theorem_of_the_heart_for_KH}.

\begin{theo}\label{th:theorem_of_the_heart_for_KH_intro}
Let $\cC$ be a small stable category with a bounded $t$-structure $(\cC_{\geq 0},\cC_{\leq 0}).$ Then the realization functor $D^b(\cC^{\heartsuit})\to \cC$ induces an equivalence of spectra 
\begin{equation*}
KH(\cC^{\heartsuit})\xto{\sim}KH(\cC).
\end{equation*}
\end{theo}

This theorem can be considered as a dual version of the Dundas-Goodwillie-McCarthy theorem \cite[Theorem 7.0.0.2]{DGM13}, see below for an explanation of this analogy. We also deduce the following d\'evissage theorem for $KH,$ which is a special case of Theorem \ref{th:devissage_for_K_and_KH} \ref{devissage_for_KH}.

\begin{theo}\label{th:devissage_for_KH_intro}
Let $F:\cB\to\cA$ be an exact fully faithful functor between small abelian categories such that $F(\cB)$ generates $\cA$ via extensions (equivalently, $F(\cB)\subset\cA$ satisfies the assumptions of Quillen's D\'evissage theorem). Then $F$ induces an equivalence of spectra
\begin{equation*}
KH(\cB)\xto{\sim}KH(\cA).
\end{equation*}
\end{theo}

The proof of Theorem \ref{th:theorem_of_the_heart_for_KH_intro} is based on the following much stronger version of Barwick's theorem of the heart \cite[Theorem 6.1]{Bar15}. This is a special case of Theorem \ref{th:coconnectivity_estimates}.

\begin{theo}\label{th:coconnectivity_estimates_intro}
Let $\cC$ and $\cD$ be small stable categories with bounded $t$-structures $(\cC_{\geq 0},\cC_{\leq 0})$ resp. $(\cD_{\geq 0},\cD_{\leq 0}).$ Let $n\geq 1$ be an integer and suppose that we have an exact and $t$-exact functor $F:\cC\to\cD,$ such that the following conditions hold:
\begin{itemize}
	\item the image $F(\cC)$ generates $\cD$ as an idempotent-complete stable subcategory;
	\item for $x,y\in\cC^{\heartsuit}$ the map $\Ext^i_{\cC}(x,y)\to\Ext^i_{\cD}(F(x),F(y))$ is an isomorphism for $i\leq n-1,$ and a monomorphism for $i=n.$
\end{itemize}
Then the induced map $K_j(\cC)\to K_j(\cD)$ is an isomorphism for $j\geq -n,$ and a monomorphism for $j = -n-1.$
\end{theo}

We mention the following immediate corollaries, which are again special cases of more general results for large categories (Corollary \ref{cor:theorem_of_the_heart_under_condition_on_Exts} and Theorem \ref{th:devissage_for_K_and_KH} \ref{devissage_for_K}).

\begin{cor}\label{cor:refined_th_of_the_heart_intro}
Let $\cC$ be a small stable category with a bounded $t$-structure $(\cC_{\geq 0},\cC_{\leq 0}).$ Suppose that for some $n\geq 1$ the maps $\Ext^i_{\cC^{\heartsuit}}(x,y)\to\Ext^i_{\cC}(x,y)$ are isomorphisms for $i\leq n$ and $x,y\in\cC^{\heartsuit}$ (this always holds for $n=1$). Then the map $K_j(\cC^{\heartsuit})\to K_j(\cC)$ is an isomorphism for $j\geq -n-1,$ and a monomorphism for $j=-n-2.$ 
\end{cor}

Note that in the situation of the above corollary the maps $\Ext^{n+1}_{\cC^{\heartsuit}}(x,y)\to \Ext^{n+1}_{\cC}(x,y)$ are automatically monomorphisms for $x,y\in\cC^{\heartsuit}.$

\begin{cor}\label{cor:devissage_for_K_intro}
Let $F:\cB\to\cA$ be a fully faithful exact functor between small abelian categories, such that $F(\cB)$ generates $\cA$ via extensions. Then the induced map $K_j(\cB)\to K_j(\cA)$ is an isomorphism for $j\geq -1,$ and a monomorphism for $j = -2.$
\end{cor}

We note that the above results are new only when the hearts are non-noetherian (and non-artinian). Namely, if $\cC^{\heartsuit}$ is noetherian, then we have $K_{<0}(\cC)=0$ by a theorem of Antieau-Gepner-Heller \cite[Theorem 3.6]{AGH19}, and we have $K(\cC)\simeq KH(\cC)$ (more precisely, $K(\cC)\simeq K(\cC\otimes \Perf(\bS[x_1,\dots,x_n]))$ for all $n\geq 0$). 

We prove that the estimates in Corollaries \ref{cor:refined_th_of_the_heart_intro} and \ref{cor:devissage_for_K_intro} are sharp (Theorem \ref{th:sharpness_of_estimates}), even if we are dealing with small dg categories resp. small abelian categories over a field. In particular we obtain that the naive $K$-theoretic theorem of the heart fails for $K_{-3},$ which is a much more precise version of a result of Ramzi-Sosnilo-Winges \cite[Theorem B]{RSW24}.

We now explain a very surprising analogy with the familiar results on connective $\bE_1$-ring spectra. Intuitively the notion of a $t$-structure is dual to the notion of a weight structure \cite{Bon, Sos19}. In particular the bounded $t$-structures are ``dual'' to bounded weight structures. Consider an idempotent-complete stable category $\cC$ with a bounded weight structure $(\cC_{w\geq 0},\cC_{w\leq 0}).$ For simplicity we assume that $\cC$ is generated by a single object $G$ (this assumption is harmless, it will make no essential difference in what follows).  In this case we may assume that $G$ is of weight zero. Putting $R=\End_{\cC}(G)$ (meaning the spectrum of endomorphisms), we obtain $\cC\simeq\Perf(R),$ and $R$ is a connective $\bE_1$-ring. Conversely, for a connective $\bE_1$-ring $R$ the category $\Perf(R)$ has a unique bounded weight structure such that $R$ has weight zero \cite{Sos19}.

The following statement is well-known. It is due to Waldhausen \cite{Wal78} for simplicial rings and connective $K$-theory, and the general case is proven by Land and Tamme \cite{LT19} (the case $n=0$ follows from the proof of \cite[Corollary 3.5]{LT19}).

\begin{prop}\cite[Proposition 1.1]{Wal78} \cite[Lemma 2.4]{LT19} \label{prop:connectivity_estimates}
Let $n\geq 0$ be an integer and let $f:R\to S$ be a map of connective $\bE_1$-rings, such that the map $\pi_i(R)\to\pi_i(S)$ is an isomorphism for $i\leq n-1,$ and an epimorphism for $i = n.$ If $n=0,$ we require that the ideal $\ker(\pi_0(f))\subset \pi_0(R)$ is nilpotent. Then the induced map $K_j(R)\to K_j(S)$ is an isomorphism for $j\leq n,$ and an epimorphism for $j = n+1.$
\end{prop}

It is clear from the formulations that our Theorem \ref{th:coconnectivity_estimates_intro} for some $n\geq 1$ is ``dual'' to Proposition \ref{prop:connectivity_estimates} for $n-1:$ both in the assumptions and in the conclusions one simply replaces $\pi_j$ with $\pi_{-1-j}$ for the relevant spectra, and one replaces epimorphisms with monomorphisms. Continuing this analogy, Corollary \ref{cor:refined_th_of_the_heart_intro} for some $n\geq 1$ is ``dual'' to the special case of Proposition \ref{prop:connectivity_estimates} for the same $n$ when $S=\pi_0(R)$ and $\pi_i(R)=0$ for $1\leq i\leq n-1.$ Further, Corollary \ref{cor:devissage_for_K_intro} is ``dual'' to the case $n=0$ and $R=\pi_0(R),$ $S=\pi_0(S).$  We do not have any conceptual explanation for why such a duality takes place, but we note that in both contexts the (co)connectivity estimates are sharp. 

Now we explain why from this point of view our Theorem \ref{th:theorem_of_the_heart_for_KH_intro} is ``dual'' to the Dundas-Goodwillie-McCarthy theorem. We recall some notation. We denote by $\Cat^{\perf}$ the category of idempotent-complete stable categories and consider the universal finitary localizing invariant $\cU_{\loc}:\Cat^{\perf}\to\Mot^{\loc}$ \cite{BGT, E25c}. The category $\Mot^{\loc}$ is naturally symmetric monoidal, and moreover it is rigid by \cite[Theorem 0.1]{E25c} (we mean rigidity in the sense of Gaitsgory and Rozenblyum \cite[Definition 9.12]{GaRo17}). We also consider the universal finitary $\A^1$-invariant localizing invariant $\cU_{\loc}^{\A^1}:\Cat^{\perf}\to\Mot^{\loc,\A^1},$ so $\Mot^{\loc,\A^1}$ is a quotient of $\Mot^{\loc}$ by the cocomplete ideal generated by the reduced motive of the flat affine line $\wt{\cU}_{\loc}(\bS[x]).$ By \cite[Theorem 8.1]{E25c} this is a smashing localization. More precisely, we have the idempotent $\bE_{\infty}$-algebra in $\Mot^{\loc},$ given by the geometric realization $A=|\cU_{\loc}(\Delta^{\bullet})|,$ where $\Delta^n=\Spec\bS[x_1,\dots,x_n].$ We have an equivalence 
\begin{equation*}
\Mot^{\loc,\A^1}\simeq\Mod\hy A\subset\Mot^{\loc}.
\end{equation*}
The inclusion functor $i_*:\Mot^{\loc,\A^1}\to\Mot^{\loc}$ has a left adjoint $i^*$ and a right adjoint $i^!,$ given by
\begin{equation*}
i^*(M)=M\otimes A,\quad i^!(M)=\un{\Hom}(A,M),
\end{equation*}
where $\un{\Hom}$ is the internal $\Hom$ in $\Mot^{\loc}.$ We also denote by $\Gamma:\Mot^{\loc,\A^1}\to\Sp$ the functor of global sections, i.e. $\Gamma=\Hom(1,-),$ where $1$ is the unit object $A.$

By \cite[Theorem 9.3]{E25c} (whose proof of course uses the DGM theorem) for a connective $\bE_1$-ring $R$ we have a natural equivalence of spectra
\begin{equation}\label{eq:corepresentability_of_K^inv}
\Gamma(i^!(\cU_{\loc}(R)))\simeq K^{\inv}(R) = \Fiber(K(R)\to \TC(R)),
\end{equation}
where $\cU_{\loc}(R)=\cU_{\loc}(\Perf(R)),$ and $\TC(R)$ is the topological cyclic homology \cite{BHM93, NS18}. The map $K(R)\to \TC(R)$ is the cyclotomic trace. In view of the equivalence \eqref{eq:corepresentability_of_K^inv}, the DGM theorem states that we have an equivalence
\begin{equation*}
\Gamma(i^!(\cU_{\loc}(R)))\xto{\sim} \Gamma(i^!(\cU_{\loc}(\pi_0(R)))).
\end{equation*}
More precisely, \cite[Theorem 7.0.0.2]{DGM13} states a more general version for a quotient of $\pi_0(R)$ by a nilpotent ideal, but this is automatic by \cite[Corollary 3.5]{LT19}.

Now, by \cite[Corollary 8.3]{E25c} for $\cC\in\Cat^{\perf}$ we have 
\begin{equation*}
\Gamma(i^*(\cU_{\loc}(\cC)))\simeq KH(\cC).
\end{equation*} Therefore, the statement of Theorem \ref{th:theorem_of_the_heart_for_KH_intro} is ``dual'' to the DGM theorem if we continue the above analogy. This was in fact our motivation for why theorem of the heart should hold for $KH.$

The original goal of this paper was much more modest: to prove the generalizations of Barwick's theorem of the heart \cite[Theorem 6.1]{Bar15} and Quillen's D\'evissage theorem \cite[Theorem 4]{Qui73} for large categories. However, the methods which we developed allowed for much better results. We briefly explain which large abelian categories we consider, and which $t$-structures on large categories are ``nice'' in this context.

A Grothendieck abelian category $\cA$ is {\it coherently assembled} if the colimit functor $\colim:\Ind(\cA)\to\cA$ has an {\it exact} left adjoint $\hat{\cY}:\cA\to\Ind(\cA).$ Recall that the mere existence of this left adjoint means that the category is compactly assembled \cite[Definition 21.1.2.1]{Lur18} (since $\cA$ is an ordinary category, this means that $\cA$ is continuous in the sense of Joyal and Johnstone \cite{JJ}). In our situation this is equivalent to the condition that $\cA$ satisfies the Grothendieck's axiom (AB6) \cite{Gro}. As explained in \cite[Proposition E.8]{E24}, if $\cA$ is a compactly generated Grothendieck abelian category, then $\cA$ is coherently assembled if and only if $\cA$ is locally coherent (i.e. equivalent to an ind-completion of a small abelian category).

Recall that if $\cA$ is a Grothendieck abelian category, then its {\it unseparated derived category} $\check{D}(\cA)$ is identified with the (dg nerve of) the dg category $\Ch(\Inj_{\cA})$ of unbounded complexes of injective objects of $\cA,$ see \cite{Kr05, Kr15} and \cite[Section C.5.8]{Lur18}. If $\cA$ is moreover coherently assembled, then the category $\check{D}(\cA)$ is dualizable (Propositions \ref{prop:St_of_Grothendieck_abelian} and \ref{prop:St_of_coherently_assembled}). Hence, we can define the continuous $K$-theory of $\cA:$
\begin{equation*}
K^{\cont}(\cA)=K^{\cont}(\check{D}(\cA)).
\end{equation*}

If $\cA$ is locally coherent, then $\check{D}(\cA)\simeq\Ind(D^b(\cA^{\omega}))$ by \cite[Theorem 4.9]{Kr15}, hence $K^{\cont}(\cA)\simeq K(\cA^{\omega}).$ 

Now, if $\cC$ is a dualizable (presentable stable) category, then a $t$ -structure $(\cC_{\geq 0},\cC_{\leq 0})$ is called {\it compactly assembled} if it is accessible, compatible with filtered colimits and the functor $\hat{\cY}:\cC\to \Ind(\cC)$ (left adjoint to $\colim$) is $t$-exact. In this case the heart $\cC^{\heartsuit}$ is automatically coherently assembled. We further impose a condition that the $t$-structure is {\it continuously bounded}, which means that $\cC$ is generated by $\cC^b$ (equivalently, by $\cC^{\heartsuit}$) as a localizing subcategory. For brevity we say that $\cC$ is a dualizable $t$-category if it is a dualizable category equipped with a compactly assembled continuously bounded $t$-structure.

For a dualizable $t$-category $\cC$ we have a (strongly continuous) realization functor $\check{D}(\cC^{\heartsuit})\to\cC,$ which follows from the universal property of $\check{D}(\cA)$ explained in \cite{Lur17b}, and also from the universal property of $\check{D}(\cA)_{\geq 0}$ from \cite[Proposition C.5.5.20, Theorem C.5.8.8]{Lur18}. If $\cC$ is compactly generated, then the $t$-structure on $\cC$ induces a bounded $t$-structure on $\cC^{\omega},$ and the functor $\check{D}(\cC^{\heartsuit})\to\cC$ is identified with the ind-completion of the realization functor $D^b(\cC^{\omega,\heartsuit})\to\cC^{\omega}.$ With this in mind, it is clear that the above statements for small stable categories with bounded $t$-structures can at least be formulated more generally for dualizable $t$-categories.

However, one has to be cautious with generalizations to the dualizable context. Namely, for a small stable category $\cD$ with a bounded $t$-structure we have $K_{-1}(\cD)=0$ by a theorem of Antieau-Gepner-Heller \cite[Theorem 2.35]{AGH19}. However, for a dualizable $t$-category $\cC$ in general we have $K_{-1}^{\cont}(\cC)\ne 0,$ even if the realization functor $\check{D}(\cC^{\heartsuit})\to\cC$ is an equivalence (Corollary \ref{cor:K_-1_not_vanishing}). This is a very rare example of a situation when a reasonable statement about dualizable categories fails in general, but holds for compactly generated categories.

We mention the following natural questions which are not studied in this paper. The first one is somewhat vague, and it is motivated by the work of Levy and Sosnilo \cite{LS25} and by the above analogy (duality) with weight structures.

\begin{ques}\label{ques:non_coherence_rank_positive} Let $\cC$ be a dualizable category with a $t$-structure $(\cC_{\geq 0},\cC_{\leq 0})$ satisfying the following conditions: the $t$-structure is accessible, compatible with filtered colimits, $\cC$ is generated by $\cC^{\heartsuit}$ as a localizing subcategory, and for some $n\geq 0$ the functor $\hat{\cY}[-n]:\cC\to\Ind(\cC)$ is left $t$-exact. The (dualizable versions of) the above results deal with the case $n=0.$ Is there a reasonable generalization to the case $n>0?$
\end{ques}

The following question is motivated by the above analogy with connective $\bE_1$-rings and by the functor $\TR$ (topological restriction) with coefficients in a bimodule \cite{LM12, KMN23}. We introduce some non-standard notation for the ``dual'' of the familiar Goodwillie tower \cite{Goo03}. Let $\Phi:\cC\to\cD$ be a (not necessarily exact) functor between $\omega_1$-presentable stable categories, commuting with $\omega_1$-filtered colimits. For $n\geq 0$ denote by $P_n^{\vee}(\Phi):\cC\to\cD$ the universal polynomial functor of degree $\leq n,$ commuting with $\omega_1$-filtered colimits, with a map $P_n^{\vee}(\Phi)\to \Phi$ (its existence follows from the adjoint functor theorem). We say that the {\it coanalytic approximation} of $\Phi$ is the colimit $\indlim[n]P_n^{\vee}(\Phi)$ (computed objectwise). 

\begin{ques}
Let $\cC$ be a dualizable category (without a $t$-structure), and let $F:\cC\to\cC$ be a colimit-preserving endofunctor. We consider the tensor algebra $T(F)=\biggplus[n\geq 0]F^{\circ n}$ as a monad on $\cC,$ and take the category of modules $\Mod_{T(F)}(\cC),$ which is dualizable by \cite[Proposition C.1]{E24}. Consider the functor
\begin{multline}\label{eq:cocyclic_K_theory}
\Fun^L(\cC,\cC)\to \Sp,\quad F\mapsto \wt{K}^{\cont}(\Mod_{T(F)}(\cC))\\
=\Cone(K^{\cont}(\cC)\to K^{\cont}(\Mod_{T(F)}(\cC))).
\end{multline}
We denote by
\begin{equation*}
\TR^{\vee}(\cC,-):\Fun^L(\cC,\cC)\to \Sp
\end{equation*}
the coanalytic approximation of \eqref{eq:cocyclic_K_theory}.

Now suppose that $\cC$ is a dualizable $t$-category and $F:\cC\to\cC$ is (colimit-preserving and) left $t$-exact. Is it true that the map
\begin{equation*}
\TR^{\vee}(\cC,F)\to \wt{K}^{\cont}(\Mod_{T(F)}(\cC))
\end{equation*}
is an equivalence of spectra?
\end{ques}

We give a brief overview of the structure of the paper.

In Section \ref{sec:coh_ass_exact} we introduce and study locally $\kappa$-coherent exact $\infty$-categories, focusing on coherently assembled exact $\infty$-categories. We are mostly interested in the presentable stable envelope $\check{\St}(\cE)$ when $\cE$ is coherently assembled, and in the properties of the category $\check{\St}(\cE),$ such as dualizability. In particular, we define the continuous $K$-theory of such exact $\infty$-categories by putting $K^{\cont}(\cE)=K^{\cont}(\check{\St}(\cE)).$ Our approach in this section is somewhat minimalistic, in particular we mostly avoid the non-coherently assembled case. We reduce most of the non-trivial statements to the familiar results about small exact $\infty$-categories and their (small) stable envelopes. Among other things, we explain a natural version of the situation of d\'evissage for coherently assembled abelian categories in Subsection \ref{ssec:situation_of_devissage}. The results of this section are used in the proofs of our main theorems, especially in Section \ref{sec:theorem_of_the_heart_for_K}. 

In Section \ref{sec:comp_ass_t_structures} we introduce and study compactly assembled $t$-structures on dualizable categories, mostly focusing on continuously bounded $t$-structures. In particular, in Subsection \ref{ssec:construction_of_t_structures} we explain a general method of constructing a nice $t$-structure via a functor from a coherently assembled abelian category, based on a similar result for small categories. This method allows to construct many auxiliary categories with $t$-structures, which are used essentially in all the proofs of the main theorems. In particular, we consider the presentable stable envelopes $\check{\St}(\cC_{[0,m]})$ with natural $t$-structures with the same heart. Here $\cC_{[0,m]}=\cC_{\geq 0}\cap\cC_{\leq m}$ is considered as an exact $\infty$-category with the induced exact structure from $\cC.$ Again, our approach in this section is minimalistic, and some statements can be generalized to the case when $\cC$ is merely presentable stable, the $t$-structure is compatible with filtered colimits and right complete, and the Grothendieck prestable category $\cC_{\geq 0}$ is anticomplete in the sense of \cite[Definition C.5.5.4]{Lur18}. 

In Section \ref{sec:theorem_of_the_heart_for_K} we prove refined theorem of the heart in the context of dualizable categories (Theorem \ref{th:theorem_of_the_heart}). It states that for a dualizable $t$-category $\cC$ and for $m\geq 0$ the spectrum $\Cone(K^{\cont}(\cC_{[0,m]})\to K^{\cont}(\cC))$ is $(-m-3)$-coconnective. This in particular implies the above Corollary \ref{cor:refined_th_of_the_heart_intro}. One of the key ingredients is the surjectivity of the map $K_{-m-2}^{\cont}(\cC_{[0,m]})\to K_{-m-2}(\cC),$ which is proved by a very non-trivial argument. In particular, it requires some careful analysis of the (compactly generated) iterated Calkin categories and natural $t$-structures on them. Another important ingredient is the construction of certain short exact sequences of dualizable categories (Proposition \ref{prop:properties_of_tilde_C}), which relate the categories $\check{\St}(\cC_{[0,m]})$ for various $m.$

In Section \ref{sec:coconnectivity_estimates} we prove Theorem \ref{th:coconnectivity_estimates} on the coconnectivity estimates, which is a generalization of Theorem \ref{th:coconnectivity_estimates_intro}. One of the main ideas is a very abstract construction: given a strongly continuous exact functor $F:\cC\to\cD$ between dualizable categories such that $F(\cC)$ generates $\cD,$ there is a canonical sequence $\cC=\cC_0\xto{F_0}\cC_1\xto{F_1}\cC_2\dots$ in $\Cat_{\st}^{\dual},$ such that $\indlim[n]\cC_n\simeq\cD$ and each monad $F_n^R\circ F_n$ on $\cC_n$ is obtained as a so-called deformed tensor algebra. If $\cC$ and $\cD$ are dualizable $t$-categories and $F$ is $t$-exact, then all $\cC_n$ are also naturally dualizable $t$-categories and the transition functors are $t$-exact. A further analysis of deformed tensor algebras allows to eventually reduce Theorem \ref{th:coconnectivity_estimates} to Theorem \ref{th:theorem_of_the_heart} (refined theorem of the heart). We also deduce similar results for higher nil groups (Corollaries \ref{cor:coconnectivity_estimates_for_N^s K} and \ref{cor:coconnectivity_for_N^sK_of_t_categories}).

In Section \ref{sec:th_of_the_heart_for_KH} we prove theorem of the heart for $KH^{\cont}$ for dualizable $t$-categories (Theorem \ref{th:theorem_of_the_heart_for_KH}), which is a generalization of Theorem \ref{th:theorem_of_the_heart_for_KH_intro}. Here we briefly explain a more precise statement, for simplicity restricting to the case of small categories. For $k\geq 0$ denote by $U_k=\Fil_k KH$ the $k$-th term of the standard filtration on $KH,$ namely $U_k(\cC)$ is obtained by taking the colimit over $[n]\in\Delta_{\leq k}^{op}$ in the right hand side of \eqref{eq:definition_of_KH_intro}. Then the localizing invariants $U_k$ form a direct sequence with $U_0(\cC)=K(\cC)$ and $\indlim[k]U_k(\cC) \simeq KH(\cC).$ We prove by induction on $m\geq 0$ that if $F:\cC\to\cD$ satisfies the assumptions of Theorem \ref{th:coconnectivity_estimates_intro} for some $n\geq 1,$ then the map of spectra
\begin{equation*}
\tau_{\geq c}\Cone(U_k(F))\to \tau_{\geq c}\Cone(U_l(F))
\end{equation*}  
is null-homotopic for
\begin{equation*}
c\geq -2^m(n-1)+2k-1,\quad l-k\geq 2^m-1.
\end{equation*}
The case $n=2$ directly implies Theorem \ref{th:theorem_of_the_heart_for_KH_intro}.

In Section \ref{sec:devissage} we deduce d\'evissage theorems for $KH^{\cont}$ and $K^{\cont}$ of coherently assembled abelian categories, which are generalizations of Theorem \ref{th:theorem_of_the_heart_for_KH_intro} and Corollary \ref{cor:devissage_for_K_intro}.

In Section \ref{sec:sharpness_of_estimates} we prove that our estimates in Corollaries \ref{cor:refined_th_of_the_heart_intro} and \ref{cor:devissage_for_K_intro} are sharp (Theorem \ref{th:sharpness_of_estimates}), working over a field $\mk.$ The constructions are based on the exact category $\cE$ of vector bundles on a (projective) cuspidal cubic curve $X$ over $\mk,$ and on the abelian category $\eff(\cE)$ of effaceable functors $\cE^{op}\to\Ab.$ Essentially we only use that $X$ is proper and $NK_0(X)\ne 0.$ We also prove the sharpness of the estimates for higher nil groups (Theorem \ref{th:sharpness_for_N^sK}), also giving a non-trivial example of a computation of $DK(\cA)=\Fiber(K(\cA)\to KH(\cA))$ for an abelian category $\cA.$

In Section \ref{sec:examples_of_comp_ass_and_coh_ass} we study examples of compactly assembled $t$-structures and coherently assembled abelian categories. In particular, we consider sheaves on locally compact Hausdorff spaces (Proposition \ref{prop:t_structure_on_sheaves}). As a special case, we obtain that the abelian category $\cA$ of sheaves of vector spaces on the real line is coherently assembled, and we have $K_{-1}^{\cont}(\cA)\cong\Z.$ Therefore, already the Schlichting's vanishing theorem for $K_{-1}$ of small abelian categories \cite[Theorem 6]{Sch06} does not generalize to coherently assembled abelian categories. We consider the category of nuclear solid modules over $\Z_p$ (the original version due to Clausen and Scholze \cite{CS20}), and show that it has a natural compactly assembled continuously bounded $t$-structure (Proposition \ref{prop:t_structure_on_Nuc_of_Z_p}). We also prove a generalization of Chase criterion \cite{Ch60} of coherence of associative unital rings. This is Theorem \ref{th:Chase}. One way to state it is as follows: a dualizable $\Ab$-module $\cA$ in $\Pr^L$ is a coherently assembled abelian category if and only if the category of flat objects in $\cA^{\vee}$ has infinite products. Here an object $x\in\cA^{\vee}$ is flat if the functor $\ev(-,x):\cA\to\Ab$ is left exact. We recall that an $\Ab$-module in $\Pr^L$ is the same thing as a presentable additive (ordinary) category $\cA.$ By the forthcoming work of Levy, Liang and Sosnilo \cite{LLS}, $\cA$ is dualizable if and only if $\cA$ is a Grothendieck abelian category satisfying (AB6) and (AB4*). Equivalently by Roos' theorem \cite{Roo65} this means that $\cA$ is equivalent to a category of almost modules, i.e. $\cA\simeq\Mod\hy R/\Mod\hy(R/I),$ where $R$ is an associative unital ring and $I\subset R$ is a two-sided ideal such that $I^2=I.$ Finally, in Subsection \ref{ssec:concluding_remarks} we mention some example and questions which are not covered by this paper.

{\noindent{\bf Convention.}} {\it We will freely use the theory of $\infty$-categories as developed in \cite{Lur09, Lur17a}. We will simply say ``category'' instead of $\infty$-category. If the category is $0$-truncated (i.e. the morphism spaces are discrete), we will say ``ordinary category''. Since a stable category $\cC$ is naturally enriched over spectra, for $x,y\in\cC$ we denote by $\Hom_{\cC}(x,y)$ the spectrum of morphisms. We denote by $\Sp$ the category of spectra, and by $\Sp_{\geq 0}$ the category of connective spectra. As we already mentioned above, we denote by $\Cat^{\perf}$ the category of idempotent-complete small stable categories and exact functors. Throughout this paper we use the homological grading, unless otherwise stated.}

{\noindent{\bf Acknowledgements.}} I am grateful to Dmitry Kaledin, Ishan Levy, Maxime Ramzi and Vladimir Sosnilo for useful discussions. This
work was performed at the Steklov International Mathematical Center and supported by
the Ministry of Science and Higher Education of the Russian Federation (agreement no.
075-15-2025-303). I was partially supported by the HSE University Basic Research
Program.

\section{Coherently assembled abelian categories and exact categories}
\label{sec:coh_ass_exact}

%\begin{defi}
%Let $\cE=(\cE,\cE_{\dagger},\cE^{\dagger})$ be a small exact $\infty$-category, and let $\kappa$ be a regular cardinal. We define an exact structure on the category $\Ind_{\kappa}(\cE)$ as follows. A $1$-morphism $f$ in $\Ind_{\kappa}(\cE)$ is in $\Ind_{\kappa}(\cE_{\dagger})$ resp. $\Ind_{\kappa}(\cE^{\dagger})$ if $f$ is a $\kappa$-filtered colimit of morphisms in $\cE_{\dagger}$ resp. $\cE^{\dagger}.$
%\end{defi}

\subsection{Reminder on exact categories}
\label{ssec:reminder_on_exact}

We recall some material from \cite{Bar15, Kle, SW26} on small exact categories. 

\begin{defi}\cite[Section 3]{Bar15}\label{def:small_exact}
An exact category is an additive category $\cE$ with a choice of two subcategories $(\iin \cE,\pr \cE),$ consisting of (exact) inclusions resp. (exact) projections, such that the following conditions hold.
 \begin{enumerate}[label=(\roman*),ref=(\roman*)]
 	\item For $x\in\cE,$ the map $0\to x$ is an inclusion and the map $x\to 0$ is a projection.
 	\item The class of inclusions is stable under pushouts along all maps, and the class of projections is stable under pullbacks along all maps.
 	\item For a square 
 	\begin{equation*}\label{eq:exact_square}
 	\begin{tikzcd}
 	x\ar[r, "i"]\ar[d, "p"] & y\ar[d, "q"]\\
 	x'\ar[r, "j"] & y'
 	\end{tikzcd}
 	\end{equation*}
 in $\cE$ the following are equivalent:
 
 (a) $i$ is an inclusion, $p$ is a projection and the square is cocartesian;
 
 (b) $j$ is an inclusion, $q$ is a projection and the square is cartesian.
 
\noindent Such a square is called an exact square. If $x'=0,$ we will say that $x\to y\to z$ is a short exact sequence in $\cE.$
 \end{enumerate}
\end{defi} 

We will omit the subcategories $\iin \cE$ and $\pr \cE$ from the notation, assuming that they are chosen. An additive functor $F:\cE\to\cE'$ between exact categories is called exact if it preserves short exact sequences, or equivalently exact squares. We denote by $\Fun^{\ex}(\cE,\cE')$ the (additive) category of exact functors $\cE\to\cE'.$ We denote by $\Exact$ the category of small exact categories and exact functors.

We say that an exact functor between exact categories $F:\cE\to\cE'$ reflects exactness if for any square in $\cE,$ if its image is exact in $\cE',$ then it is exact in $\cE.$

If we denote by $\Add$ the category of small additive categories and additive functors, then the forgetful functor $\Exact\to\Add$ has a left adjoint $\cA\mapsto\cA^{\add}.$ Here $\cA^{\add}$ is the exact category $(\cA,\iin \cA,\pr \cA),$ where $\iin \cA$ consists of split monomorphisms and $\pr \cA$ consists of split epimorphisms. We call this a split exact structure on $\cA.$ Note that if $\cE$ is a small exact category, then the identity functor $\cE^{\add}\to \cE$ is exact, but it does not reflect exactness in general. 

Note that any stable category $\cC$ can be considered as an exact category with $\iin\cC=\pr\cC=\cC.$ For stable categories exactness of a functor in the above sense is equivalent to the exactness in the usual sense. Denoting by $\Cat^{\ex}$ the category of small stable categories and exact functors, we obtain a fully faithful functor $\Cat^{\ex}\hto \Exact.$ By \cite[Proposition 4.22, Corollary 4.24]{Kle} it has a left adjoint, which we denote by $\St:\Exact\to\Cat^{\ex}_{\st}.$ The category $\St(\cE)$ is called a stable envelope of the category $\cE.$ The following is due to Klemenc.

\begin{prop}\cite[Propositions 4.17, Proposition 4.25]{Kle}\label{prop:fully_faithful_E_to_St_of_E}
Let $\cE$ be a small exact category, and denote by $j:\cE\to\St(\cE)$ the universal exact functor. Then $j$ is fully faithful and reflects exactness, and its essential image is closed under extensions. 
\end{prop}

We refer to \cite{SW26} for further details. We will need the following more precise statement about the stable envelope of an exact category $\cE.$ First, consider the stable category $\St(\cE^{\add})\simeq\St^{\add}(\cE)$ with the universal additive functor $\cY:\cE\to\St(\cE^{\add}).$ We have
\begin{equation*}
\Ind(\St(\cE^{\add}))\simeq\Fun^{\add}(\cE^{op},\Sp),
\end{equation*}
where the superscript ``add'' means that we consider the additive functors. We identify $\St(\cE^{\add})$ with the full subcategory of this category of functors, and we have $\cY(x)=\Hom_{\cE}(-,x).$ Clearly, we have a quotient functor $\St(\cE^{\add})\to\St(\cE).$ Its kernel is denoted by $\Ac(\cE),$ and it is generated as a stable subcategory by the objects of the form $E(f)\in\Ac(\cE),$ where $x\xto{f}y\to z$ is a short exact sequence in $\cE$ and
\begin{equation}\label{eq:object_E_of_f}
E(f)=\Cone(\Cone(\cY(f))\to \cY(z)).
\end{equation}
As explained for example in \cite[Proposition 3.9]{SW26}, the associated functor $\cE^{op}\to\Sp$ takes values in $\Sp^{\heartsuit}\simeq\Ab,$ and we have
\begin{equation*}
E(f)\cong \coker(\pi_0\cY(y)\to\pi_0\cY(z)).
\end{equation*} 
Such functors $\cE^{op}\to\Ab$ are called effaceable, and they form an ordinary additive category, denoted by $\eff(\cE).$ This category is in fact abelian, more precisely the following holds.

\begin{prop}\cite[Lemma 1.2]{Nee21} \cite[Proposition G.7]{E24} \cite[Corollary 3.12]{SW26} \label{prop:t_structure_on_Ac}
The category $\Ac(\cE)$ has a bounded $t$-structure, which is induced by the standard (Postnikov) $t$-structure on $\Fun^{\add}(\cE^{op},\Sp).$ We have $\Ac(\cE)^{\heartsuit}\simeq \eff(\cE).$
\end{prop}

\begin{remark}
If $\cE$ is an ordinary exact category, then the above category $\Ac(\cE)$ is denoted by $\Ac^b(\cE)$ in \cite{Nee21}, since it is the category of bounded acyclic complexes. This should not lead to confusion, because we will not consider the unbounded version.
\end{remark}

 We also recall a more general class of functors. Recall that for an exact category $\cE,$ an additive functor $F:\cE^{op}\to\Ab$ is called {\it weakly effaceable} if for any $x\in\cE$ and for any $\alpha\in F(x)$ there exists a projection $f:y\to x$ in $\cE$ such that $F(f)(\alpha) = 0.$ We denote by $\Eff(\cE)$ the category of weakly effaceable functors $\cE^{op}\to\Ab.$ It is easy to see that we have $\eff(\cE)\subset\Eff(\cE).$ Moreover, by \cite[Proposition 3.9]{SW26} an additive functor $\cE^{op}\to\Ab$ is effaceable if and only if it is finitely presented and weakly effaceable. It is easy to see (assuming Proposition \ref{prop:t_structure_on_Ac}) that the category $\Eff(\cE)$ is abelian and locally coherent, and its full subcategory of finitely presented objects coincides with $\eff(\cE).$ Indeed, it suffices to observe that any map from a finitely presented additive functor $F:\cE^{op}\to\Ab$ to a weakly effaceable functor $G:\cE^{op}\to\Ab$ factors through an effaceable functor.
 
We will use the following notation. If $\cC$ is a stable category, then for $x,y\in\cC$ and $n\in\Z$ we put $\Ext^n_{\cC}(x,y)=\pi_{-n}\Hom_{\cC}(x,y).$ For a small exact category $\cE$ and for $x,y\in\cE$ we put $\Ext^n_{\cE}(x,y)=\Ext^n_{\St(\cE)}(j(x),j(y)),$ where $j:\cE\to\St(\cE)$ is the universal exact functor. By Proposition \ref{prop:fully_faithful_E_to_St_of_E} we have $\Ext^{-n}_{\cE}(x,y)=\pi_n\Map_{\cE}(x,y)$ for $n\geq 0.$ Also, the abelian group $\Ext^1_{\cE}(x,y)$ classifies extensions of $x$ by $y$ in $\cE.$

The following is a well-known observation, but we are not aware of a precise reference.

\begin{prop}\label{prop:mono_on_Ext^2_for_small}
Let $\cC$ be a stable category. Let $\cE\subset\cC$ be a small full additive subcategory closed under extensions. We consider $\cE$ as an exact category with the induced exact structure from $\cC.$ Consider the exact functor $\St(\cE)\to\cC$ induced by the inclusion $\cE\to\cC.$ Then for $x,y\in\cE$ the induced map $\Ext^n_{\cE}(x,y)\to\Ext^n_{\cC}(x,y)$ is an isomorphism for $n\leq 1,$ and a monomorphism for $n=2.$
\end{prop}

\begin{proof}
The isomorphisms for $n\leq 1$ follow from Proposition \ref{prop:fully_faithful_E_to_St_of_E} and from the description of $\Ext^1$ via extensions. Now take an element $\alpha\in\Ext^2_{\cE}(x,y).$ By \cite[Lemma 4.3]{SW26} there exists a projection $p:z\to x$ such that we have $p^*(\alpha)=0\in\Ext^2_{\cE}(z,y).$ Putting $w=\Fiber(\alpha)\in\cE,$ we see that $\alpha=\beta\circ\gamma,$ where $\beta\in\Ext^1_{\cE}(w,y),$ $\gamma\in\Ext^1_{\cE}(x,w).$ If the image of $\alpha$ in $\Ext^2_{\cC}(x,y)$ is zero, then the map $w\to y[1]$ in $\cC$ factors through $z,$ hence $\beta$ is contained in the image of the map $\Ext^1_{\cE}(z,y)\to\Ext^1_{\cE}(w,y).$ But this exactly means that $\alpha=0.$ This shows the injectivity of the map $\Ext^2_{\cE}(x,y)\to\Ext^2_{\cC}(x,y).$    
\end{proof}

We recall the following class of exact subcategories. The condition in the following definition already appears in \cite{Kel96}.

\begin{defi}\cite[Definition 1.1]{SW26}\label{def:left_special} Let $\cE$ be an exact category. A full additive subcategory $\cE'\subset\cE$ is called left special if $\cE'$ is closed under extensions in $\cE,$ and for any projection $p:y\to x$ in $\cE$ with $x\in\cE',$ there exists $z\in\cE'$ and a map $f:z\to y,$ such that the composition $p\circ f:z\to x$ is a projection in $\cE.$
\end{defi}

\begin{theo}\cite[Theorem 4.5]{SW26}\label{th:left_special_implies_fully_faithful}
Let $\cE$ be a small exact category, and let $\cE'\subset\cE$ be a left special full subcategory. We consider $\cE'$ as an exact category with an induced exact structure. Then the functor $\St(\cE')\to\St(\cE)$ is fully faithful.
\end{theo}

For completeness we mention the following statement, which might be not in the literature. In the case of ordinary exact categories this is \cite[Theorem 2.8]{BS01}, with a different proof.

\begin{prop}
Let $\cE$ be an idempotent-complete exact category. Then the category $\St(\cE)$ is also idempotent-complete.
\end{prop}

\begin{proof}
Denote by $\St(\cE)^{\Kar}$ the idempotent completion. The category $\St(\cE^{\add})$ is idempotent-complete by \cite[Theorem 5.3.1, Proposition 6.2.1]{Bon}, so we have a short exact sequence in $\Cat^{\perf}:$
\begin{equation*}
0\to\Ac(\cE)\to \St(\cE^{\add})\to\St(\cE)^{\Kar}\to 0.
\end{equation*}
By Proposition \ref{prop:t_structure_on_Ac} the category $\Ac(\cE)$ has a bounded $t$-structure, hence by \cite[Theorem 2.35]{AGH19} we have $K_{-1}(\Ac(\cE))=0.$ The long exact sequence gives the surjectivity of the map $K_0(\St(\cE^{\add}))\to K_0(\St(\cE)^{\Kar}).$ By \cite[Theorem 2.1]{Th} this implies that $\St(\cE)$ is idempotent-complete.
\end{proof}

\subsection{Accessibly exact categories}

To deal with large exact categories we will need the following basic construction. We will use the following notation: if $\cA$ is a small category and $\kappa$ is a regular cardinal, then we write the objects of $\Ind_{\kappa}(\cA)$ as $\inddlim[i\in I]x_i,$ where $I$ is a $\kappa$-directed poset and we are assuming a functor $I\to\cA,$ $i\mapsto x_i.$ One can more generally take $I$ to be a $\kappa$-filtered $\infty$-category, but recall that by \cite[Proposition 5.3.1.18]{Lur09} for such $I$ there exists a $\kappa$-directed poset $J$ and a cofinal functor $J\to I.$ 

\begin{prop}\label{prop:exact_structure_on_Ind_kappa} Let $\kappa\leq \lambda$ be a regular cardinals. Let $\cE$ be a small exact category. Then the category $\Ind_{\kappa}(\cE)^{\lambda}$ has the following exact structure. A morphism $f$ in  $\Ind_{\kappa}(\cE)^{\lambda}$ is an inclusion resp. a projection if it is a $\lambda$-small $\kappa$-directed colimit of inclusions resp. projections in $\cE.$ 

Taking the union over $\lambda,$ we obtain an exact structure on $\Ind_{\kappa}(\cE).$ The (fully faithful) functor $\Ind_{\kappa}(\cE)\to \Ind_{\kappa}(\St(\cE))$ preserves and reflects exactness. Its  essential image is closed under extensions.
\end{prop}
%\begin{enumerate}[label=(\roman*),ref=(\roman*)]
	%\item Let $\cC$ be a small stable category with a bounded heart structure $(\cC_{\leq 0},\cC_{\geq 0}).$ Then the pair $(\Ind(\cC_{\leq 0})^{\kappa},\Ind(\cC_{\geq 0})^{\kappa})$ is a heart structure on $\Ind(\cC)^{\kappa}.$ \label{heart_structure_on_Ind}
	%\item 
%\end{enumerate}

\begin{proof} 
Consider the stable envelope $\St(\cE),$ and denote by $j':\cE\to\St(\cE)$ the universal exact functor. By Proposition \ref{prop:fully_faithful_E_to_St_of_E} $j'$ is fully faithful, reflects exactness, and its essential image is closed under extensions in $\St(\cE).$ Consider the stable category $\cC=\Ind_{\kappa}(\St(\cE))^{\lambda}.$ Then we have a fully faithful functor $j:\Ind_{\kappa}(\cE)^{\lambda}\to\cC.$ 

We claim that the essential image of $j$ is closed under extensions. Indeed, consider a cofiber sequence of the form $j(x)\to z\to j(y)$ in $\cC.$ By \cite[Proposition 5.3.5.15]{Lur09} and its proof the morphism $j(y)\to j(x)[1]$ is a $\lambda$-small $\kappa$-directed colimit of morphisms $j'(y_i)\to j'(x_i)[1],$ where $x_i,y_i\in \cE,$ $i\in I.$ Hence, $z\cong\inddlim[i\in I]\Fiber(j'(y_i)\to j'(x_i)[1])\in\cC$ is contained in the essential image of $j,$ since each object $\Fiber(j'(y_i)\to j'(x_i)[1])$ is contained in the essential image of $j'.$ This shows that the essential image of $j$ is closed under extensions.

Consider the induced exact structure on $\Ind_{\kappa}(\cE)^{\lambda}.$ The same argument as above shows that short exact sequences in this category are exactly the $\lambda$-small $\kappa$-directed colimits of short exact sequences in $\cE.$

Finally, the assertions about the functor $\Ind_{\kappa}(\cE)\to \Ind_{\kappa}(\St(\cE))$ follow from the above.
\end{proof}

\begin{defi}
We define a locally $\kappa$-coherent exact category to be a category of the form $\Ind_{\kappa}(\cE),$ where $\cE$ is a small exact category and $\kappa$ is a regular cardinal. The exact structure is described in Proposition \ref{prop:exact_structure_on_Ind_kappa}. For $\kappa=\omega$ we say ``locally coherent'' instead of ``locally $\omega$-coherent''.

An accessibly exact category is a locally $\kappa$-coherent exact category for some $\kappa.$
\end{defi}

%From now on for brevity we will simply say ``accessibly exact category'', meaning an $\infty$-category, and similarly for locally $\kappa$-coherent exact categories. 
We will consider the accessible stable categories as accessibly exact categories in the natural way. 

\begin{prop}\label{prop:properties_of_Ind_kappa_of_small_exact}
Let $\cE$ be a small exact category and let $\kappa$ be a regular cardinal.
\begin{enumerate}[label=(\roman*),ref=(\roman*)]
	\item The inclusion functor $\Ind_{\kappa}(\cE)\to \Ind(\cE)$ preserves and reflects exactness. Its essential image is closed under extensions. \label{Ind_kappa_inside_Ind}
	\item Let $\cE'\subset\cE$ be an additive subcategory closed under extensions, equipped with the induced exact structure from $\cE.$ Then the functor $\Ind_{\kappa}(\cE')\to\Ind_{\kappa}(\cE)$ preserves and reflects exactness. Its essential image is closed under extensions. \label{Ind_kappa_of_subcategory}
\end{enumerate}
\end{prop}

\begin{proof}
\ref{Ind_kappa_inside_Ind} If $\cE$ is stable with the standard exact structure, then the statement is evident. The general case reduces to the stable case using Proposition \ref{prop:exact_structure_on_Ind_kappa}. Namely, we have a commutative square
\begin{equation*}
\begin{tikzcd}
\Ind_{\kappa}(\cE) \ar[r]\ar[d] & \Ind(\cE)\ar[d]\\
\Ind_{\kappa}(\St(\cE))\ar[r] & \Ind(\St(\cE)),
\end{tikzcd}
\end{equation*}
and both vertical functors are fully faithful, preserve and reflect exactness, and their essential images are closed under extensions. The same holds for the lower horizontal functor by the above special case. Hence, the same holds for the upper horizontal functor.

\ref{Ind_kappa_of_subcategory} The proof is similar, but a little caution is required: the functor $\St(\cE')\to \St(\cE)$ does not have to be fully faithful. Instead, we first observe that we can replace $\cE$ by $\St(\cE),$ by Proposition \ref{prop:exact_structure_on_Ind_kappa}. Assuming that $\cE$ is stable, the same argument as in loc. cit. shows that the essential image of $\Ind_{\kappa}(\cE')$ in $\Ind_{\kappa}(\cE)$ is closed under extensions. It also shows that this functor reflects exactness.
\end{proof}

The following definition is essentially tautological, but we spell it out for completeness.

\begin{defi}
Let $\cE$ be an accessibly exact category and let $\kappa$ be a regular cardinal such that $\cE$ has $\kappa$-filtered colimits. We say that the $\kappa$-filtered colimits are exact in $\cE$ if the class of short exact sequences in $\cE$ is closed under $\kappa$-filtered colimits. 
\end{defi}

This exactness holds in the following basic situation.

\begin{prop}\label{prop:locally_kappa_coherent_satisfies_exactness_of_kappa_filtered}
Let $\cE$ be a locally $\kappa$-coherent exact category. Then the $\kappa$-filtered colimits are exact in $\cE.$
\end{prop}

\begin{proof}
Again, the stable case is evident. The general case reduces to the stable case by Proposition \ref{prop:exact_structure_on_Ind_kappa}: the functor $\cE\to\Ind_{\kappa}(\St(\cE^{\kappa}))$ commutes with $\kappa$-filtered colimits and reflects exactness.
\end{proof}

The following observation is slightly subtle.

\begin{prop}\label{prop:kappa_compact_closed_under_extensions}
Let $\cE$ be an accessibly exact category and let $\kappa$ be a regular cardinal. Suppose that $\cE$ has exact $\kappa$-filtered colimits. Then the full subcategory $\cE^{\kappa}\subset\cE$ is closed under extensions.
\end{prop}

\begin{proof}
Let $x\to y\to z$ be a short exact sequence in $\cE,$ and suppose that $x,z\in\cE^{\kappa}.$ Let $(w_i)_{i\in I}$ be a $\kappa$-directed system in $\cE.$ We need to prove that the map $\indlim[i]\Map_{\cE}(y,w_i)\to \Map_{\cE}(y,\indlim[i]w_i)$ is an equivalence of spaces. By assumption this holds if we replace $y$ with $x$ or $z.$ Using the long exact sequences and the five-lemma, we reduce the question to showing that the map of abelian groups 
\begin{equation}\label{eq:Ext^1_and_colimits_mono}
\indlim[i]\Ext^1_{\cE}(z,w_i)\to \Ext^1_{\cE}(z,\indlim[i] w_i)
\end{equation} 
is a monomorphism.

Take some $i_0\in I$ and an element $\alpha\in\Ext^1_{\cE}(z,w_{i_0}).$ It corresponds to a short exact sequence $w_{i_0}\to u\to z.$ For $i\geq i_0$ consider the pushout $u_i=u\sqcup_{w_{i_0}}w_i.$ Since $\kappa$-filtered colimits are exact by assumption, the image of $\alpha$ in $\Ext^1_{\cE}(z,\indlim[i]w_i)$ corresponds to the short exact sequence $\indlim[i\geq i_0]w_i\to \indlim[i\geq i_0]u_i\to z.$ Suppose that this extension splits, i.e. we have a section $s:z\to\indlim[i\geq i_0] u_i.$ Then $s$ factors through $u_{i_1}$ for some $i_1\geq i_0.$ This means that the image of $\alpha$ in $\Ext^1_{\cE}(z,w_{i_1})$ vanishes. This proves the injectivity of \eqref{eq:Ext^1_and_colimits_mono} and the proposition. 
\end{proof}

In the situation of Proposition \ref{prop:kappa_compact_closed_under_extensions} we consider $\cE^{\kappa}$ as an exact category with the induced exact structure from $\cE.$ We make further observations about the subcategories $\cE^{\lambda}$ for various $\lambda.$

\begin{prop}\label{prop:various_E^lambda}
Let $\cE$ be a locally $\kappa$-coherent exact category and let $\lambda\geq\kappa$ be a regular cardinal. 
\begin{enumerate}[label=(\roman*),ref=(\roman*)]
	\item For any $x\in\cE^{\kappa}$ the functor $\Ext^1_{\cE}(x,-):\cE\to\Ab$ commutes with $\kappa$-filtered colimits. \label{Ext^1_kappa_continuous}
	\item The inclusion functor $\cE^{\kappa}\to\cE^{\lambda}$ is left special. In particular, the functor $\St(\cE^{\kappa})\to\St(\cE^{\lambda})$ is fully faithful. \label{left_speciaal_E_kappa_in_E_lambda}
	%\item The functor $\Ind(\cE^{kappa})\to \Ind(\cE^{\lambda})$ preserves and reflects exactness. \label{Ind_E_kappa_Ind_E_lambda}
\end{enumerate} 
\end{prop}

\begin{proof}
We prove \ref{Ext^1_kappa_continuous}. If $\cE$ is $\kappa$-accessible stable with the standard exact structure, then the statement is evident. The general case reduces to the stable case by the proof of Proposition \ref{prop:exact_structure_on_Ind_kappa}: we have a fully faithful $\kappa$-accessible functor $i:\cE\to\Ind_{\kappa}(\St(\cE^{\kappa}))=\cC,$ and $\Ext^1_{\cE}(x,-)\cong \Ext^1_{\cC}(i(x),i(y)).$ This proves \ref{Ext^1_kappa_continuous}.

The second assertion of \ref{left_speciaal_E_kappa_in_E_lambda} follows from the first one by Theorem \ref{th:left_special_implies_fully_faithful}. We prove the left speciality. Consider a short exact sequence $x\to y\to z$ in $\cE^{\lambda}$ with $z\in\cE^{\kappa}.$ Denote by $\alpha\in\Ext^1_{\cE}(z,x)$ the corresponding element. Let $x=\indlim[i\in I]x_i,$ where $I$ is $\lambda$-small and $\kappa$-directed, and $x_i\in \cE^{\kappa}$ for $i\in I.$ By \ref{Ext^1_kappa_continuous} there exists some $i_0\in I$ and an element $\beta\in\Ext^1_{\cE}(z,x_{i_0})$ which maps to $\alpha.$ We obtain the associated short exact sequence $x_{i_0}\to w\to z.$ By Proposition \ref{prop:kappa_compact_closed_under_extensions} we have $w\in\cE^{\kappa}.$ By construction, the projection $w\to z$ factors through $y,$ which proves that $\cE^{\kappa}$ is left special in $\cE^{\lambda}.$
\end{proof}

In the situation of Proposition \ref{prop:various_E^lambda} we define the category $\St(\cE)$ to be the directed union of $\St(\cE^{\lambda})$ over all regular $\lambda\geq\kappa.$

%We will deal essentially only with those accessibly exact categories which are either small or have filtered colimits.

\subsection{Grothendieck exact categories}

\begin{defi}\label{def:Grothendieck_exact}
Let $\cE$ be an accessibly exact category. We say that $\cE$ is a Grothendieck exact category if
$\cE$ has exact filtered colimits. % the colimit functor $\colim:\Ind(\cE^{\kappa})\to\cE$ is exact.
\end{defi}

It is convenient to use the following convention: if $\cE$ is a locally $\kappa$-coherent exact category, then $\Ind(\cE)$ is a union of exact categories $\Ind(\cE^{\lambda})$ for all regular $\lambda\geq \kappa.$

\begin{prop}\label{prop:Grothendieck_exact_reformulation}
Let $\cE$ be a locally $\kappa$-coherent exact category with filtered colimits. Suppose that the colimit functor $\colim:\Ind(\cE^{\kappa})\to\cE$ is exact. Then $\cE$ is a Grothendieck exact category. 
\end{prop}

\begin{proof}
We need to show that the functor $\colim:\Ind(\cE)\to \cE$ is exact. By Proposition \ref{prop:locally_kappa_coherent_satisfies_exactness_of_kappa_filtered} for any small exact category $\cD$ the colimit functor $\Ind(\Ind(\cD))\to\Ind(\cD)$ is exact. Hence, so is the composition
\begin{equation*}
\Ind(\cE)\simeq \Ind(\Ind_{\kappa}(\cE^{\kappa}))\to \Ind(\Ind(\cE^{\kappa}))\xto{\colim}\Ind(\cE^{\kappa})\xto{\colim}\cE.
\end{equation*}
This proves the proposition.
\end{proof}

\begin{example}
A Grothendieck abelian category can be considered as a Grothendieck exact category.
\end{example}

\begin{example}\label{ex:locally_coherent_exact}
Let $\cE$ be a small exact category. Then $\Ind(\cE)$ is a Grothendieck exact category by Proposition \ref{prop:locally_kappa_coherent_satisfies_exactness_of_kappa_filtered}.
\end{example}

\begin{example}
Any presentable stable category can be considered as a Grothendieck exact category. More generally, any Grothendieck prestable category (\cite[Definition C.1.4.2]{Lur18})  is a Grothendieck exact category.
\end{example}

\subsection{Criterion for an exact category to be locally $\kappa$-coherent}

We will use the following general characterization of locally $\kappa$-coherent exact categories.

\begin{prop}\label{prop:criterion_locally_kappa_coherent}
	Let $\cE$ be an accessibly exact category and let $\kappa$ be a regular cardinal. Then $\cE$ is locally $\kappa$-coherent if and only if the following conditions hold.
	\begin{enumerate}[label=(\roman*),ref=(\roman*)]
		\item $\cE$ is $\kappa$-accessible (as an abstract category). \label{kappa_accessible}
		\item The $\kappa$-filtered colimits are exact in $\cE.$ \label{exactness_of_kappa_filtered}
		\item For $x\in\cE^{\kappa}$ the functor $\Ext^1_{\cE}(x,-):\cE\to\Ab$ commutes with $\kappa$-filtered colimits. \label{Ext^1_kappa_continuous_again}
	\end{enumerate}
	
	Moreover, in this case the full subcategory $\cE^{\kappa}\subset \cE$ is closed under fibers of projections in $\cE.$
\end{prop}

\begin{proof}
	We first prove the ``only if'' direction. \ref{kappa_accessible} holds by definition, \ref{exactness_of_kappa_filtered} holds by Proposition \ref{prop:locally_kappa_coherent_satisfies_exactness_of_kappa_filtered} and \ref{Ext^1_kappa_continuous_again} holds by Proposition \ref{prop:various_E^lambda}. We prove the ``moreover'' assertion. Let $x\to y\to z$ be a short exact sequence in $\cE$ with $y,z\in\cE^{\kappa}.$ Consider the functor $i:\cE\to\Ind_{\kappa}(\St(\cE^{\kappa}))$ from Proposition \ref{prop:exact_structure_on_Ind_kappa}. Then the objects $i(y),i(z)$ are $\kappa$-compact, hence so is $i(x).$ Since $i$ is fully faithful and commutes with $\kappa$-filtered colimits, it follows that $x$ is $\kappa$-compact in $\cE.$
	
	Now we prove the ``if'' direction. By \ref{exactness_of_kappa_filtered} and by Proposition \ref{prop:kappa_compact_closed_under_extensions} the full subcategory $\cE^{\kappa}\subset\cE$ is closed under extensions. We consider $\cE^{\kappa}$ as a small exact category with the induced exact structure from $\cE.$ We equip $\Ind_{\kappa}(\cE^{\kappa})$ with the exact structure from Proposition \ref{prop:exact_structure_on_Ind_kappa}. By \ref{exactness_of_kappa_filtered} the natural functor $\Phi:\Ind_{\kappa}(\cE^{\kappa})\to\cE$ is exact. By \ref{kappa_accessible} $\Phi$ is an equivalence (of abstract categories). It remains to prove that $\Phi$ reflects exactness. In other words, we need to prove that the class of short exact sequences in $\cE$ is the smallest class which contains short exact sequences in $\cE^{\kappa}$ and is closed under $\kappa$-filtered colimits.
	
	Let $x\to y\to z$ be a short exact sequence in $\cE,$ and denote by $\alpha\in\Ext^1_{\cE}(z,x)$ the corresponding element. Let $z\cong\indlim[i\in I]z_i,$ where $I$ is $\kappa$-directed and $z_i\in\cE^{\kappa}.$ Putting $y_i = y\times_z z_i,$ we see that $x\to y\to z$ is a $\kappa$-directed colimit of short exact sequences $x\to y_i\to z_i.$ Hence, we may and will assume that $z\in \cE^{\kappa}.$ Let $x\cong \indlim[j\in J]x_j,$ where $J$ is $\kappa$-directed and $x_j\in\cE^{\kappa}.$ By \ref{Ext^1_kappa_continuous_again} there exists $j_0\in J$ and an element $\beta\in\Ext^1_{\cE}(z,x_{j_0})$ which maps to $\alpha.$ Consider the corresponding short exact sequence $x_{j_0}\to y_{j_0}\to z.$ For $j\geq j_0$ put $y_j = y_{j_0}\sqcup_{x_{j_0}} x.$ By Proposition \ref{prop:kappa_compact_closed_under_extensions} we have $y_j\in\cE^{\kappa}.$ We conclude that $x\to y\to z$ is the $\kappa$-directed colimit of $x_j\to y_j\to z$ over $j\geq j_0,$ which proves that $\Phi$ reflects exactness.   
\end{proof}

\subsection{Presentable stable envelopes}

Let $\cE$ be a Grothendieck exact category, and let $\cC$ is presentable stable category. Denote by $\Fun^{\cont,\ex}(\cE,\cC)$ the category of continuous (i.e. filtered colimit-preserving) exact functors $\cE\to\cC.$  

\begin{defi}
For a Grothendieck exact category $\cE$ we denote by $\check{\St}(\cE)$ a presentable stable category with a continuous exact functor $j:\cE\to\check{\St}(\cE),$ such that for any $\cC\in\Pr^L_{\st}$ the precomposition with $j$ induces an equivalence
\begin{equation*}
\Fun^L(\check{\St}(\cE),\cC)\xto{\sim}\Fun^{\cont,\ex}(\cE,\cC).
\end{equation*}
We call $\check{\St}(\cE)$ the presentable stable envelope of $\cE.$
\end{defi}

We use the notation $\check{\St}(\cE)$ because this is a generalization of the unseparated derived category $\check{D}(\cA)$ of a Grothendieck abelian category $\cA,$ see Proposition \ref{prop:St_of_Grothendieck_abelian} below.

\begin{prop}\label{prop:check_St_of_Grothendieck_exact_category}
For a Grothendieck exact category $\cE,$ its presentable stable envelope exists and it is given by 
\begin{equation}\label{eq:presentable_stable_envelope}
\check{\St}(\cE)=\{F:\cE^{op}\to \Sp\mid F\text{ is exact and commutes with cofiltered limits}\}.
\end{equation}
\end{prop}

\begin{proof}
Let us temporarily denote by $\cC$ the right hand side of \eqref{eq:presentable_stable_envelope}. Choose a regular cardinal $\kappa$ such that $\cE$ is locally $\kappa$-coherent. Then restriction to $\cE^{\kappa}$ defines a fully faithful functor
\begin{equation*}
F:\cC\hto \Fun^{\ex}((\cE^{\kappa})^{op},\Sp)\simeq \Ind(\St(\cE^{\kappa})).
\end{equation*}
Its essential image consists of functors which commute with $\kappa$-small cofiltered colimits. It follows that $\cC$ is presentable stable, and $F$ is accessible and commutes with infinite products. Therefore, $F$ has a left adjoint $F^L.$  Consider the composition
\begin{equation*}
j':\cE^{\kappa}\to\St(\cE^{\kappa})\to \Ind(\St(\cE^{\kappa})).
\end{equation*}
Then the kernal of $F^L$ is generated (as a localizing subcategory) by objects of the form $\Cone(\indlim[i\in I]j'(x_i)\to j'(\indlim[i] x_i)),$ where $I$ is a $\kappa$-small directed poset and $I\to\cE^{\kappa}$, $i\mapsto x_i,$ is a functor.

Now define $j'':\cE\to\Ind(\St(\cE^{\kappa}))$ to be the left Kan extension of $j'.$ Then $j''$ is exact. We further define
\begin{equation*}
j=F^L\circ j'':\cE\to \cC.
\end{equation*}
We claim that $j$ satisfies the required universal property. First, note that $j$ commutes with filtered colimits. To check this, we only need to see that $j_{\mid \cE^{\kappa}}$ commutes with $\kappa$-small directed colimits, and this follows from the above description of the kernel of $F^L.$

Now for any $\cD\in\Pr^L_{\st}$ the functor $j''$ induces an equivalence
\begin{equation*}
\Fun^L(\Ind(\St(\cE^{\kappa})),\cD)\cong \Fun^{\kappa\hy\cont,ex}(\cE,\cD),
\end{equation*}
where the superscript ``$\kappa\hy\cont$'' means the commutation with $\kappa$-filtered colimits. It follows formally that $j$ induces an equivalence
\begin{equation*}
\Fun^L(\cC,\cD)\xto{\sim}\Fun^{\cont,\ex}(\cE,\cD). \qedhere
\end{equation*}
\end{proof}

The case of locally coherent exact categories is studied in \cite{NW25}.

\begin{prop}\label{prop:St_of_locally_coherent}
	Let $\cE$ be a locally coherent exact category. Then we have an equivalence $\check{\St}(\cE)\simeq \Ind(\St(\cE^{\omega})).$ 
\end{prop}

\begin{proof}
	The description of $\check{\St}(\cE)$ follows from the universal property, see also \cite[Corollary 2.24]{NW25}.
\end{proof}

The following basic example is familiar from \cite{Lur18}.

\begin{prop}\label{prop:St_of_Grothendieck_prestable}
Let $\cC$ be a Grothendieck prestable category, which we consider as a Grothendieck exact category. Then we have an equivalence $\check{\St}(\cC)\simeq\Sp(\cC),$ compatible with the universal functors from $\cC.$
\end{prop}

\begin{proof}
This follows from the universal properties. Namely, for a presentable stable category $\cD$ a functor $F:\cC\to\cD$ is continuous and exact if and only if $F$ commutes with colimits.
\end{proof}

Recall that for a Grothendeick abelian category $\cA$ the so-called unseparated derived category $\check{D}(\cA)$ can be described as (the dg nerve of) the dg category $\Ch(\Inj_{\cA})$ of complexes of injective objects of $\cA.$ It has a natural right complete $t$-structure, compatible with filtered colimits, and we have $\check{D}(\cA)^+\simeq D^+(\cA).$ Moreover, $\check{D}(\cA)$ is generated by the heart as a localizing subcategory. We refer to \cite{Kr05, Kr15} and \cite[Section C.5.8]{Lur18} for a detailed account. 

\begin{prop}\label{prop:St_of_Grothendieck_abelian}
For a Grothendieck abelian category $\cA$ we have
\begin{equation*}
\check{\St}(\cA)\simeq\check{D}(\cA).
\end{equation*}
\end{prop}

\begin{proof}
This is explained in \cite{Lur17b}.
%We first consider the case $\cA=\Ind(\cA_0),$ where $\cA_0$ is a small has enough projective objects
\end{proof}

\begin{remark}
Essentially the same argument as in loc. cit. proves a similar result for the so-called Grothendieck abelian $(n+1)$-categories, i.e. categories of the form $\cA=\tau_{\leq n}\cC,$ where $\cC$ is a Grothendieck prestable category and $n\geq 0$ \cite[Definition C.5.4.1]{Lur18}. Namely, if we equip $\cA$ with the induced exact structure from $\cC,$ then $\check{\St}(\cA)$ is equivalent to $\Sp(\Phi_n(\cA)),$ where $\Phi_n:\Groth_n^{\lex}\to\Groth_{\infty}^{\lex}$ is left adjoint to $\tau_{\leq n}(-).$ Here $\Groth_{\infty}^{\lex}$ resp. $\Groth_n^{\lex}$ is the category of Grothendieck prestable categories resp. Grothendieck abelian $(n+1)$-categories, where the $1$-morphisms are colimit-preserving left exact functors.

We will crucially use such categories $\check{\St}(\tau_{\leq n}\cC)$ below, but only in the relatively elementary case when $\cC$ is the connective part of a compactly assembled $t$-structure on a dualizable category (Definition \ref{def:compactly_assembled_t_structure}). We will not need the above description, instead we will apply the general machinery of coherently assembled exact categories from Subsection \ref{ssec:coh_ass_exact_cats}.
%We will crucially use the following class of Grothendieck exact categories, the so-called Grothendieck abelian $n$-categories, where $n\geq 0.$ These are defined in [...] as categories of the form $\tau_{\leq n}\cC,$ where $\cC$ is a Grothendieck prestable category. We consider $\tau_{\leq n}\cC$ as an exact category with the induced exact structure from $\cC.$ More precisely, the short exact sequences in $\tau_{\leq n}\cC$ are {\it all} fiber-cofiber sequences. It is clear that $\tau_{\leq n}\cC$ is a Grothendieck exact category. Denote by $\Groth_{\infty}^{\lex}$ the category of prestable Grothendieck categories, and for $n\geq 0$ denote by $\Groth_n^{\lex}$ the category of abelian $n$-categories. In both cases the $1$-morphisms are functors commuting with all colimits and with finite limits. Denote by $\Phi_n:\Groth_n^{\lex}\to\Groth_{\infty}^{\lex}$ the left adjoint to the functor $\tau_{\leq n}(-).$ Then for $\cA\in\Groth_n^{\lex}$ we have a natural equivalence
%\begin{equation*}
%	\check{\St}(\cA)\simeq \Sp(\Phi_n(\cA)).
%\end{equation*}
\end{remark}

\subsection{Reminder on compactly assembled and dualizable categories}
\label{ssec:reminder_on_compass}

We recall some basic notions from \cite{Lur18} and \cite{E24}, and also fix the notation and terminology. Recall that we simply say ``category'' instead of ``$\infty$-category'', and say ``ordinary category'' when the spaces of morphisms are discrete.

If $\cC$ is an accessible category with filtered colimits, then $\cC$ is called {\it compactly assembled} if the colimit functor $\colim:\Ind(\cC)\to\cC$ has a left adjoint $\hat{\cY}=\hat{\cY}_{\cC}:\cC\to\Ind(\cC).$ This is \cite[Definition 21.1.2.1]{Lur18}. By \cite[Proposition 1.24]{E24} any compactly assembled category $\cC$ is $\omega_1$-accessible, hence the functor $\hat{\cY}_{\cC}$ takes values in $\Ind(\cC^{\omega_1})\subset\Ind(\cC)$ (the latter follows for example by the same argument as in \cite[Corollary 1.22]{E24}).

If $\cC$ is a compactly assembled category, then a morphism $f:x\to y$ in $\cC$ is called {\it compact} if the map $\cY(x)\to \cY(y)$ in $\Ind(\cC)$ factors through $\hat{\cY}(y),$ where $\cY:\cC\to\Ind(\cC)$ is the Yoneda embedding. Recall that any compact morphism in $\cC$ is (homotopic to) a composition of two compact morphisms, by the same argument as in \cite[Corollary 1.41]{E24}.

A functor $F:\cC\to\cD$ between compactly assembled categories is called {\it continuous} if it commutes with filtered colimits. Further, $F$ is called {\it strongly continuous} if it is continuous and the following (lax commutative) square commutes:
\begin{equation*}
\begin{tikzcd}
\cC \ar{r}{F}\ar{d}{\hat{\cY}_{\cC}} & \cD\ar{d}{\hat{\cY}_{\cD}}\\
\Ind(\cC)\ar{r}{\Ind(F)} &\Ind(\cD).
\end{tikzcd}
\end{equation*}
Note that if $F$ has a right adjoint $F^R,$ then the strong continuity of $F$ exactly means that $F^R$ is continuous (the same proof as in \cite[Proposition 1.14]{E24}).
In general, the strong continuity of a continuous functor $F:\cC\to\cD$ exactly means that $F$ preserves compact morphisms (it suffices to check this for morphisms between the objects of $\cC^{\omega_1}$).

We recall that $\Pr^L$ denotes the category of presentable categories and colimit-preserving functors. It is symmetric monoidal with the Lurie tensor product, and the unit object is given by the category of spaces $\cS.$ We denote by $\Pr^L_{\st}$ the full subcategory of presentable stable categories, which is also the category of modules over the idempotent $\bE_{\infty}$-algebra $\Sp.$ In particular, $\Pr^L_{\st}$ is closed under tensor products in $\Pr^L,$ and the unit object is given by $\Sp.$ A presentable stable category $\cC$ is called {\it dualizable} if it is a dualizable object in $\Pr^L_{\st}.$ We will simply say ``dualizable category'', assuming that it is presentable and stable. 

By \cite[Proposition D.7.3.1]{Lur18}, $\cC\in\Pr^L_{\st}$ is dualizable if and only if it is compactly assembled. In particular, the above notions apply. In this paper we will deal with some functors between dualizable categories, which are strongly continuous but not exact, such as a truncation endofunctor $\tau_{\geq 0}$ for a nice $t$-structure. For this reason we will specify if the functor is supposed to be exact. We denote by $\Cat_{\st}^{\dual}\subset\Pr^L_{\st}$ the non-full subcategory of dualizable categories and strongly continuous exact functors. For $\cC,\cD\in\Cat_{\st}^{\dual}$ we denote by $\Fun^{LL}(\cC,\cD)$ the (idempotent-complete small stable) category of strongly continuous exact functors.

We recall that again by \cite[Proposition D.7.3.1]{Lur18} a presentable stable category $\cC$ is dualizable if and only if it is a retract in $\Pr^L_{\st}$ of a compactly generated category. We have a fully faithful embedding $\Ind(-):\Cat^{\perf}\to\Cat_{\st}^{\dual}.$

We will use the convention from \cite{E25c} for Calkin categories. Namely, for $\cC\in\Cat_{\st}^{\dual}$ we put
\begin{equation*}
\Calk_{\omega_1}(\cC)=\Ind(\cC^{\omega_1})/\hat{\cY}(\cC)\simeq\ker(\colim:\Ind(\cC^{\omega_1})\to\cC).
\end{equation*}
This is a compactly generated (presentable stable) category. We have
\begin{equation*}
\Calk_{\omega_1}(\cC)\simeq\Ind(\Calk_{\omega_1}^{\cont}(\cC)),
\end{equation*}
where the category $\Calk_{\omega_1}^{\cont}(\cC)\in\Cat^{\perf}$ is introduced in \cite[Definition 1.60]{E24}. In this paper we will simply write $\Calk_{\omega_1}(\cC)^{\omega}$ for the latter category.

Finally, we recall the continuous $K$-theory and more general (accessible) localizing invariants for dualizable categories, as defined in \cite{E24}. In this paper we consider the $\Sp$-valued localizing invariants $F:\Cat^{\perf}\to\Sp$ \cite{BGT}. For such $F$ and for $\cC\in\Cat_{\st}^{\dual}$ we put
\begin{equation*}
F^{\cont}(\cC)=\Omega F(\Calk_{\omega_1}(\cC)^{\omega}).
\end{equation*} 
Then $F^{\cont}:\Cat_{\st}^{\dual}\to\Sp$ is a localizing invariant by \cite[Proposition 4.6]{E24} and we have a functorial isomorphism $F^{\cont}(\Ind(\cA))\simeq F(\cA)$ for $\cA\in\Cat^{\perf}$ by \cite[Proposition 4.7]{E24}. Moreover, $F^{\cont}$ is uniquely determined by these properties by \cite[Theorem 4.10]{E24}.

\subsection{Coherently assembled exact categories}
\label{ssec:coh_ass_exact_cats}

The following definition specifies the class of large exact categories for which one can define $K$-theory and other localizing invariants.

\begin{defi}\label{def:coherently_assembled_exact}
Let $\cE$ be a Grothendieck exact category. We say that $\cE$ is coherently assembled if $\cE$ is compactly assembled and the functor $\hat{\cY}:\cE\to\Ind(\cE)$ is exact.

We denote by $\Cohass_{\ex}$ the category of coherently assembled exact categories, where the $1$-morphisms are strongly continuous exact functors. 
\end{defi}

\begin{remark}\label{rem:hat_Y_for_locally_kappa_coherent}
In the situation of Definition \ref{def:coherently_assembled_exact} suppose that $\cE$ is locally $\kappa$-coherent for a regular cardinal $\kappa.$ Then the essential image of $\hat{\cY}$ is contained in $\Ind(\cE^{\kappa}).$ The inclusion $\Ind(\cE^{\kappa})\to\Ind(\cE)$ reflects exactness by Propositions \ref{prop:properties_of_Ind_kappa_of_small_exact} and \ref{prop:kappa_compact_closed_under_extensions}, hence the functor $\hat{\cY}:\cE\to\Ind(\cE^{\kappa})$ is exact.
\end{remark}

\begin{example}
Let $\cE$ be a dualizable (presentable stable) category with the standard exact structure. Then $\cE$ is coherently assembled. This gives a fully faithful functor $\Cat_{\st}^{\dual}\hto\Cohass_{\ex}.$
\end{example}

We will see below that a coherently assembled exact category is always locally $\omega_1$-coherent. First we observe an equivalent characterization.

\begin{prop}\label{prop:coherently_assembled_as_retract}
Let $\cE$ be a Grothendieck exact category. The following are equivalent.
\begin{enumerate}[label=(\roman*),ref=(\roman*)]
	\item $\cE$ is coherently assembled. \label{coherently_assembled}
	\item There exists a locally coherent exact category $\cE'$ and a retraction $\cE\xto{F}\cE'\xto{G}\cE,$ such that both $F$ and $G$ are exact and commute with filtered colimits. \label{retract_of_locally_coherent}
\end{enumerate}
In particular, if $\cE$ is locally coherent, then it is coherently assembled.
\end{prop}

\begin{proof}
Note that the final assertion is obvious: if $\cE$ is locally coherent, then $\hat{\cY}$ is obtained by applying $\Ind$ to the exact inclusion $\cE^{\omega}\to\cE.$

\Implies{coherently_assembled}{retract_of_locally_coherent}. Suppose that $\cE$ is locally $\kappa$-coherent. Consider the retraction $\cE\xto{\hat{\cY}}\Ind(\cE^{\kappa})\xto{\colim}\cE.$ Both functors commute with filtered colimits. By Remark \ref{rem:hat_Y_for_locally_kappa_coherent} the functor $\hat{\cY}$ is exact, and by definition of a Grothendieck exact category the functor $\colim$ is exact. This proves the implication.

\Implies{retract_of_locally_coherent}{coherently_assembled}. Since $F$ and $G$ are continuous, the functor $\hat{\cY}_{\cE}$ is isomorphic to the composition
\begin{equation*}
\cE\xto{F}\cE'\xto{\hat{\cY}_{\cE'}} \Ind(\cE')\xto{\Ind(G)}\Ind(\cE).
\end{equation*}
Here the first and the third functors are exact by assumption, and the second functor is exact by the above observation on the locally coherent case. This proves the proposition.
\end{proof}

We explain a natural analogue of split exact structures in the context of large exact categories, see also Remark \ref{rem:dualizable_additive_infty_categories} below.

\begin{prop}\label{prop:minimal_exact_structure}
Let $\cE$ be a compactly assembled additive category. Then $\cE$ has an exact structure such that a map $f:x\to y$ in $\cE$ is an exact projection if and only if any compact morphism $z\to y$ factors through $x.$ With this exact structure $\cE$ is a coherently assembled exact category. Moreover, short exact sequences in $\cE$ are exactly the filtered colimits of split short exact sequences. 
\end{prop}

\begin{proof}
We first consider the case when $\cE$ is compactly generated, and put $\cA=\cE^{\omega}.$ Then $\cE$ is identified with the category of flat additive functors $\cA^{op}\to\Sp_{\geq 0},$ this is a straightforward generalization of \cite[Theorem 7.2.2.15]{Lur17a}. The definition of an exact projection in this case gives the class of morphisms which are effective epimorphisms in the ambient category $\Fun^{\add}(\cA^{op},\Sp_{\geq 0}).$ Hence, we simply obtain the standard exact structure on $\cE\simeq\Ind(\cA).$

Now consider the general case. Choose a retraction $\cE\xto{F}\cE'\xto{G}\cE,$ where $\cE'$ is additive and compactly generated, $F$ and $G$ are additive and continuous, and $F$ is moreover strongly continuous (e.g. take $\cE'=\Ind(\cE^{\omega_1})$). Denote by $\check{\St}^{\add}(\cE)$ the category of cofiltered limit-preserving functors $\cE^{op}\to\Sp.$ As in Proposition \ref{prop:check_St_of_Grothendieck_exact_category} we see that we have a universal continuous additive functor $j:\cE\to\check{\St}^{\add}(\cE).$ Consider also a similar functor $j':\cE'\to\check{\St}^{\add}(\cE').$ The latter is simply the ind-completion of the functor $\cE^{'\omega}\to\St(\cE^{'\omega}),$ where $\cE^{'\omega}$ is equipped with the split exact structure. It follows from the retraction that $j$ is fully faithful and its essential image is closed under extensions. Consider the induced exact structure on $\cE.$ Using the retraction again we see that $\cE$ is a coherently assembled exact category (by Proposition \ref{prop:coherently_assembled_as_retract}), and any short exact sequence in $\cE$ is a filtered colimit of split short exact sequences. It remains to obtain the stated description of exact projections in $\cE.$

If $f:x\to y$ is an exact projection in $\cE,$ then $f$ is a filtered colimit of split exact projections $x_i\to y_i.$ Any compact morphism $z\to y$ factors through some $y_i,$ hence through $x_i,$ hence through $x.$

Conversely, left $f:x\to y$ be a morphism in $\cE$ such that any compact morphism $z\to y$ factors through $x.$ Let $\hat{\cY}(y)\cong \inddlim[i\in I]y_i,$ where $I$ is directed. Each map $y_i\to y$ is compact, hence so is $F(y_i)\to F(y)$ by the strong continuity of $F.$ Moreover, for each $i$ the map $F(y_i)\to F(y)$ factors through $F(x).$ It follows from the above special case that $F(x)\to F(y)$ is an exact projection in $\cE'.$ Applying $G:\cE'\to\cE,$ we conclude that $x\to y$ is an exact projection in $\cE.$
\end{proof}

\begin{remark}\label{rem:dualizable_additive_infty_categories}
The category $\check{\St}^{\add}(\cE)$ from the proof of Proposition \ref{prop:minimal_exact_structure} has a natural $t$-structure such that the connective part $\cC=\check{\St}^{\add}(\cE)_{\geq 0}$ is identified with the category of cofiltered limit-preserving functors $\cE^{op}\to\Sp_{\geq 0}.$ Moreover, $\cC$ is a dualizable module over $\Sp_{\geq 0},$ i.e. a dualizable object of $\Pr^L_{\Add}$ -- the symmetric monoidal category of additive presentable $\infty$-categories \cite{LLS}. Moreover, any dualizable additive category $\cC$ can be obtained in this way: one recovers $\cE$ as the full subcategory of flat objects of $\cC.$ Here an object $x\in\cC$ is flat if the functor $\ev(x,-):\cC^{\vee}\to\Sp_{\geq 0}$ commutes with finite limits.

The dualizable additive categories are also exactly the categories of connective almost modules in the setting of higher almost ring theory, studied by Hebestreit and Scholze \cite{HS24}. The notion of a flat object is compatible with the standard notion of an almost flat almost module \cite{Fa02, GR03}.
\end{remark}

Next we show that the presentable stable envelope is well-behaved.

\begin{prop}\label{prop:St_of_coherently_assembled}
\begin{enumerate}[label=(\roman*),ref=(\roman*)]
	\item Let $\cE$ be a coherently assembled exact category. The category $\check{\St}(\cE)$ is dualizable and the functor $j:\cE\to\check{\St}(\cE)$ is strongly continuous and fully faithful. Moreover, the essential image of $\cE$ in $\check{\St}(\cE)$ is closed under extensions, and $j$ reflects exactness. \label{St_of_coh_ass}
	\item Let $F:\cE\to\cE'$ be a strongly continuous exact functor between coherently assembled exact categories. Then the induced functor $\check{\St}(F):\check{\St}(\cE)\to\check{\St}(\cE')$ is also strongly continuous. Therefore, we have a well-defined functor $\check{\St}:\Cohass_{\ex}\to\Cat_{\st}^{\dual}.$ It is left adjoint to the inclusion. \label{St_from_cohass_ex_to_cat_st_dual}
\end{enumerate}
\end{prop}

\begin{proof}
\ref{St_of_coh_ass} We first consider the case when $\cE$ is locally coherent. Then by Proposition \ref{prop:St_of_locally_coherent} the category $\check{\St}(\cE)\simeq \Ind(\St(\cE^{\omega}))$ is compactly generated, hence dualizable. The strong continuity is immediate: the functor $j$ is identified with $\Ind(\cE^{\omega}\to \St(\cE^{\omega})).$ The fully faithfulness follows from Proposition \ref{prop:fully_faithful_E_to_St_of_E}. The remaining assertions follow from Proposition \ref{prop:exact_structure_on_Ind_kappa}.

Now consider the general case. By Proposition \ref{prop:coherently_assembled_as_retract} we have a retraction $\cE\xto{F}\cE'\xto{G}\cE,$ where $\cE'$ is locally coherent and both $F$ and $G$ are exact and continuous. Now, both functors $\check{\St}(F)$ and $\check{\St}(G)$ are continuous, and the functor $j:\cE\to\check{\St}(\cE)$ is a retract of the functor $j':\cE'\to \check{\St}(\cE').$ Hence, all the assertions about $j$ and $\check{\St}(\cE)$ follow from the above special case when $\cE$ is locally coherent.

\ref{St_from_cohass_ex_to_cat_st_dual} Put $\Phi=\check{\St}(F).$ We need to show that for $x\in\check{\St}(\cE)$ we have an isomorphism $\hat{\cY}(\Phi(x))\xto{\sim}\Ind(\Phi)(\hat{\cY}(x)).$ It suffices to consider the case $x=j(y),$ where $y\in\cE$ and $j:\cE\to \check{\St}(\cE)$ is the universal continuous exact functor. Then the isomorphism follows directly from \ref{St_of_coh_ass}: we have a commutative square of compactly assembled categories
\begin{equation*}
\begin{tikzcd}
\cE\ar[r, "F"]\ar[d] & \cE'\ar[d]\\
\check{\St}(\cE)\ar[r, "\Phi"] & \check{\St}(\cE'),
\end{tikzcd}
\end{equation*}
in which both vertical arrows and the upper horizontal arrow are strongly continuous.

It follows formally that the functor $\check{\St}:\Cohass_{\ex}\to\Cat_{\st}^{\dual}$ is left adjoint to the inclusion: the adjunction counit $\cE\to\check{\St}(\cE)$ is well-defined by \ref{St_of_coh_ass}.
\end{proof}

The above proposition allows to define the continuous $K$-theory and other localizing invariants (such as $KH$) for coherently assembled exact categories.

\begin{defi}\label{def:F^cont_for_coh_ass_exact_categories}
Let $F:\Cat^{\perf}\to \Sp$ be a localizing invariant. For a coherently assembled exact category $\cE$ we put
\begin{equation*}
F^{\cont}(\cE)=F^{\cont}(\check{\St}(\cE)).
\end{equation*} 
This defines a functor
\begin{equation*}
F^{\cont}:\Cohass_{\ex}\to\Sp.
\end{equation*} 
\end{defi}

Next, we show the expected local $\omega_1$-coherence.

\begin{prop}\label{prop:coherently_assembled_is_locally_omega_1_coherent}
Let $\cE$ be a coherently assembled exact category. Then $\cE$ is locally $\omega_1$-coherent.
\end{prop}

\begin{proof}
It suffices to check that the conditions from Proposition \ref{prop:criterion_locally_kappa_coherent} are satisfied. Since $\cE$ is compactly assembled, it is $\omega_1$-accessible by \cite[Proposition 1.24]{E24}. Next, all filtered colimits are exact in $\cE,$ in particular, $\omega_1$-filtered colimits are exact. It remains to prove that for $x\in\cE^{\omega_1}$ the functor $\Ext^1_{\cE}(x,-):\cE\to\Ab$ commutes with $\omega_1$-filtered colimits. By Proposition \ref{prop:St_of_coherently_assembled} the functor $j:\cE\to\check{\St}(\cE)$ is strongly continuous, in particular, it preserves $\omega_1$-compact objects. Hence, we have $j(x)\in\check{\St}(\cE)^{\omega_1}.$ It follows that the functor $\Ext^1_{\check{\St}(\cE)}(j(x),-):\check{\St}(\cE)\to\Ab$ commutes with $\omega_1$-filtered colimits by stability. Applying Proposition \ref{prop:St_of_coherently_assembled} again, we see that for $y\in\cE$ we have an isomorphism $\Ext^1_{\cE}(x,y)\cong \Ext^1_{\check{\St}(\cE)}(j(x),j(y)).$ Hence, the functor $\Ext^1_{\cE}(x,-)$ also commutes with $\omega_1$-filtered colimits.
\end{proof}

From now on, for a coherently assembled exact category $\cE$ we consider the small category $\cE^{\omega_1}$ as an exact category with the induced exact structure from $\cE,$ and similarly for any regular $\kappa\geq\omega_1.$ Recall that by Proposition \ref{prop:various_E^lambda} for any uncountable regular cardinals $\kappa\leq \lambda$ the functor $\St(\cE^{\kappa})\to\St(\cE^{\lambda})$ is fully faithful, and we defined $\St(\cE)$ to be the directed union of $\St(\cE^{\kappa}).$ 

\begin{prop}\label{prop:St_of_E_to_check_St_of_E}
Let $\cE$ be a coherently assembled exact category. Then the natural functor $\St(\cE)\to \check{\St}(\cE)$ is fully faithful.
\end{prop}

\begin{proof}
By Proposition \ref{prop:coherently_assembled_as_retract} we have a retraction $\cE\xto{F}\cE'\xto{G}\cE,$ where $\cE'$ is locally coherent and both $F$ and $G$ are continuous and exact. Hence, the functor $\St(\cE)\to\check{\St}(\cE)$ is a retract of the functor $\St(\cE')\to\check{\St}(\cE').$ Thus, we may and will assume that $\cE$ is locally coherent.

Consider the functor $j:\cE\to\check{\St}(\cE),$ which is fully faithful by Proposition \ref{prop:St_of_coherently_assembled}. Let $\kappa$ be an uncountable regular cardinal. By Proposition \ref{prop:various_E^lambda} it suffices to prove that the functor $\St(\cE^{\kappa})\to\check{\St}(\cE)$ is fully faithful. Take the restriction $i=j_{\mid \cE^{\kappa}}:\cE^{\kappa}\to\check{\St}(\cE).$ By \cite[Lemmas 4.2 and 4.3]{SW26} it suffices to prove that for $x\in\cE^{\kappa}$ and for $n>0$ the functor $\Ext^n_{\check{\St}(\cE)}(i(-),i(x)):(\cE^{\kappa})^{op}\to\Ab$ is weakly effaceable. Namely, we need to show that for any $y\in\cE^{\kappa}$ and for any element $\alpha\in\Ext^n_{\check{\St}(\cE)}(i(y),i(x))$ there exists a projection $z\to y$ in $\cE^{\kappa}$ such that the image of $\alpha$ in $\Ext^n_{\check{\St}(\cE)}(i(z),i(y))$ is zero.

We first consider the case $y\in\cE^{\omega}.$ Let $x\cong \indlim[s\in S]x_s,$ where $S$ is directed and $x_s\in\cE^{\omega}.$ By Proposition \ref{prop:St_of_locally_coherent} $i(y)$ is compact in $\check{\St}(\cE),$ hence there exists $s\in S$ such that $\alpha$ is the image of some element  $\beta\in\Ext^n_{\check{\St}(\cE)}(i(y),i(x_s))\cong \Ext^n_{\cE^{\omega}}(y,x_s).$ By \cite[Lemma 4.3]{SW26} there exists a projection $z\to y$ in $\cE^{\omega}$ such that the image of $\beta$ in $\Ext^n_{\check{\St}(\cE)}(i(z),i(x_s))$ vanishes. Then the image of $\alpha$ in $\Ext^n_{\check{\St}(\cE)}(i(z),i(x))$ vanishes, as required.

Next, if $y$ is a $\kappa$-small direct sum of some $y_k\in\cE^{\omega},$ then it suffices to take $z$ to be the direct sum of $z_k\in\cE^{\omega},$ where the projections $z_k\to y_k$ are as above. Now consider the general case, i.e. $y=\indlim[t\in T]y_t,$ where $T$ is a $\kappa$-small directed poset, and $y_t\in\cE^{\omega}.$ It suffices to prove that the map
\begin{equation}\label{eq:from_coproduct_to_colimit}
\biggplus[t\in T]y_t\to \indlim[t\in T]y_t
\end{equation}
is a projection in $\cE^{\kappa}.$ By Proposition \ref{prop:criterion_locally_kappa_coherent} it suffices to prove that \eqref{eq:from_coproduct_to_colimit} is a projection in $\cE.$ We may and will assume that the exact structure on $\cE^{\omega}$ is split, i.e. we have only split short exact sequences. Then $\cE$ is identified with the category of flat additive functors $(\cE^{\omega})^{op}\to \Sp_{\geq 0},$ with the induced exact structure from the Grothendieck prestable 
category $\cC=\Fun^{\add}((\cE^{\omega})^{op},\Sp_{\geq 0}).$ The class of flat objects in $\cC$ is closed under fibers of effective epimorphisms, which implies that \eqref{eq:from_coproduct_to_colimit} is a projection in $\cE,$ as required.
\end{proof}

The following corollary is immediate.

\begin{cor}\label{cor:mono_for_Ext^2}
Let $\cE$ be a coherently assembled exact category, and let $\cC$ be a presentable stable category. Suppose that we have a continuous exact functor $\Phi:\cE\to\cC$ which is fully faithful, reflects exactness and its essential image is closed under extensions. Consider the induced (continuous exact) functor $\Psi:\check{\St}(\cE)\to\cC,$ and denote by $j:\cE\to\check{\St}(\cE)$ the universal continuous exact functor, so that $\Psi\circ j\cong\Phi.$ Then for $x,y\in\cE$ the map 
\begin{equation}\label{eq:map_on_Exts}
\Ext^n_{\check{\St}(\cE)}(j(x),j(y))\to\Ext^n_{\cC}(\Phi(x),\Phi(y))
\end{equation} is an isomorphism for $n\leq 1,$ and a monomorphism for $n=2.$  
\end{cor}

\begin{proof}
Let $\kappa$ be an uncountable regular cardinal such that $x,y\in\cE^{\kappa}.$ By Propositions \ref{prop:St_of_E_to_check_St_of_E} and \ref{prop:various_E^lambda} the source of \eqref{eq:map_on_Exts} is identified with $\Ext^n_{\cE^{\kappa}}(x,y).$ Further, by Proposition \ref{prop:kappa_compact_closed_under_extensions} the essential image $\Phi(\cE^{\kappa})$ in $\cC$ is also closed under extensions. Hence, the assertion follows from Proposition \ref{prop:mono_on_Ext^2_for_small}.
\end{proof}

Finally, we prove a version of \cite[Theorem 4.5]{SW26}, which is deduced from loc. cit.

\begin{prop}\label{prop:left_special_coherently_assembled}
Let $\Phi:\cU\to\cE$ be a functor between coherently assembled exact categories, which is fully faithful, strongly continuous, reflects exactness, and such that the essential image of $\Phi$ is closed under extensions. Suppose that the inclusion $\Phi(\cU^{\omega_1})\hto \cE^{\omega_1}$ is left special (Definition \ref{def:left_special}). Then the induced functor $\Psi:\check{\St}(\cU)\to\check{\St}(\cE)$ is fully faithful.
\end{prop}

\begin{proof}
The functor $\Psi$ is a retract (in $\Pr^L_{\st}$) of the functor $\Ind(\St(\cU^{\omega_1}))\to\Ind(\St(\cE^{\omega_1})),$ which is fully faithful by Theorem \ref{th:left_special_implies_fully_faithful}.
\end{proof}

\begin{remark}
In Proposition \ref{prop:left_special_coherently_assembled} one can equivalently assume that $\Phi(\cU)$ is left special in $\cE.$ However, we prefer to formulate the condition in terms of small categories.
\end{remark}

\subsection{Coherently assembled abelian categories}
\label{ssec:coh_ass_abelian_cats}

\begin{defi}\label{def:coh_ass_abelian_cats}
A Grothendieck abelian category $\cA$ is coherently assembled if it is such as an exact category. This means that if $\cA$ is compactly assembled and the functor $\hat{\cY}:\cA\to\Ind(\cA)$ is exact.

We denote by $\Cohass_{\ab}$ the category of coherently assembled abelian categories, where the $1$-morphisms are strongly continuous exact functors.
\end{defi}

\begin{remark}
Note that a Grothendieck abelian category $\cA$ is compactly assembled if and only if $\cA$ satisfies (AB6) (the same proof as in \cite[Proposition 1.53]{E24}).
\end{remark}

We recall the following statement from \cite{E24}.

\begin{prop}\label{prop:coh_ass_comp_gen}\cite[Proposition E.8]{E24}
	Let $\cA$ be a Grothendieck abelian category which is compactly generated. Then $\cA$ is coherently assembled if and only if it is locally coherent.
\end{prop}

For completeness we give the following non-standard characterization of locally noetherian abelian categories.

\begin{prop}
Let $\cA$ be a Grothendieck abelian category. The following are equivalent.
\begin{enumerate}[label=(\roman*),ref=(\roman*)]
	\item $\cA$ is locally noetherian. \label{loc_noeth}
	\item $\cA$ is compactly assembled and for any object $x\in\cA$ the map $\hat{\cY}(x)\to\cY(x)$ is a monomorphism in $\cA.$ \label{mono_from_hat_Y_to_Y}
\end{enumerate}
\end{prop}

\begin{proof}
\Implies{loc_noeth}{mono_from_hat_Y_to_Y}. %Since $\cA$ is locally noetherian, it is locally coherent, hence coherently assembled by Proposition \ref{prop:coh_ass_comp_gen}. 
Any object $x\in\cA$ is a directed union of noetherian subobjects $x_i\subset x,$ hence the map $\hat{\cY}(x)\cong\indlim[i]\cY(x_i)\to\cY(x)$ is a monomorphism.

\Implies{mono_from_hat_Y_to_Y}{loc_noeth}. We first observe that $\cA$ is coherently assembled. Indeed, the functor $\hat{\cY}:\cA\to\Ind(\cA)$ is automatically right exact, so we only need to check that it preserves monomorphisms. If $x\to y$ is a monomorphism in $\cA,$ then in the commutative square
\begin{equation*}
\begin{tikzcd}
\hat{\cY}(x) \ar[r]\ar[d] & \hat{\cY}(y)\ar[d]\\
\cY(x)\ar[r] & \cY(y)
\end{tikzcd}
\end{equation*}
both vertical arrows and the lower horizontal arrow are monomorphisms, hence so is the upper horizontal arrow, as required.

Next, we see that the functor
\begin{equation*}
\Phi:\cA\to\Ind(\cA),\quad \Phi(x)=\coker(\hat{\cY}(x)\to\cY(x)),
\end{equation*}
is exact. Moreover, an object $x\in\cA$ is compact if and only if $\Phi(x)=0.$

Now take some $x\in\cA,$ and let $\hat{\cY}(x)\cong\indlim[i\in I]\cY(x_i),$ where $I$ is directed. Since the map $\hat{\cY}(x)\to\cY(x)$ is a monomorphism, we may and will assume that each map $x_i\to x$ is a monomorphism in $\cA.$ For each $i\in I$ the map $\Phi(x_i)\to\Phi(x)$ is zero, hence $\Phi(x_i)=0$ by the left exactness of $\Phi.$ Therefore, each $x_i$ is compact and $\cA$ is compactly generated. By Proposition \ref{prop:coh_ass_comp_gen} $\cA$ is locally coherent. It remains to show that the class of compact objects in $\cA$ is closed under subobjects. This again follows from the left exactness of $\Phi:$ if $x\in\cA$ is compact and $y\subset x,$ then $\Phi(x)=0,$ hence $\Phi(y)=0,$ i.e. $y$ is compact.
\end{proof}

We record the special cases of general results on coherently assembled exact categories.

\begin{cor}\label{cor:coh_ass_abelian_general}
Let $\cA$ be a coherently assembled abelian category.

\begin{enumerate}[label=(\roman*),ref=(\roman*)]
\item $\cA$ is locally $\omega_1$-coherent. \label{coh_ass_locally_omega_1_coherent}
\item The unseparated derived category $\check{D}(\cA)$ is dualizable and the functor $\cA\to\check{D}(\cA)$ is strongly continuous. Moreover, we have a well-defined functor $\check{D}(-):\Cohass_{\ab}\to\Cat_{\st}^{\dual}.$ \label{check_D_is_dualizable}
\end{enumerate} 
\end{cor}

\begin{proof}
\ref{coh_ass_locally_omega_1_coherent} is a special case of Proposition \ref{prop:coherently_assembled_is_locally_omega_1_coherent}, and \ref{check_D_is_dualizable} follows from Propositions \ref{prop:St_of_Grothendieck_abelian} and \ref{prop:St_of_coherently_assembled}.
\end{proof}

The localizing invariants of coherently assembled abelian categories are obtained by applying Definition \ref{def:F^cont_for_coh_ass_exact_categories}.

\begin{prop}\label{prop:coh_ass_ab_retraction_left_exact}
Let $\cA$ be a Grothendieck abelian category and let $\cB$ be a coherently assembled abelian category. Suppose that we have a retraction $\cA\xto{F}\cB\xto{G}\cA,$ such that both $F$ and $G$ are left exact and commute with filtered colimits. Then $\cA$ is coherently assembled.
\end{prop}

\begin{proof}
The retraction implies that $\cA$ is compactly assembled. The functor $\hat{\cY}_{\cA}:\cA\to\Ind(\cA)$ is automatically right exact (it is a left adjoint). Hence, we only need to check the left exactness. This follows from the retraction since both $F$ and $G$ are left exact.
\end{proof}

Some non-trivial examples of (not locally coherent) coherently assembled abelian categories are given in Section \ref{sec:examples_of_comp_ass_and_coh_ass}.

We recall the following notation. We denote by $\Groth^{\lex}_{\ab}$ the category of Grothendieck abelian categories, where the $1$-morphisms are colimit-preserving left exact functors. We denote by $\Groth^{c,\lex}_{\ab}\subset\Groth^{\lex}$ the non-full subcategory with the same objects, where $1$-morphisms are functors $F:\cC\to\cD$ such that the right adjoint $F^R:\cD\to\cC$ commutes with filtered colimits. The following statement is certainly known to experts, but we could not find a reference.

\begin{prop}
The categories $\Groth^{c,\lex}_{\ab}$ and $\Groth^{\lex}_{\ab}$ have filtered colimits, and the (non-full subcategory inclusion) functors $\Groth^{c,\lex}_{\ab}\to\Groth^{\lex}_{\ab}\to\Pr^L$ commute with filtered colimits. 
\end{prop}

\begin{proof}
Consider an ind-system $(\cA_i)_{i\in I}.$ By \cite[Proposition C.5.5.20]{Lur18} we have a fully faithful functor $\Groth^{\lex}_{\ab}\to \Groth^{lex}_{\infty},$ $\cA\mapsto \check{D}(\cA)_{\geq 0},$ left adjoint to the functor $\cC\mapsto\tau_{\leq 0}\cC.$ Putting $\cC_i=\check{D}(\cA_i)_{\geq 0},$ we obtain an ind-system $(\cC_i)_{i\in I}$ in $\Groth_{\infty}^{\lex}$ such that $\cA_i\simeq \tau_{\leq 0}\cC_i$ (functorially in $i\in I$). By \cite[Proposition C.3.3.5]{Lur18} the category $\Groth_{\infty}^{\lex}$ has filtered colimits, which are preserved by the functor $\Groth_{\infty}^{\lex}\to\Pr^L.$ Recall that a colimit in $\Pr^L$ can be identified with the limit of the associated diagram in $\Pr^R$ (with right adjoint transition functors). The right adjoints preserve discrete objects, which implies that $\indlim[i]\cA_i\cong\tau_{\leq 0}(\indlim[i]\check{D}(\cA_i)_{\geq 0})$ is a Grothendieck abelian category (the former colimit is also computed in $\Pr^L$). Using again the fully faithfulness of the functor $\cA\mapsto \check{D}(\cA)_{\geq 0},$ we conclude that the functor $\Groth^{\lex}_{\ab}\to\Pr^L$ creates filtered colimits, as stated.
	
It follows formally that the functor $\Groth^{c,\lex}_{\ab}\to\Groth^{\lex}_{\ab}$ also creates filtered colimits, by applying the argument from the proof of \cite[Proposition 3.5.3]{Lur18} or \cite[Proposition 1.7]{E24}.
\end{proof}

Note that $\Cohass_{\ab}\subset \Groth^{c,\lex}$ is a (strictly) full subcategory. The following result is an (almost) straightforward generalization of the following fact: if $(R_i)_{i\in I}$ is an ind-system of (ordinary) associative right coherent rings with left flat transition maps, then the ring $\indlim[i]R_i$ is right coherent.

\begin{prop}\label{prop:cohass_closed_under_filtered_colimits}
The full subcategory $\Cohass_{\ab}\subset\Groth^{c,\lex}_{\ab}$ is closed under filtered colimits.
%Let $I$ be a filtered $\infty$-category, and let $I\to\Groth^{c,\lex}_{\ab},$ $i\mapsto\cA_i,$ be a functor such that each $\cA_i$ is coherently assembled. Then the colimit $\cA=\indlim[i]\cA_i$ (taken in $\Groth^{c,\lex}_{\ab}$) is also a coherently assembled abelian category. 
%Moreover, all the functors $\cA_i\to\cA$ are left exact. 
\end{prop}

\begin{proof}
Let $I$ be a directed poset, and let $I\to\Cohass_{\ab},$ $i\mapsto\cA_i,$ be a functor.

First suppose that each $\cA_i$ is locally coherent. Then we have an $I$-indexed diagram of small abelian categories, $i\mapsto\cA_i^{\omega}.$ The transition functors are exact, hence the colimit is abelian and we have $\cA\simeq\Ind(\indlim[i]\cA_i^{\omega}).$ 
%The functors $\cA_i^{\omega}\to\cA^{\omega}$ are left exact, hence so are the functors $\cA_i\to\cA.$

The general case follows since the diagram $i\mapsto\cA_i$ in $\Groth^{\lex}_{\ab}$ is a retract of the diagram $i\mapsto\Ind(\cA_i^{\omega_1}).$
\end{proof}

Finally, we point out that the weak (AB5) axiom holds.

\begin{prop}\label{prop:weak_AB5_Groth_cohass}
The categories $\Groth^{c,\lex}_{\ab}$ and $\Cohass_{\ab}$ satisfy the weak (AB5) axiom: the class of fully faithful functors is closed under filtered colimits.
\end{prop}

\begin{proof}
By Proposition \ref{prop:cohass_closed_under_filtered_colimits} it suffices to prove this for $\Groth^{c,\lex}_{\ab}.$ In this case the proof is the same as in \cite[Proposition 1.67]{E24}.
\end{proof}

Finally, we mention an expected statement about localizations (Serre quotients).

\begin{prop}\label{prop:Serre_quotient_coh_ass}
Let $\cA$ be a coherently assembled abelian category, and let $\cB\subset\cA$ be a coherently assembled localizing subcategory such that the inclusion functor $i:\cB\to\cA$ is strongly continuous. Then the category $\cA/\cB$ is coherently assembled and the quotient functor $q:\cA\to\cA/\cB$ is strongly continuous.

If moreover $\cA$ and $\cB$ are locally coherent, then we have an equivalence $\Ind(\cA^{\omega}/\cB^{\omega})\xto{\sim}\cA/\cB.$
\end{prop}

\begin{proof}
The equivalence in the locally coherent case follows from the universal properties.

Denoting by $i^R$ and $q^R$ the right adjoints, for $x\in\cA$ we have a functorial short exact sequence
\begin{equation*}
0\to i(i^R(x))\to x\to q^R(q(x))\to 0.
\end{equation*}
By assumption, $i^R$ is continuous, hence so is $q^R.$ The retraction $\cA/\cB\xto{q^R}\cA\xto{q}\cA/\cB$ implies that $\cA/\cB$ is coherently assembled by Proposition \ref{prop:coh_ass_ab_retraction_left_exact}. The strong continuity of $q$ is already established.
\end{proof}

\subsection{The situation of d\'evissage}
\label{ssec:situation_of_devissage}

The conditions of Quillen's D\'evissage theorem \cite[Theorem 4]{Qui73} naturally generalize to coherently assembled abelian categories. However, this is not immediately straightforward and requires some clarification.

We first treat the case of small categories.

\begin{prop}\label{prop:devissage_condition_for_small}
Let $\cA$ be a small abelian category, and let $\cB\subset\cA$ be a strictly full abelian subcategory such that the inclusion functor $i:\cB\to\cA$ is exact. Suppose that every object of $\cA$ has a finite filtration with subquotients in $\cB.$ Then $\cB$ is closed under taking subobjects and quotients in $\cA.$ 
\end{prop}

\begin{proof}
It suffices to prove that for any morphism $f:x\to y$ in $\cA,$ if $x\in\cB$ then $\ker(f)\in\cB.$ Let $F_{\bullet}y$ be a finite increasing filtration such that $F_0 y=0$ and $F_n y=y.$ If $n=0,$ then there is nothing to prove. If $n>0,$ then put $y'=F_{n-1}y,$ $x'=\ker(x\to y/y')$ and $f'=f_{\mid x'}:x'\to y'.$ Then $\ker(f)=\ker(f'),$ and applying induction on $n$ we obtain $\ker(f)\in\cB.$   
\end{proof}

Next we consider the locally coherent case.

\begin{prop}\label{prop:devissage_condition_for_locally_coherent}
Let $\cA$ be a small abelian category and $\cB\subset\cA$ a strictly full abelian subcategory such that the inclusion functor $i:\cB\to\cA$ is exact. We identify $\Ind(\cB)$ with its essential image in $\Ind(\cA).$ The following are equivalent.
\begin{enumerate}[label=(\roman*),ref=(\roman*)]
	\item Every object of $\cA$ has a finite filtration with subquotients in $\cB.$ \label{devissage_for_small}
	\item $\Ind(\cB)\subset\Ind(\cA)$ is closed under subobjects and quotients, and every object of $\Ind(\cA)$ has an $\N$-indexed increasing exhaustive filtration with subquotients in $\Ind(\cB).$ \label{devissage_for_loc_coh}
	\item $\Ind(\cB)$ generates $\Ind(\cA)$ via extensions and filtered colimits. \label{loc_coh_generates_by_extensions_and_filtered_colimits}
\end{enumerate}
\end{prop}

\begin{proof}
\Implies{devissage_for_small}{devissage_for_loc_coh}. By Proposition \ref{prop:devissage_condition_for_small}, $\cB\subset\cA$ is closed under subobjects and quotients, hence the same holds for $\Ind(\cB)\subset\Ind(\cA).$ Abusing the notation we write $i$ for $\Ind(i):\Ind(\cB)\to\Ind(\cA).$ Then the right adjoint $i^R:\Ind(\cA)\to\Ind(\cB)$ commutes with filtered colimits and for any $x\in\Ind(\cA)$ the map $i(i^R(x))\to x$ is a monomorphism. More precisely, $i(i^R(x))$ is the largest subobject of $x$ which is contained in $\Ind(\cB).$ Define the functorial non-negative filtration $F_{\bullet}x$ inductively: $F_0 x=0$ and $F_{n+1}x\subset x$ is the preimage of $i(i^R(x/F_n x))\subset x/F_n x.$ By induction each of the endofunctors $F_n:\Ind(\cA)\to\Ind(\cA)$ commutes with filtered colimits. It follows from \ref{devissage_for_small} that for any $x\in\cA$ we have $F_n x=x$ for large $n.$ It follows that for any $x\in\Ind(\cA)$ the filtration $F_{\bullet}x$ is exhaustive. By construction, its subquotients are in $\Ind(\cB),$ which proves the implication.

The implication \Implies{devissage_for_loc_coh}{loc_coh_generates_by_extensions_and_filtered_colimits} is evident.

\Implies{loc_coh_generates_by_extensions_and_filtered_colimits}{devissage_for_small}. We identify $\Ind(\cA)$ with the heart of the standard $t$-structure on $\Ind(D^b(\cA))\simeq \check{D}(\Ind(\cA)).$ Then \ref{loc_coh_generates_by_extensions_and_filtered_colimits} implies that $\cB$ generates $\Ind(D^b(\cA))$ as a localizing subcategory. Hence, $\cB$ generates $D^b(\cA)$ as an idempotent-complete stable subcategory. Let $\cA'\subset\cA$ be the smallest weak Serre subcategory containing $\cB,$ i.e. $\cA'$ is generated by $\cB$ via extensions. Denote by $D^b_{\cA'}(\cA)\subset D^b(\cA)$ the (idempotent-complete stable) subcategory consisting of complexes with homology in $\cA'.$ Then $\cB$ is contained in $D^b_{\cA'}(\cA),$ hence the latter category coincides with $D^b(\cA)$ and $\cA'=\cA.$ This proves the implication.
\end{proof}

Finally, we consider the coherently assembled case.

\begin{prop}\label{prop:devissage_condition_for_coherently_assembled}
Let $\cA$ be a coherently assembled abelian category, and let $\cB\subset\cA$ be a strictly full coherently assembled abelian subcategory such that the inclusion functor $i:\cB\to\cA$ is strongly continuous and exact. The following are equivalent.
\begin{enumerate}[label=(\roman*),ref=(\roman*)]
	\item $\cB\subset\cA$ is closed under subobjects and quotients, and any object of $\cA$ has an $\N$-indexed exhaustive filtration with subquotients in $\cB.$ \label{devissage_for_coh_ass}
	\item $\cB$ generates $\cA$ via extensions and filtered colimits. \label{coh_ass_generates_by_extensions_and_filtered_colimits}
\end{enumerate}
\end{prop}

\begin{proof}
Again, the implication \Implies{devissage_for_coh_ass}{coh_ass_generates_by_extensions_and_filtered_colimits} is evident.

\Implies{coh_ass_generates_by_extensions_and_filtered_colimits}{devissage_for_coh_ass}. Denote by $\cE\subset\cA^{\omega_1}$ the weak Serre subcategory generated by $\cB^{\omega_1}$ (note that $\cE$ is not necessarily closed under countable coproducts in $\cA^{\omega_1}$). By Proposition \ref{prop:properties_of_Ind_kappa_of_small_exact} the full subcategory $\Ind(\cE)\subset\Ind(\cA^{\omega_1})$ is closed under extensions. By assumption the functor $\hat{\cY}_{\cA}:\cA\to\Ind(\cA^{\omega_1})$ sends $\cB$ to $\Ind(\cB^{\omega_1})\subset\Ind(\cE).$ Since $\hat{\cY}_{\cA}$ is exact, the condition \ref{coh_ass_generates_by_extensions_and_filtered_colimits} implies that we have $\hat{\cY}_{\cA}(\cA)\subset\Ind(\cE).$

We show that $\cB\subset\cA$ is closed under subobjects, and the assertion about quotients follows. Let $x\in\cB$ and $y\subset x,$ a priori $y\in\cA.$ Then $\hat{\cY}_{\cA}(y)\subset \hat{\cY}_{\cA}(x)\in\Ind(\cB^{\omega_1}),$ hence by Proposition \ref{prop:devissage_condition_for_locally_coherent} (applied to $\cB^{\omega_1}\subset\cE$) we have $\hat{\cY}_{\cA}(y)\in\Ind(\cB^{\omega_1}).$ Applying the colimit functor, we obtain $y\in\cB,$ as stated.

The same argument shows the existence of a filtration with required properties. Namely, let $x\in\cA,$ then by Proposition \ref{prop:devissage_condition_for_locally_coherent} the object $\hat{\cY}_{\cA}(x)\in\Ind(\cE)$ has an $\N$-indexed filtration with subquotients in $\Ind(\cB^{\omega_1}).$ Applying the colimit functor $\Ind(\cE)\to\cA,$ we obtain the required filtration on $x.$ This proves the implication.
\end{proof}

The following is a basic example when the d\'evissage condition holds.

\begin{prop}\label{prop:loc_nilp_endomorphisms_coherently_assembled}
Let $\cB$ be a coherenly assembled category. Let $\cA$ be the category of pairs $(x,N),$ where $x\in\cB$ and $N:x\to x$ is a locally nilpotent endomorphism, i.e. $\colim(x\xto{N}x\xto{N}\dots)=0.$ Then $\cA$ is coherently assembled, and the inclusion $\cB\to\cA,$ $x\mapsto (x,0),$ satisfies the conditions of Proposition \ref{prop:devissage_condition_for_coherently_assembled}.
\end{prop}

\begin{proof}
The only statement requiring a proof is that $\cA$ is coherently assembled. This can be deduced from Proposition \ref{prop:t_structure_on_nilpotent_endomorphisms} below. For competeness we give a more direct proof.

For any $n\in\N$ consider the full subcategory $\cA_n\subset\cA$ consisting of pairs $(x,N)$ such that $N^n=0.$ Then we have a direct sequence $(\cA_n)_{n\geq 0}$ in $\Groth^{c,\lex}_{\ab}.$ By Proposition \ref{prop:weak_AB5_Groth_cohass} the functor $\indlim[n]\cA_n\to\cA$ is fully faithful. It is an equivalence since $\cA$ is generated by all $\cA_n$ via filtered colimits. By Proposition \ref{prop:cohass_closed_under_filtered_colimits} it suffices to prove that $\cA_n$ is coherently assembled for each $n\geq 0.$

We first note that $\cA_n$ is compactly assembled: the forgetful functor $\Phi_n:\cA_n\to\cB,$ $\Phi_n(x,N)=x,$ commutes with colimits, and the essential image of its left adjoint $\Phi_n^L$ generates the target via colimits. Next, the right adjoint $\Phi_n^R$ is continuous, and $\Phi_n$ is exact and conservative, hence  $\Ind(\Phi_n):\Ind(\cA_n)\to\Ind(\cB)$ is also exact and conservative. This formally implies that $\cA_n$ is coherently assembled since $\cB$ is.  
\end{proof}

We mention an example showing that one has to be cautious when dealing with the d\'evissage condition for coherently assembled categories. Suppose that we have a Grothendieck abelian category $\cA$ and a strictly full Grothendieck abelian subcategory $\cB\subset\cA,$ such that the inclusion functor $i:\cB\to\cA$ is exact and colimit-preserving, the right adjoint $i^R:\cA\to\cB$ is continuous, $\cB$ is closed under subobjects and quotients in $\cA,$ and every object of $\cA$ has an $\N$-indexed exhaustive filtration with subquotients in $\cB.$ It is tempting to think that if $\cB$ is coherently assembled, then so is $\cA.$ However, the following example shows that this is not the case in general, even if $\cB$ is locally coherent.

\begin{example}
Let $\mk$ be a field, and consider the ring $R=\mk[x,y_1,y_2,\dots]/(xy_n; n\geq 1).$ Let $\cA$ be the category of $R$-modules $M$ such that $x$ acts locally nilpotently on $M.$ Then $\cA$ is a Grothendieck category. Let $\cB=\Mod\hy R/x\subset\cA.$ Then $\cB$ is locally coherent since $R/x\cong\mk[y_1,y_2,\dots]$ is a sequential colimit of noetherian rings with flat transition maps. Further, $\cB\subset\cA$ is closed under subobjects and quotients, and any object $M\in\cA$ has a filtration $F_{\bullet}M,$ given by $F_n M=\ker(x^n:M\to M),$ with $\bigcup\limits_n F_n M=M$ and $\gr_n^F M\in\cB.$ The right adjoint to the inclusion is given by $M\mapsto\ker(x:M\to M),$ hence it is continuous. However, $\cA$ is not coherently assembled.

Assume the contrary. Then by Corollary \ref{cor:coh_ass_abelian_general} for any compact object $M\in\cA^{\omega}$ the functor $\Ext_{\cA}^1(M,-):\cA\to\Ab$ commutes with filtered colimits. Clearly, $R/x$ is compact in $\cA.$ Consider the direct sequence $(N_n)_{n\geq 0},$ given by $N_n=R/(x,y_1,\dots,y_n).$ Put $N=\indlim[n]N_n\cong\mk.$ Take the element $\alpha\in\Ext^1_{\cA}(R/x,N)$ corresponding to the short exact sequence
\begin{equation*}
0\to N\to R/x^2\to R/x\to 0.
\end{equation*}
Then $\alpha$ does not factor through $N_n$ for any $n,$ since $\Ext^1_{\cA}(R/x,N_n)$ is a quotient of $\Hom_R(Rx,N_n)=0.$ This gives a contradiction, hence $\cA$ is not coherently assembled. 
\end{example}

\section{Compactly assembled $t$-structures on dualizable categories}
\label{sec:comp_ass_t_structures}

\subsection{Reminder on $t$-structures}
\label{ssec:reminder_on_t_structures}

We start by recalling the basic notions related to $t$-structures, also fixing the notation. The basic reference is \cite{BBD82} in the context of triangulated categories, and \cite[Section 1.2]{Lur17a}, \cite[Appendix C]{Lur18} in the context of stable categories. 

Recall from \cite{BBD82} that a $t$-structure on a triangulated category $\cT$ is a pair of full subcategories $(\cT_{\geq 0},\cT_{\leq 0}),$ such that $\cT_{\geq 0}[1]\subset\cT_{\geq 0},$ $\cT_{\leq 0}[-1]\subset\cT_{\leq 0},$ we have $\Hom_{\cT}(y,z)=0$ for $y\in\cT_{\geq 0},$ $z\in\cT_{\leq 0}[-1],$ and for any $x\in\cT$ there exists an exact triangle
\begin{equation*}
\tau_{\geq 0}x\to x\to \tau_{\leq -1}x,\quad \tau_{\geq 0}x\in\cT_{\geq 0},\,\tau_{\leq -1}x\in\cT_{\leq 0}[-1].
\end{equation*}
A $t$-structure on a stable category $\cC$ is a pair of full subcategories $(\cC_{\geq 0},\cC_{\leq 0})$ such that $(\h \cC_{\geq 0},\h \cC_{\leq 0})$ is a $t$-structure on $\h \cC.$ We will only deal with $t$-structures on stable categories. 

If $\cD$ is another stable category with a $t$-structure $(\cD_{\geq 0},\cD_{\leq 0}),$ then an exact functor $F:\cC\to\cD$ is called left $t$-exact if $F(\cC_{\leq 0})\subset\cD_{\leq 0},$ and right $t$-exact if $F(\cC_{\geq 0})\subset\cD_{\geq 0}.$ Further, $F$ is called $t$-exact if it is both left and right $t$-exact.

Given a $t$-structure $(\cC_{\geq 0},\cC_{\leq 0})$ on a stable category $\cC,$ for any $a\in\Z$ we put
\begin{equation*}
\cC_{\geq a}=\cC_{\geq 0}[a],\quad \cC_{\leq a}=\cC_{\leq 0}[a].
\end{equation*}
We denote by $\tau_{\geq a}:\cC\to\cC_{\geq a}$ the right adjoint to the inclusion, and by $\tau_{\leq a}:\cC\to\cC_{\leq a}$ the left adjoint to the inclusion. We also denote by the same symbols the compositions $\cC\xto{\tau_{\geq a}}\cC_{\geq a}\hto\cC,$ $\cC\xto{\tau_{\leq a}}\cC_{\leq a}\hto\cC.$ This should not lead to confusion.

For integers $a\leq b$ we put $\cC_{[a,b]}=\cC_{\geq a}\cap\cC_{\leq b}.$ In particular, $\cC_{[0,0]}=\cC^{\heartsuit}$ is the heart of the $t$-structure, which is an abelian category. We denote by $\tau_{[a,b]}:\cC\to\cC_{[a,b]}$ the functor $\tau_{\geq a}\circ \tau_{\leq b}\cong\tau_{\leq b}\circ\tau_{\geq a}.$ Since we are working with not necessarily $\Z$-linear categories, we use the notation
\begin{equation*}
\pi_a(x)=\tau_{[a,a]}(x)[-a]\in\cC^{\heartsuit},\quad a\in\Z,\,x\in\cC.
\end{equation*}

We recall the notation
\begin{equation*}
\cC^+=\bigcup\limits_{a\in\Z}\cC_{\leq a},\quad \cC^-=\bigcup\limits_{a\in\Z}\cC_{\geq a},\quad \cC^b=\cC^+\cap\cC^-=\bigcup\limits_{n\geq 0}\cC_{[-n,n]}.
\end{equation*}
The $t$-structure $(\cC_{\geq 0},\cC_{\leq 0})$ is called left bounded resp. right bounded resp. bounded if $\cC=\cC^+$ resp. $\cC=\cC^-$ resp. $\cC=\cC^b.$ 

Recall that for a $t$-structure $(\cC_{\geq 0},\cC_{\leq 0})$ on $\cC,$ the left completion is given by the limit $\hhat{\cC}^l=\prolim[n\in\N]\cC_{\leq n}.$ This is a stable category, and it naturally has a $t$-structure $(\hhat{\cC}^l_{\geq 0},\hhat{\cC}^l_{\leq 0}),$ where $\hhat{\cC}^l_{\geq 0}=\prolim[n\in\N]\cC_{[0,n]},$ $\hhat{\cC}^l_{\leq 0}=\cC_{\leq 0}.$ We have a natural exact and $t$-exact functor $\cC\to\hhat{\cC}^l.$ If it is an equivalence, then the $t$-structure on $\cC$ is called {\it left complete}. Dually, we have a notion of the right completion $\hhat{\cC}^r=\prolim[n\in\N]\cC_{\geq -n},$ and similarly for right completeness. Note that in general the right completion of $\hhat{\cC}^l$ is equivalent to the left completion of $\hhat{\cC}^r,$ and both are given by $\prolim[n\in\N]\cC_{[-n,n]}.$

It is convenient to use the following terminology for brevity.

\begin{defi}
A small $t$-category is a small stable category $\cC$ with a bounded $t$-structure $(\cC_{\geq 0},\cC_{\leq 0}).$
\end{defi}

We recall that if $\cA$ is a small abelian category (considered also as an exact category), then we have $D^b(\cA)\simeq\St(\cA).$ In particular, if $\cC$ is a small $t$-category, then we have a realization functor $D^b(\cC^{\heartsuit})\to\cC,$ corresponding to the inclusion $\cC^{\heartsuit}\to\cC.$

\subsection{$t$-structures compatible with filtered colimits}

In this subsection we consider $t$-structures on presentable stable categories, before specializing to dualizable categories.

Recall that an accessible $t$-structure $(\cC_{\geq 0},\cC_{\leq 0})$ on $\cC\in\Pr^L_{\st}$ is said to be compatible with filtered colimits if the subcategory $\cC_{\leq 0}$ is closed under filtered colimits (note that $\cC_{\geq 0}$ is automatically closed under all colimits). Equivalently, this means that the functor $\tau_{\geq 0}:\cC\to\cC_{\geq 0}$ commutes with filtered colimits. In this case $\cC_{\geq 0}$ is a Grothendieck prestable category \cite[Proposition C.1.4.1]{Lur18}. In particular, the heart $\cC^{\heartsuit}$ is a Grothendieck abelian category.

It is convenient to introduce the following notation.

\begin{defi}\label{def:C^-_C^+_C^b}
Let $\cC$ be a presentable stable category, and let $(\cC_{\geq 0},\cC_{\leq 0})$ be an accessible $t$-structure compatible with filtered colimits. We denote by $\la \cC^-\ra\subset\cC$ the localizing subcategory generated by $\cC^-,$ and similarly for $\la \cC^+\ra$ and $\la \cC^b\ra.$

We say that the $t$-structure is continuously bounded if $\la\cC^b\ra=\cC.$
\end{defi}

Note that equivalently $\la \cC^-\ra$ is generated by $\cC_{\geq 0},$ $\la \cC^+\ra$ is generated by $\cC_{\leq 0},$ and $\la \cC^b\ra$ is generated by $\cC^{\heartsuit}.$ Since the $t$-structure is accessible, each of the categories $\la \cC^-\ra,$ $\la \cC^+\ra$ and $\la \cC^b\ra$ is presentable.

We first make some almost tautological observations.

\begin{prop}\label{prop:C^-_C^+_C^b_t_structures}
Let $\cC$ and $(\cC_{\geq 0},\cC_{\leq 0})$ be as in Definition \ref{def:C^-_C^+_C^b}. Then the $t$-structure on $\cC$ induces the $t$-structures on $\la\cC^-\ra,$ $\la\cC^+\ra$ and $\la\cC^b\ra,$ with the same heart. Moreover, we have equalities $\la\la \cC^-\ra^+\ra = \la\la \cC^+\ra^-\ra=\la\cC^b\ra$ (of strictly full subcategories of $\cC$). 
\end{prop}

\begin{proof}
The first statement is evident: the (continuous) composition $\cC\xto{\tau_{\geq 0}}\cC_{\geq 0}\hto\cC$ preserves each of the subcategories $\la\cC^-\ra,$ $\la\cC^+\ra$ and $\la\cC^b\ra.$ Next, we have an inclusion $\cC^b=\cC^-\cap\cC^+\subset \la \cC^-\ra^+,$ hence $\la\cC^b\ra\subset \la\la \cC^-\ra^+\ra.$ On the other hand, $\la \cC^-\ra_{\leq 0}$ is generated via extensions and filtered colimits by $\cC^-_{\leq 0}\subset\cC^b,$ hence we have $\la\cC^b\ra = \la\la \cC^-\ra^+\ra.$ The same argument proves the equality $\la\cC^b\ra = \la\la \cC^+\ra^-\ra.$
\end{proof}

We recall the following relation with the right completion.

\begin{prop}\label{prop:C^-_and_right_completion}
Let $\cC$ and $(\cC_{\geq 0},\cC_{\leq 0})$ be as in Definition \ref{def:C^-_C^+_C^b}.
\begin{enumerate}[label=(\roman*),ref=(\roman*)]
	\item We have $\la\cC^-\ra\simeq\Sp(\cC_{\geq 0}).$
	\item The right adjoint to the inclusion $\la\cC^-\ra\to\cC$ is given by $x\mapsto\indlim[n\in\N]\tau_{\geq -n}x.$ In particular, it commutes with colimits.
	\item We have $\cC=\la\cC^-\ra$ if and only if the $t$-structure $(\cC_{\geq 0},\cC_{\leq 0})$ is right complete.
\end{enumerate}
\end{prop}

\begin{proof}
All the assertions are well-known. Namely, the right completion of $\cC$ is identified with $\Sp(\cC_{\geq 0}),$ see for example \cite[Lemma A.8]{AN21}. For convenience, we write the objects of $\Sp(\cC_{\geq 0})$ as $(x_n)_{n\geq 0},$ where $x_n\in\cC_{\geq 0}$ and we assume the isomorphisms $x_n\cong \Omega x_{n+1}.$ The natural functor $\Phi:\cC\to\Sp(\cC_{\geq 0})$ is given by $x\mapsto (\tau_{\geq 0}(\Sigma^n x))_{n\geq 0}.$ Its left adjoint is given by
\begin{equation*}
\Phi^L:\Sp(\cC_{\geq 0})\to\cC,\quad \Phi^L((x_n)_{n\geq 0}) = \indlim[n] \Omega_{\cC}^n x_n
\end{equation*}
Since the $t$-structure is compatible with filtered colimits, the adjunction counit $\Id\to \Phi\circ\Phi^L$ is an isomorphism, i.e. $\Phi^L$ is fully faithful. It follows that the essential image of $\Phi^L$ is identified with $\la \cC^-\ra.$ This proves all the assertions.
\end{proof}

We deduce the following statement, which will have an important $K$-theoretic application.

\begin{prop}\label{prop:SOD_for_C_mod_C^b}
Let $\cC$ and $(\cC_{\geq 0},\cC_{\leq 0})$ be as in Definition \ref{def:C^-_C^+_C^b}.
\begin{enumerate}[label=(\roman*),ref=(\roman*)]
	\item The following (lax commutative) square commutes:
	\begin{equation}\label{eq:Beck_Chevalley1}
		\begin{tikzcd}
			\la \cC^+\ra \ar[r]\ar[d] & \cC\ar[d]\\
			\la \cC^b\ra \ar[r] & \la \cC^-\ra.
		\end{tikzcd}
	\end{equation}
	Here the horizontal functors are the inclusions and the vertical functors are right adjoints to the inclusions. \label{Beck_Chevalley1}
	\item We have a semi-orthogonal decomposition in $\Pr^L_{\st}:$
	\begin{equation}\label{eq:SOD_for_C_mod_C^b}
		\cC/\la \cC^b\ra = \la \la \cC^+\ra/\la \cC^b\ra, \la \cC^-\ra/\la \cC^b\ra\ra
	\end{equation}
	In particular, we have an equivalence
	\begin{equation}\label{eq:equivalence_of_quotients1}
		\la \cC^-\ra/\la \cC^b\ra\xto{\sim} \cC/\la \cC^+\ra.
	\end{equation} \label{SOD_for_C_mod_C^b}
\end{enumerate} 
\end{prop}

\begin{proof}
\ref{Beck_Chevalley1} follows directly from Propositions \ref{prop:C^-_C^+_C^b_t_structures} and \ref{prop:C^-_and_right_completion}: both vertical functors in \eqref{eq:Beck_Chevalley1} are given by $x\mapsto\indlim[n\in\N]\tau_{\geq -n}x.$

Next, \ref{SOD_for_C_mod_C^b} follows directly from \ref{Beck_Chevalley1}. Namely, the inclusion $\la\cC^-\ra/\la\cC^b\ra\to\cC/\la\cC^b\ra$ is colimit-preserving, hence it has a right adjoint. The right orthogonal is identified with $\la\cC^-\ra^{\perp}\subset\cC,$ or equivalently with $\{x\in \cC\mid \indlim[n\in\N]\tau_{\geq -n}x = 0\}.$ It follows from \ref{Beck_Chevalley1} that the image of $\la\cC^+\ra/\la\cC^b\ra$ in $\cC/\la\cC^b\ra$ is contained in the right orthogonal to $\la\cC^-\ra/\la\cC^b\ra.$ Since $\la\cC^-\ra$ and $\la\cC^+\ra$ generate $\cC,$ we obtain the semi-orthogonal decomposition \eqref{eq:SOD_for_C_mod_C^b}. The equivalence \eqref{eq:equivalence_of_quotients1} follows.
\end{proof}

\subsection{Compactly assembled $t$-structures}

%All the $t$-structures on presentable stable categories will be assumed to be accessible and compatible with filtered colimits. 
We first recall the notion of a compactly generated $t$-structure.

\begin{prop}\label{prop:compactly_generated_t_structure}
Let $\cT$ be a small stable category with a $t$-structure $(\cT_{\geq 0},\cT_{\leq 0}).$ Then the category $\Ind(\cT)$ has a $t$-structure $(\Ind(\cT)_{\geq 0},\Ind(\cT)_{\leq 0}),$ where $\Ind(\cT)_{\geq 0}$ is the essential image of the functor $\Ind(\cT_{\geq 0})\hto\Ind(\cT),$ and similarly for $\Ind(\cT)_{\leq 0}.$

Moreover, this $t$-structure on $\Ind(\cT)$ is continuously bounded if and only if the $t$-structure $(\cT_{\geq 0},\cT_{\leq 0})$ is bounded.
\end{prop}

\begin{proof}
The first assertion is \cite[Lemma C.2.4.3]{Lur18}. For the second assertion, note that the $t$-structure $(\cT_{\geq 0},\cT_{\leq 0})$ is bounded iff $\cT$ is generated by $\cT^{\heartsuit}$ as an idempotent-complete stable subcategory, which equivalently means that $\Ind(\cT)$ is generated by $\Ind(\cT^{\heartsuit})\simeq \Ind(\cT)^{\heartsuit}$ as a localizing subcategory, i.e. the $t$-structure $(\Ind(\cT)_{\geq 0},\Ind(\cT)_{\leq 0})$ is continuously bounded. 
\end{proof}

The following is immediate.

\begin{cor}\label{cor:Ind_of_T_a_b}
In the situation of Proposition \ref{prop:compactly_generated_t_structure}, for any integers $a\leq b$ the full subcategory $\Ind(\cT)_{[a,b]}\subset\Ind(\cT)$ is the essential image of the functor $\Ind(\cT_{[a,b]})\hto\Ind(\cT).$
\end{cor}

\begin{proof}
Clearly, the essential image of $\Ind(\cT_{[a,b]})$ is contained in $\Ind(\cT)_{[a,b]}.$ On the other hand, for an object $x=\inddlim[i]x_i\in\Ind(\cT)_{[a,b]}$ the maps $\inddlim[i]\tau_{\geq a}x_i\to\inddlim[i]x_i$ and $\inddlim[i]\tau_{\geq a}x_i\to\inddlim[i]\tau_{[a,b]}x_i$ are isomorphisms. This proves the inverse inclusion.
\end{proof}

It is technically convenient to consider the natural $t$-structure on the ``twice large'' category $\Ind(\cC)$ when $\cC$ is a presentable stable category with an accessible $t$-structure. Namely, $\Ind(\cC)$ is a union of $\Ind(\cC^{\kappa}),$ for sufficiently large regular $\kappa$ the $t$-structure on $\cC$ induces a $t$-structure on $\cC^{\kappa},$ and the inclusion functors $\Ind(\cC^{\kappa})\to\Ind(\cC^{\lambda})$ are $t$-exact. The compatibility of $(\cC_{\geq 0},\cC_{\leq 0})$ with filtered colimits exactly means that the functor $\colim:\Ind(\cC)\to\cC$ is $t$-exact (it is automatically right $t$-exact).

We introduce the following natural generalization of a compactly generated $t$-structure. 

\begin{defi}\label{def:compactly_assembled_t_structure}
Let $\cC$ be a dualizable category, and let $(\cC_{\geq 0},\cC_{\leq 0})$ be an accessible $t$-structure compatible with filtered colimits. We say that this $t$-structure is compactly assembled if the functor $\hat{\cY}:\cC\to\Ind(\cC)$ is $t$-exact.

We will say that $\cC$ is a dualizable $t$-category if the $t$-structure is compactly assembled and continuously bounded.
\end{defi}

\begin{remark}\label{rem:comp_ass_t_structures}
\begin{enumerate}[label=(\roman*),ref=(\roman*)]
	\item In the situation of the Definition \ref{def:compactly_assembled_t_structure} the functor $\hat{\cY}$ is automatically right $t$-exact, since its right adjoint is $t$-exact.
	\item A compactly generated $t$-structure is compactly assembled.
	\item Let $(\cC_{\geq 0},\cC_{\leq 0})$ be a compactly assembled $t$-structure on a dualizable category $\cC.$ It follows from Proposition \ref{prop:SOD_for_C_mod_C^b} that the condition $\cC=\la \cC^b\ra$ is equivalent to the equalities $\cC=\la\cC^-\ra=\la\cC^+\ra.$ By Proposition \ref{prop:C^-_and_right_completion} the condition $\cC=\la\cC^-\ra$ is equivalent to the right completeness. One can show that $\cC=\la \cC^+\ra$ if and only if the Grothendieck prestable category $\cC_{\geq 0}$ is anticomplete in the sense of \cite[Definition C.5.5.4]{Lur18}
\end{enumerate}
\end{remark}

%We first explain a direct relation with coherently assembled exact categories.

%\begin{prop}
%Let $\cC$ be a dualizable category, and let $(\cC_{\geq 0},\cC_{\leq 0})$ be an  accessible $t$-structure compatible with filtered colimits. Then the category $\cC_{\geq 0}$ is compactly assembled and the inclusion $\cC_{\geq 0}\to\cC$ is strongly continuous. Moreover, the following are equivalent.
%\begin{enumerate}[label=(\roman*),ref=(\roman*)]
%	\item The $t$-structure $(\cC_{\geq 0},\cC_{\leq 0})$ is compactly assembled.
%	\item The category $\cC_{\geq 0}$ is coherently assembled as an exact category.
%	\item The functor $\hat{\cY}:\cC_{\geq 0}\to\Ind(\cC_{\geq 0})$ is left exact (i.e. commutes with finite limits).
%\end{enumerate}
%\end{prop}

%\begin{proof}
%The right adjoint to the inclusion $\cC_{\geq 0}\to \cC$ commutes with filtered colimits, hence $\cC_{\geq 0}$ is compactly assembled. 

%\end{proof}

To avoid confusion we spell out the following special case.

\begin{prop}\label{prop:comp_ass_on_compactly_generated}
Let $\cC$ be a compactly generated presentable stable category, and let $(\cC_{\geq 0},\cC_{\leq 0})$ be an accessible $t$-structure compatible with filtered colimits. Then this $t$-structure is compactly assembled if and only if it induces a $t$-structure on $\cC^{\omega},$ which equivalently means that we are in the situation of Proposition \ref{prop:compactly_generated_t_structure}. 
\end{prop}

\begin{proof}
The ``if'' direction follows from Proposition \ref{prop:compactly_generated_t_structure}. For the ``only if'' direction we need to show that for $x\in\cC^{\omega}$ we have $\tau_{\geq 0}x\in\cC^{\omega}.$ This follows from the isomorphisms in $\Ind(\cC):$
\begin{equation*}
\hat{\cY}(\tau_{\geq 0}x)\cong \tau_{\geq 0}\hat{\cY}(x)\cong \tau_{\geq 0}\cY(x)\cong\cY(\tau_{\geq 0}x).\qedhere
\end{equation*}
\end{proof}

As usual with dualizable categories, we obtain the $\omega_1$-accessibility statement.

\begin{prop}\label{prop:comp_ass_t_structure_omega_1_accessible}
	Let $(\cC_{\geq 0},\cC_{\leq 0})$ be a compactly assembled $t$-structure on a dualizable category $\cC.$ Then this $t$-structure is $\omega_1$-accessible.
\end{prop}

\begin{proof}
	Let $x\in\cC^{\omega_1},$ then $\hat{\cY}(x)\cong\inddlim[n\in\N]x_n.$ We obtain $\hat{\cY}(\tau_{\geq 0}x)\cong \inddlim[n]\tau_{\geq 0}x_n,$ hence $\tau_{\geq 0}x\in\cC^{\omega_1}.$ This proves that the $t$-structure on $\cC$ induces a $t$-structure on $\cC^{\omega_1}.$ Recall that $\cC$ is $\omega_1$-presentable by \cite[Corollary 1.21]{E24}. Since the $t$-structure on $\cC$ commutes with filtered colimits, it follows that this $t$-structure is $\omega_1$-accessible.
\end{proof}

Next, we have an interpretation via retracts.

\begin{prop}\label{prop:comp_ass_t_structure_via_retracts}
Let $\cC$ be a dualizable category with a $t$-structure $(\cC_{\geq 0},\cC_{\leq 0})$ compatible with filtered colimits. The following are equivalent.
\begin{enumerate}[label=(\roman*),ref=(\roman*)]
	\item The $t$-structure is compactly assembled. \label{comp_ass_t_structure}
	\item There exists a small stable category $\cT$ with a $t$-structure $(\cT_{\geq 0},\cT_{\leq 0})$ and a retraction $\cC\xto{F}\Ind(\cT)\xto{G}\cC,$ such that both $F$ and $G$ are colimit-preserving and $t$-exact. \label{retract_with_t_exactness}
	\item There exists a dualizable category $\cC'$ with a compactly assembled $t$-structure $(\cC'_{\geq 0},\cC'_{\leq 0})$ and a retraction $\cC\xto{F}\cC'\xto{G}\cC,$ such that both $F$ and $G$ are colimit-preserving and left $t$-exact. \label{retract_with_left_t_exactness}
\end{enumerate}
Moreover, similar assertions are equivalent for compactly assembled continuously bounded $t$-structures, where in \ref{retract_with_t_exactness} the $t$-structure $(\cT_{\geq 0},\cT_{\leq 0})$ is required to be bounded.
\end{prop}

\begin{proof}
\Implies{comp_ass_t_structure}{retract_with_t_exactness}. By Proposition \ref{prop:comp_ass_t_structure_omega_1_accessible} it suffices to take $\cT=\cC^{\omega_1},$ $F=\hat{\cY},$ $G=\colim.$ If the $t$-structure $(\cC_{\geq 0},\cC_{\leq 0})$ is continuously bounded, then we instead take $\cT=(\cC^{\omega_1})^b.$ Note that the functor $\hat{\cY}$ indeed sends $\cC$ to $\Ind((\cC^{\omega_1})^b)$ since it sends $\cC^{\heartsuit}$ to $\Ind((\cC^{\omega_1})^{\heartsuit}).$ 

\Implies{retract_with_t_exactness}{retract_with_left_t_exactness} is trivial.

\Implies{retract_with_left_t_exactness}{comp_ass_t_structure}. By Remark \ref{rem:comp_ass_t_structures} the functor $\hat{\cY}_{\cC}$ is automatically right $t$-exact, so we only need to show the left $t$-exactness. This follows from the retraction, since $\hat{\cY}_{\cC}$ is identified with the composition of left $t$-exact functors 
\begin{equation*}
\cC\xto{F}\cC'\xto{\hat{\cY}_{\cC"}}\Ind(\cC')\xto{\Ind(G)}\Ind(\cC).\qedhere
\end{equation*}

Suppose that the $t$-structure on $\cC'$ is continuously bounded. We need to show that the $t$-structure on $\cC$ is also continuously bounded. By Remark \ref{rem:comp_ass_t_structures} it suffices to show the equalities $\cC=\la\cC^-\ra=\la\cC^+\ra.$ First, $\cC$ is generated by $G(\cC'_{\leq 0})\subset\cC_{\leq 0},$ hence $\cC=\la\cC^+\ra.$ It remains to show that $\cC=\la\cC^-\ra.$ Take $x\in\cC,$ and put $y=\Cone(\indlim[n\in\N]\tau_{\geq -n}x\to x).$ Then $y\in\bigcap\limits_{n}\cC_{\leq -n},$ hence $F(y)\in\bigcap\limits_n \cC'_{\leq -n}=0.$ We get $y\cong G(F(y))=0,$ so $x\in\la\cC^-\ra.$ This proves the proposition.  
\end{proof}

\begin{cor}\label{cor:induced_t_structure_is_compactly_assembled}
Let $\cC$ be a dualizable category with a compactly assembled $t$-structure $(\cC_{\geq 0},\cC_{\leq 0}).$ Let $\cD\subset\cC$ be a dualizable subcategory such that the inclusion functor $i:\cD\to\cC$ is strongly continuous. Suppose that the $t$-structure on $\cC$ induces a $t$-structure $(\cD_{\geq 0},\cD_{\leq 0})$ on $\cD.$ Then the latter $t$-structure is also compactly assembled.

If moreover the $t$-structure $(\cC_{\geq 0},\cC_{\leq 0})$ is continuously bounded, then so is the $t$-structure $(\cD_{\geq 0},\cD_{\leq 0}).$
\end{cor}

\begin{proof}
Denote by $i^R:\cC\to\cD$ the (continuous) right adjoint to $i.$ Then $i^R$ is left $t$-exact and $i$ is $t$-exact. Hence, the retraction $\cD\xto{i}\cC\xto{i^R}\cD$ implies that $(\cD_{\geq 0},\cD_{\leq 0})$ is compactly assembled by Proposition \ref{prop:comp_ass_t_structure_via_retracts}. The ``moreover'' assertion also follows from loc. cit.
\end{proof}

It is useful to keep in mind a reformulation in terms of compact morphisms (although we will not use it).

\begin{prop}\label{prop:comp_ass_t_structures_via_compact_morphisms}
Let $\cC$ be a dualizable category, and let $(\cC_{\geq 0},\cC_{\leq 0})$ be an accessible $t$-structure compatible with filtered colimits. The following are equivalent.
\begin{enumerate}[label=(\roman*),ref=(\roman*)]
	\item The $t$-structure $(\cC_{\geq 0},\cC_{\leq 0})$ is compactly assembled. \label{t_structure_comp_ass}
	\item For any compact morphism $f:x\to y$ in $\cC,$ the morphism $\tau_{\geq 0}f:\tau_{\geq 0}x\to\tau_{\geq 0}y$ is also compact in $\cC.$ \label{truncation_pf_compact_morphisms}
\end{enumerate} 
\end{prop}

\begin{proof}
This is formal: the $t$-structure is compactly assembled if and only if the endofunctor $\tau_{\geq 0}:\cC\to\cC$ is strongly continuous. For completeness we give the details.
	
\Implies{t_structure_comp_ass}{truncation_pf_compact_morphisms}. Compactness of $f:x\to y$ means that $\cY(f):\cY(x)\to\cY(y)$ factors through $\hat{\cY}(y)$ in $\Ind(\cC).$ Since $\hat{\cY}(\tau_{\geq 0}y)\cong \tau_{\geq 0}\hat{\cY}(y),$ we conclude that $\tau_{\geq 0}f$ is also a compact morphism.

\Implies{truncation_pf_compact_morphisms}{t_structure_comp_ass}. Consider $\tau_{\geq 0}$ as a continuous endofunctor of $\cC,$ and similarly for $\Ind(\cC).$ We need to show that the morphism $\hat{\cY}\circ\tau_{\geq 0}\to\tau_{\geq 0}\circ\hat{\cY}$ is an isomorphism of functors from $\cC$ to $\Ind(\cC).$ Let $x\in\cC$ and $\hat{\cY}(x)=\inddlim[i\in I]x_i,$ where $I$ is directed. Then for each $i\in I$ there exists $j\geq i$ such that the transition morphism $x_i\to x_j$ is compact, hence so is $\tau_{\geq 0}x_i\to\tau_{\geq 0}x_j.$ This proves that the map $\hat{\cY}(\tau_{\geq 0}x)\to\inddlim[i]\tau_{\geq 0}x_i$ is an isomorphism, as required.
\end{proof}

The following basic properties of the categories $\cC_{[a,b]}$ are almost immediate.

\begin{prop}\label{prop:C_a_b_for_comp_ass_t_structure}
Let $\cC$ be a dualizable category with a compactly assembled $t$-structure $(\cC_{\geq 0},\cC_{\leq 0}).$
\begin{enumerate}[label=(\roman*),ref=(\roman*)]
	\item Let $a\leq b,$ where $a\in\Z\cup\{-\infty\},$ $b\in\Z\cup\{+\infty\}.$ The exact category $\cC_{[a,b]}$ (with the induced exact structure from $\cC$) is coherently assembled and the inclusion functor $\cC_{[a,b]}\to\cC$ is strongly continuous. In particular the abelian category $\cC^{\heartsuit}$ is coherently assembled. \label{C_a_b_coherently_assembled}
	\item The categories $\la\cC^-\ra,$ $\la\cC^+\ra$ and $\la\cC^b\ra$ are dualizable, and their inclusion functors into $\cC$ are strongly continuous. The induced $t$-structures are compactly assembled. \label{C^-_C^+_C^b_dualizable}
\end{enumerate}
\end{prop}

\begin{proof}
\ref{C_a_b_coherently_assembled} By Proposition \ref{prop:comp_ass_t_structure_omega_1_accessible} we have the induced $t$-structure on $\cC^{\omega_1}.$ It suffices to prove that the functor $\hat{\cY}:\cC\to\Ind(\cC^{\omega_1})$ takes $\cC_{[a,b]}$ to (the essential image of) $\Ind((\cC^{\omega_1})_{[a,b]}).$ Let $x\in\C_{[a,b]}$ and let $\hat{\cY}(x)=\inddlim[i]x_i\in\Ind(\cC^{\omega_1}).$ By Corollary \ref{cor:Ind_of_T_a_b} we have an isomorphism $\inddlim[i]x_i\cong\inddlim[i]\tau_{[a,b]}x_i,$ as required.

\ref{C^-_C^+_C^b_dualizable} By \ref{C_a_b_coherently_assembled} the functor $\hat{\cY}_{\cC}$ takes $\cC_{\geq 0}$ to $\Ind(\cC_{\geq 0}),$ hence it takes $\la\cC^-\ra$ to $\Ind(\cC^-)\subset\Ind(\la \cC^-\ra).$ This proves that $\la \cC^-\ra$ is dualizable and its inclusion functor into $\cC$ is strongly continuous. Similar arguments show the same for $\la \cC^+\ra$ and $\la \cC^b\ra$ (replace $\cC_{\geq 0}$ with $\cC_{\leq 0}$ resp. $\cC^{\heartsuit}$).
 \end{proof}

From now on, in the situation of Proposition \ref{prop:C_a_b_for_comp_ass_t_structure} \ref{C_a_b_coherently_assembled} we write $\cC^{\omega_1}_{[a,b]}$ for $(\cC^{\omega_1})_{[a,b]}\simeq (\cC_{[a,b]})^{\omega_1}.$ Also, for a morphism $f:x\to y$ in $\cC_{[a,b]}$ the compactness of $f$ in $\cC$ is equivalent to its compactness in $\cC_{[a,b]},$ which we will tacitly use in the proofs. Note that for $a'\leq a\leq b\leq b'$ the inclusion functor $\cC_{[a,b]}\to\cC_{[a',b']}$ is exact and strongly continuous.

We formulate the following statement only for dualizable $t$-categories, but it holds in a much more general setting, see Remark \ref{rem:colim_of_St_of_C_0_n_for_presentable} below.

\begin{prop}\label{prop:colimit_of_St_of_C_0_n}
Let $\cC$ be a dualizable $t$-category (Definition \ref{def:compactly_assembled_t_structure}). Then we have an equivalence
\begin{equation}\label{eq:colim_of_St_of_C_0_n}
\indlim[n]\check{\St}(\cC_{[0,n]})\xto{\sim}\cC.
\end{equation}
\end{prop}

\begin{proof}
By Proposition \ref{prop:comp_ass_t_structure_via_retracts} we have a retraction $\cC\xto{F}\cC'\xto{G}\cC,$ where $\cC'\simeq\Ind(\cT),$ $\cT$ is equipped with a bounded $t$-structure, and the functors $F$ and $G$ are continuous, exact and $t$-exact. Then the functor $\indlim[n]\check{\St}(\cC_{[0,n]})\to\cC$ is a retract of the functor $\indlim[n]\check{\St}(\cC'_{[0,n]})\to\cC'$ (in $\Pr^L_{\st}$), hence we may and will assume that $\cC=\cC'.$

In this case by Proposition \ref{prop:St_of_locally_coherent} we have $\check{\St}(\cC_{[0,n]})\simeq\Ind(\St(\cT_{[0,n]})).$ Hence, it suffices to consider the corresponding colimit in $\Cat^{\perf}.$ We have
\begin{equation*}
\indlim[n]\St(\cT_{[0,n]})\xto{\sim}\St(\indlim[n]\cT_{[0,n]})\simeq \St(\cT_{\geq 0})\simeq\cT.
\end{equation*}
Taking the ind-completions, we obtain the equivalence \eqref{eq:colim_of_St_of_C_0_n}.
\end{proof}

\begin{remark}\label{rem:colim_of_St_of_C_0_n_for_presentable}
The statement of Proposition \ref{prop:colimit_of_St_of_C_0_n} holds more generally when $\cC$ is presentable stable, the accessible $t$-structure $(\cC_{\geq 0},\cC_{\geq 0})$ is right complete and compatible with filtered colimits, and the Grothendieck prestable category $\cC_{\geq 0}$ is anticomplete.
\end{remark}  

The following two important statements are special cases of much more general results on Grothendieck abelian $(m+1)$-categories, but our proofs use only the general machinery of presentable stable envelopes of coherently assembled exact categories.

\begin{prop}\label{prop:properties_of_functor_St_of_C_0_m_to_C}
	Let $\cC$ be a dualizable $t$-category, and let $m\geq 0$ be an integer. Denote by $j:\cC_{[0,m]}\to\check{\St}(\cC_{[0,m]})$ the universal functor. Then for $x,y\in\cC^{\heartsuit}$ the map
	\begin{equation*}
		\Ext^n_{\check{\St}(\cC_{[0,m]})}(j(x),j(y))\to\Ext^n_{\cC}(x,y)
	\end{equation*} 
	is an isomorphism for $n\leq m+1,$ and a monomorphism for $n=m+2.$ 
\end{prop}

\begin{proof}
We have isomorphisms
\begin{equation*}
\Ext^n_{\check{\St}(\cC_{[0,m]})}(j(x),j(y))\cong \Ext^{n-m}_{\check{\St}(\cC_{[0,m]})}(j(x),j(y[m])),\quad n\in\Z.
\end{equation*}
Applying Corollary \ref{cor:mono_for_Ext^2}, we see that the map
\begin{equation*}
	\Ext^{n-m}_{\check{\St}(\cC_{[0,m]})}(j(x),j(y[m]))\to \Ext^n_{\cC}(x,y)
\end{equation*}
is an isomorphism for $n\leq m+1,$ and a monomorphism for $n=m+2,$ as required.
\end{proof}

Recall that in the situation of Proposition \ref{prop:colimit_of_St_of_C_0_n} we have $\check{\St}(\cC_{[0,0]})\simeq \check{D}(\cC^{\heartsuit})$ -- the unseparated derived category of the heart.
% We are interested in a situation when a certain approximation of the realization functor is an equivalence. More precisely, we have the following elementary proposition, which is again a very special case of a much more general statement.

\begin{prop}\label{prop:when_partial_realization_is_an_equivalence}
Let $\cC$ be a dualizable $t$-category and let $m\geq 0$ be an integer. Consider the (strongly continuous exact) functor $\Phi:\check{\cD}(\cC^{\heartsuit})\to \check{\St}(\cC_{[0,m]})$ induced by the inclusion $\cC^{\heartsuit}\hto\cC_{[0,m]}.$ The following are equivalent.
\begin{enumerate}[label=(\roman*),ref=(\roman*)]
	\item $\Phi$ is an equivalence. \label{partial_realization_equivalence}
	\item For $x,y\in\cC^{\heartsuit},$ the map $\Ext^n_{\cC^{\heartsuit}}(x,y)\to \Ext^n_{\cC}(x,y)$ is an isomorphism for $n\leq m+1,$ and a monomorphism for $n=m+2.$ \label{-m-2_connectivity_for_realization_functor}
	\item For $x,y\in\cC^{\omega_1,\heartsuit},$ the map $\Ext^n_{\cC^{\heartsuit}}(x,y)\to \Ext^n_{\cC}(x,y)$ is an isomorphism for $n\leq m+1.$ \label{isomorphisms_on_Ext_leq_m+1}
\end{enumerate}
\end{prop}

\begin{proof}
The implication \Implies{partial_realization_equivalence}{-m-2_connectivity_for_realization_functor} follows directly from Proposition \ref{prop:properties_of_functor_St_of_C_0_m_to_C}.

The implication \Implies{-m-2_connectivity_for_realization_functor}{isomorphisms_on_Ext_leq_m+1} is trivial.

\Implies{isomorphisms_on_Ext_leq_m+1}{partial_realization_equivalence}. Clearly, the essential image of $\Phi$ generates $\check{\St}(\cC_{[0,m]})$ as a localizing subcategory. Hence, we only need to show that $\Phi$ is fully faithful. By Proposition \ref{prop:left_special_coherently_assembled}, it suffices to show that for $x\in \cC^{\omega_1}_{[0,m]}$ there exists $y\in\cC^{\omega_1,\heartsuit}$ and a map $f:y\to x$ such that the map $\pi_0(f):y\to\pi_0(x)$ is an epimorphism. If $m=0,$ then there is nothing to prove. If $m>0,$ consider the map $g:\pi_0(x)\to\tau_{\geq 1}x[1].$ It follows from \ref{isomorphisms_on_Ext_leq_m+1} (applied to $\Ext^2$) that we can find an epimorphism $f_1:y_1\to\pi_0(x)$ in $\cC^{\omega_1,\heartsuit}$ such that the composition $y_1\xto{f_1}\pi_0(x)\xto{g}\tau_{\geq 1}x[1]\to\pi_1(x)[2]$ is zero. Then $g\circ f_1$ factors through $\tau_{\geq 2}x[1],$ so we obtain a map $g_1:y_1\to\tau_{\geq 2}x[1].$ Continuing the process, we find an epimorphism $f_m:y_m\to\pi_0(x)$ in $\cC^{\omega_1,\heartsuit}$ such that the composition $y_m\xto{f_m}\pi_0(x)\xto{g}\tau_{\geq 1}x[1]$ is zero. This means that $f_m$ factors through $x,$ which proves the implication.
\end{proof}

We conclude this subsection with a $K$-theoretic application of the above results.

\begin{cor}\label{cor:cartesian_square_C^-_C^+_C^b}
Let $\cC$ be a dualizable category with a compactly assembled $t$-structure $(\cC_{\geq 0},\cC_{\leq 0}).$ Then the following square of spectra is cartesian:
\begin{equation}\label{eq:cartesian_K_theory}
	\begin{tikzcd}
		K^{\cont}(\la \cC^b\ra) \ar[r]\ar[d] & K^{\cont}(\la \cC^-\ra)\ar[d]\\
		K^{\cont}(\la \cC^+\ra) \ar[r] & K^{\cont}(\cC).
	\end{tikzcd}
\end{equation}
\end{cor}

\begin{proof}
By Proposition \ref{prop:C_a_b_for_comp_ass_t_structure} this square is well-defined: the four categories are dualizable and the inclusion functors are strongly continuous. It suffices to prove that the map between the cofibers of horizontal arrows is an equivalence. This follows directly from Proposition \ref{prop:SOD_for_C_mod_C^b}: we have
\begin{equation*}
K^{\cont}(\la\cC^-\ra/\la\cC^b\ra)\xto{\sim} K^{\cont}(\cC/\la\cC^+\ra).\qedhere
\end{equation*}
\end{proof}

\subsection{General constructions of $t$-structures}
\label{ssec:construction_of_t_structures}

We will need a general efficient method for constructing a $t$-structure. First we deal with small categories, for which the construction seems to be a folklore knowledge. In the following proposition we do not a priori assume the idempotent-completeness, but it holds a posteriori.

\begin{prop}\label{prop:t_structure_on_small_via_functor_from_abelian}
Let $\cC$ be a small stable category, and let $\cA$ be a small abelian category. Suppose that $F:\cA\to\cC$ is a fully faithful exact functor such that $F(\cA)$ generates $\cC$ as a stable subcategory. Then $\cC$ has a unique bounded $t$-structure such that $F(\cA)\subset \cC^{\heartsuit}.$ Each object of $\cC^{\heartsuit}$ has a finite filtration with subquotients in $F(\cA).$

Moreover, if $F$ induces isomorphisms $\Ext^1_{\cA}(x,y)\xto{\sim}\Ext^1_{\cC}(F(x),F(y))$ for $x,y\in\cA,$ then we have $\cC^{\heartsuit}\simeq\cA.$
\end{prop}

\begin{proof}
We will construct the $t$-structure with required properties, and its uniqueness follows from the proof.
	
We define $\cC_{\geq 0}\subset\cC$ resp. $\cC_{\leq 0}\subset\cC$ to be the full subcategory generated via extensions by $F(\cA)[n],$ where $n\geq 0$ resp. $n\leq 0.$ Clearly, $\cC_{\geq 0}[1]\subset\cC_{\geq 0}$ $\cC_{\leq 0}[-1]\subset\cC_{\leq 0}$ and $F(\cA)\subset\cC_{\geq 0}\cap \cC_{\leq 0}.$ Also, the orthogonality holds since $\Ext_{\cC}^{<0}(F(x),F(y))=0$ for $x,y\in\cA$ by assumption. 

As usual we denote by $[m]$ the poset $\{0,1\dots,m\}$ with the usual order. We only need to prove that for any $x\in\cC$ there exists some $m\geq 0$ and a functor $[m]\mapsto \cC,$ $i\mapsto x_i,$ such that $x_0=0,$ $x_m\cong x$ and we have $\Cone(x_{i-1}\to x_i)\in F(\cA)[k_i]$ for $1\leq i\leq m,$ where $k_1\geq k_2\geq\dots\geq k_m.$ If this is the case for some object $x,$ we say that $x$ has a good filtration. Since $F(\cA)$ generates $\cC$ as a stable subcategory, we only need to show that for any object $x\in \cC$ with a good filtration, for any $y\in\cA,$ for any $l\in\Z$ and for any morphism $f:F(y)[l]\to x,$ the object $\Cone(f)$ has a good filtration.

Since the class of objects with a good filtration is closed under shifts, we may and will assume that $l=0.$ Let $[m]\to\cC,$ $x_i$ and $k_i$ be as above, with $x_m\cong x$ etc. If $m=0,$ then $x=0$ and there is nothing to prove. So we assume that $m>0.$ If $k_1<0,$ then $f=0$ and $\Cone(f)\cong x\oplus F(y)[1]$ has a good filtration.. Hence, we assume that $k_1\geq 0.$ We take the largest $p$ such that $k_p\geq 0,$ then we get a unique factorization $f:F(y)\xto{g}x_p\to x.$ It suffices to prove that $\Cone(g)$ has a good filtration with subquotients in $F(\cA)[\geq 0].$ We prove this by induction on $s=|\{i: k_i=0\}|.$

If $s=0,$ then $k_p\geq 1,$ hence the cofiber sequence $x_p\to \Cone(g)\to F(y)[1]$ shows that $\Cone(g)$ has a good filtration with subquotients in $F(\cA)[\geq 1].$ Now suppose that $s>0$ and the assertion holds for smaller values of $s.$ We have $k_p=0,$ so there is some $z\in \cA$ and an isomorphism $\Cone(x_{p-1}\to x_p)\cong F(z).$ Denote by $\varphi:y\to z$ the morphism in $\cA$ such that the composition $F(y)\xto{g}x_p\to F(z)$ is homotopic to $F(\varphi).$ First suppose that $\varphi$ is not an epimorphism. Consider the composition $\psi:x_p\to F(z)\to F(\coker(\varphi)),$ and put $x_p'=\Fiber(\psi).$ Then we have a unique factorization $g:F(y)\xto{g'} x_p'\to x_p,$ and it suffices to prove that $\Cone(g')$ has a good filtration with subquotients in $F(\cA)[\geq 0].$ Hence, we may and will assume that $\varphi:y\to z$ is an epimorphism. Put $w=\ker(\varphi),$ then the composition $F(w)\to F(y)\xto{g} x_p$ factors uniquely through $x_{p-1}.$ Moreover, we have $\Cone(g)\cong \Cone(F(w)\to x_{p-1}).$ By the induction hypothesis, this object has a good filtration with subquotients in $F(\cA)[\geq 0].$ This proves the induction step.   

Finally, the ``moreover'' assertion is immediate: the assumption implies that $F(\cA)$ is closed under extensions in $\cC^{\heartsuit},$ hence the functor $\cA\to\cC^{\heartsuit}$ is essentially surjective. 
\end{proof}

We deduce an analogue for dualizable categories.

\begin{prop}\label{prop:t_structure_on_dualizable_via_functor_from_abelian}
Let $\cC$ be a dualizable category, and let $\cA$ be a coherently assembled abelian category. Suppose that $F:\cA\to\cC$ is a strongly continuous fully faithful exact functor, such that $F(\cA)$ generates $\cC$ as a localizing subcategory. 
\begin{enumerate}[label=(\roman*),ref=(\roman*)]
	\item $\cC$ has a unique compactly assembled continuously bounded $t$-structure such that $F(\cA)\subset\cC^{\heartsuit}.$ Moreover, $F(\cA)$ generates $\cC^{\heartsuit}$ via extensions and filtered colimits. \label{t_structure_via_functor_from_abelian}
	\item If $F$ induces isomorphisms $\Ext^1_{\cA}(x,y)\xto{\sim}\Ext^1_{\cC}(F(x),F(y))$ for $x,y\in\cA^{\omega_1},$ then we have $\cC^{\heartsuit}\simeq\cA.$ \label{when_A_equivalent_to_heart_of_C}
\item Let $\cD$ be another dualizable $t$-category and let $G:\cC\to\cD$ be a strongly continuous exact functor. Then $G$ is $t$-exact if and only if we have $G(F(\cA^{\omega_1}))\subset\cD^{\heartsuit}.$ \label{sufficient_for_t_exactness}
\end{enumerate}
\end{prop}

\begin{proof}
\ref{t_structure_via_functor_from_abelian} Again, we construct the $t$-structure with required properties, and its uniqueness follows automatically.

Denote by $\cT\subset\cC^{\omega_1}$ the (small) stable subcategory generated by $F(\cA^{\omega_1}).$ By Proposition \ref{prop:t_structure_on_small_via_functor_from_abelian}, $\cT$ has a unique bounded $t$-structure such that $F(\cA^{\omega_1})\subset \cT^{\heartsuit}.$ The strong continuity of $F$ implies that $\hat{\cY}_{\cC}(F(\cA))\subset \Ind(\cT).$ Since $F(\cA)$ generates $\cC$ as a localizing subcategory, it follows that $\hat{\cY}_{\cC}(\cC)\subset\Ind(\cT).$ 

We claim that the $t$-structure on $\Ind(\cT)$ induces a $t$-structure on $\hat{\cY}_{\cC}(\cC).$ To see this, it suffices to prove that the composition $\Phi:\Ind(\cT)\xto{\colim}\cC\xto{\hat{\cY}_{\cC}}\Ind(\cT)$ is $t$-exact. Since the $t$-structure on $\cT$ is bounded, we only need to show the inclusion $\Phi(\cT^{\heartsuit})\subset \Ind(\cT^{\heartsuit}).$ By Proposition \ref{prop:t_structure_on_small_via_functor_from_abelian}, each object of $\cT^{\heartsuit}$ has a finite filtration with subquotients in $F(\cA^{\omega_1}).$ Hence, it suffices to show that $\Phi(F(\cA^{\omega_1}))\subset \Ind(F(\cA^{\omega_1})).$ The latter inclusion follows from the strong continuity of $F.$

We claim that the constructed $t$-structure on $\cC$ satisfies the desired properties. Since the $t$-structure on $\Ind(\cT)$ is compactly assembled and continuously bounded, so is the $t$-structure on $\cC$ by Corollary \ref{cor:induced_t_structure_is_compactly_assembled}. Since $F(\cA^{\omega_1})$ generates $\cT^{\heartsuit}$ via extensions (by Proposition \ref{prop:t_structure_on_small_via_functor_from_abelian}), the retraction $\cC^{\heartsuit}\to\Ind(\cT^{\heartsuit})\to\cC^{\heartsuit}$ implies that $\cC^{\heartsuit}$ is generated by $F(\cA)$ via extensions and filtered colimits.

\ref{when_A_equivalent_to_heart_of_C} follows from the proof of \ref{t_structure_via_functor_from_abelian}: we have $\cT^{\heartsuit}\simeq\cA^{\omega_1}$ by Proposition \ref{prop:t_structure_on_small_via_functor_from_abelian}. Hence, $\cC^{\heartsuit}$ is generated by $F(\cA^{\omega_1})$ via filtered colimits, which shows that the functor $\cA\to\cC^{\heartsuit}$ is essentially surjective.

\ref{sufficient_for_t_exactness} The ``only if'' direction is clear. For the ``if'' direction we first note that $G(\cC^{\heartsuit})\subset\cD^{\heartsuit}$ since $\cC^{\heartsuit}$ is generated by $F(\cA^{\omega_1})$ via extensions and filtered colimits. Next, we have $F(\cC_{\geq 0})\subset\cD_{\geq 0}$ since $\cC^{\heartsuit}$ generates $\cC_{\geq 0}$ via extensions and colimits. Finally, we have $F(\cC_{\leq 0})\subset\cD_{\leq 0}$ since $\cC_{\leq 0}$ is generated by $\cC^{\heartsuit}$ via finite limits, extensions and filtered colimits (by the right completeness).
\end{proof}

\begin{cor}\label{cor:t_structure_on_check_St_of_C_0_m}
Let $\cC$ be dualizable $t$-category and let $m\geq 0$ be an integer. There is a unique compactly assembled continuously bounded $t$-structure on $\check{\St}(\cC_{[0,m]})$ such that the functor $\check{D}(\cC^{\heartsuit})\to \check{\St}(\cC_{[0,m]})$ is $t$-exact. Moreover, we have $\check{\St}(\cC_{[0,m]})^{\heartsuit}\simeq \cC^{\heartsuit}.$
\end{cor}

\begin{proof}
This follows directly from Proposition \ref{prop:t_structure_on_dualizable_via_functor_from_abelian}, applied to the functor $\cC^{\heartsuit}\to\cC_{[0,m]}\to\check{\St}(\cC_{[0,m]}).$
\end{proof}

Let $\cC$ be a dualizable category and let $A$ be a continuous exact monad on $\cC,$ i.e. $A\in\Alg_{\bE_1}(\Fun^L(\cC,\cC)).$ Here we consider the category $\Fun^L(\cC,\cC)$ as a presentable stable $\bE_1$-monoidal category (i.e. an $\bE_1$-algebra in $\Pr^L_{\st}$). Then we can consider the dualizable category $\Mod_A(\cC)$ of $A$-modules in $\cC.$ The forgetful functor $\Mod_A(\cC)\to\cC$ has a left adjoint $x\mapsto A(x),$ which is automatically strongly continuous. By \cite[Proposition C.1]{E24} this gives a fully faithful functor
\begin{equation}\label{eq:from_monads_to_undercategory}
	\Alg_{\bE_1}(\Fun^L(\cC,\cC))\hto (\Cat_{\st}^{\dual})_{\cC/}.
\end{equation} 
Its essential image consists of pairs $(\cD,\Phi:\cC\to\cD)$ such that $\Phi(\cC)$ generates $\cD$ as a localizing subcategory. For such $(\cD,\Phi)$ the functor $\Phi^R$ is monadic, so we have an equivalence
\begin{equation*}
	\cD\simeq \Mod_{\Phi^R\circ\Phi}(\cC).\end{equation*} More generally, the right adjoint to \eqref{eq:from_monads_to_undercategory} is given by
\begin{equation*}
	(\cD,\Phi)\mapsto \Phi^R\circ \Phi.
\end{equation*}

We now apply this point of view to dualizable $t$-categories.

\begin{cor}\label{cor:t_structures_on_modules}
	Let $\cC$ be a dualizable $t$-category. Let $A\in\Alg_{\bE_1}(\Fun^L(\cC,\cC))$ be a continuous exact monad. Suppose that the functor $\Cone(\Id_{\cC}\to A)[1]$ is left $t$-exact. Then the category $\Mod_A(\cC)$ has a unique compactly assembled continuously bounded $t$-structure such that the functor $\cC\to \Mod_A(\cC)$ is $t$-exact. 
\end{cor}

\begin{proof}
By adjunction, the left $t$-exactness of the functor $\Cone(\Id_{\cC}\to A)[1]$ implies that the composition $\cC^{\heartsuit}\to \cC\to\Mod_A(\cC)$ is fully faithful. Since its image generates the target as a localizing subcategory, the assertion of the corollary is a special case of Proposition \ref{prop:t_structure_on_dualizable_via_functor_from_abelian}.
\end{proof}

We also deduce the following basic result.

\begin{prop}\label{prop:t_structure_on_nilpotent_endomorphisms}
	Let $\cC$ be a dualizable $t$-category. Let $y$ be a formal variable of degree $0,$ and consider the (compactly generated) category $\Mod_{y\hy\tors}\hy \bS[y]$ of $\bS[y]$-modules such that $y$ acts locally nilpotently on the homotopy groups. Denote by $i:\Sp\to \Mod_{y\hy\tors}\hy \bS[y]$ the restriction of scalars functor for $\bS[y]\to\bS,$ $y\mapsto 0.$ 
	\begin{enumerate}[label=(\roman*),ref=(\roman*)]
		\item The category $(\Mod_{y\hy\tors}\hy \bS[y])\otimes \cC$ has a unique compactly assembled continuously bounded $t$-structure such that the functor $i\otimes\cC:\cC\to (\Mod_{y\hy\tors}\hy \bS[y])\otimes \cC$ is $t$-exact. Moreover, the abelian category $((\Mod_{y\hy\tors}\hy \bS[y])\otimes \cC)^{\heartsuit}$ is identified with the category of pairs $(X,f),$ where $X\in\cC^{\heartsuit}$ and $f:X\to X$ is a locally nilpotent endomorphism. \label{t_structure_for_nilpotent_endomorphisms}
		\item If we have an equivalence $\check{D}(\cC^{\heartsuit})\xto{\sim}\cC,$ then we also have an equivalence
		\begin{equation*}
		\check{D}(((\Mod_{y\hy\tors}\hy \bS[y])\otimes \cC)^{\heartsuit})\xto{\sim} (\Mod_{y\hy\tors}\hy \bS[y])\otimes \cC.
		\end{equation*} \label{realization_is_equivalence_for_nil_endomorphisms}
	\end{enumerate}
\end{prop}

\begin{proof}
	\ref{t_structure_for_nilpotent_endomorphisms} It is well-known that the category $\Mod_{y\hy\tors}\hy \bS[y]$ is generated by the single compact object $i(\bS)\cong \Cone(\bS[y]\xto{y}\bS[y]),$ in particular the functor $i$ is strongly continuous. Thus, the functor $i\otimes \cC:\cC\to (\Mod_{y\hy\tors}\bS[y])\otimes\cC$ is also strongly continuous and its image generates the target as a localizing subcategory. 
	
	We have $\Cone(\bS\to i^R(i(\bS)))\cong \bS[-1],$ in particular the functor $\Cone(\Id_{\cC}\to (i\otimes\cC)^R\circ (i\otimes\cC))[1]\cong \Id_{\cC}$ is left $t$-exact. Applying Corollary \ref{cor:t_structures_on_modules}, we obtain the stated $t$-structure and its uniqueness. 
	
	By construction the forgetful functor $(\Mod_{y\hy\tors}\bS[y])\otimes\cC\to\cC$ is $t$-exact and conservative. This gives the description of the heart and proves
	\ref{t_structure_for_nilpotent_endomorphisms}.
	
	\ref{realization_is_equivalence_for_nil_endomorphisms} By Proposition \ref{prop:coherently_assembled_as_retract} and its proof, we may and will assume that $\cC^{\heartsuit}$ is locally coherent, i.e. $\cC$ is compactly generated. Put $\cA=(\cC^{\heartsuit})^{\omega},$ and let $\Nil(\cA)$ be the category of pairs $(X,f),$ where $X\in\cA$ and $f:X\to X$ is a nilpotent endomorphism. We need to prove that the realization functor
	\begin{equation}\label{eq:realization_equiv_for_nil_endomorphisms}
	D^b(\Nil(\cA))\to \Perf_{y\hy\tors}(\bS[y])\otimes D^b(\cA)=\cD
	\end{equation}
	is an equivalence. By \cite[Lemma 2.1]{Nee21} it suffices to show that for $M\in \cD_{\geq 0}$ there exists $N\in\cD^{\heartsuit}$ and a map $f:N\to M$ such that $\Cone(f)\in\cD_{\geq 1}.$ Denote by $\Phi:\cD\to D^b(\cA)$ the forgetful functor. Then there exists some $N'\in\cA$ and a map $g:N'\to\Phi(M)$ such that $\Cone(g)\in D^b(\cA)_{\geq 1}.$ By adjunction $g$ corresponds to a map $g':N'[y]\to M,$ where we consider $N'[y]$ as an object of $\Perf(\bS[y])\otimes D^b(\cA).$ Choose $n>0$ such that $y^n$ is acting by zero on $M.$ Then $g'$ factors through $N=\Cone(N'[y]\xto{y^n}N'[y])$ (non-uniquely), which gives a map $f:N\to M$ such that $\pi_0(f):N\to\pi_0(M)$ is an epimorphism, i.e. $\Cone(f)\in\cD_{\geq 1}.$ Since $N\in\cD^{\heartsuit},$ this shows that \eqref{eq:realization_equiv_for_nil_endomorphisms} is an equivalence and proves \ref{realization_is_equivalence_for_nil_endomorphisms}.
\end{proof}

\subsection{Sufficient conditions for left $t$-exactness}

In this subsection we consider dualizable $t$-categories, and we record some semi-trivial observations on how to check the left $t$-exactness of a functor between them. The statements of course hold in more general situations.

\begin{prop}\label{prop:left_t_exactness_suffices_for_the_heart}
Let $F:\cC\to\cD$ be a continuous exact functor between dualizable $t$-categories. The following are equivalent.
\begin{enumerate}[label=(\roman*),ref=(\roman*)]
	\item $F$ is left $t$-exact.
	\item We have $F(\cC^{\omega_1,\heartsuit})\subset \cD_{\leq 0}.$
\end{enumerate}
\end{prop}

\begin{proof}
The $t$-structure on $\cC$ is right complete and $\omega_1$-accessible, hence $\cC_{\leq 0}$ is generated by $\cC^{\omega_1,\heartsuit}$ via finite limits, extensions and filtered colimits. This implies the equivalence.
\end{proof}

\begin{prop}\label{prop:left_t_exactness_suffices_for_composition}
Let $\cC$ and $\cD$ be dualizable $t$-categories, and let $F:\cC\to\cD$ be a strongly continuous, exact and $t$-exact functor, such that $F(\cC)$ generates $\cD$ as a localizing subcategory and moreover the restriction $F_{\mid \cC^{\heartsuit}}:\cC^{\heartsuit}\to\cD^{\heartsuit}$ is fully faithful. For a continuous exact endofunctor $G:\cD\to\cD$ the following are equivalent.
\begin{enumerate}[label=(\roman*),ref=(\roman*)]
	\item $G$ is left $t$-exact. \label{G_left_t_exact}
	\item The composition $F^R\circ G\circ F$ is left $t$-exact. \label{composition_left_t_exact}
\end{enumerate}
\end{prop}

\begin{proof}
The implication \Implies{G_left_t_exact}{composition_left_t_exact} is trivial since both $F$ and $F^R$ are left $t$-exact.

\Implies{composition_left_t_exact}{G_left_t_exact}. We first show that $F^R\circ G:\cD\to\cC$ is left $t$-exact. By Proposition \ref{prop:t_structure_on_dualizable_via_functor_from_abelian}, $F(\cC^{\heartsuit})$ generates $\cD^{\heartsuit}$ via extensions and filtered colimits. Hence, $F^R(G(\cD^{\heartsuit}))\subset\cC_{\leq 0}.$ Thus, by Proposition \ref{prop:left_t_exactness_suffices_for_the_heart} the functor $F^R\circ G$ is left $t$-exact.

Now we show that $G$ is left $t$-exact. Let $x\in\cD_{\leq 0}.$ By adjunction for $y\in\cC^{\heartsuit}$ we have $\Hom_{\cD}(F(y),G(x))\in\Sp_{\leq 0}.$ Again, since $F(\cC^{\heartsuit})$ generates $\cD^{\heartsuit}$ via extensions and filtered colimits, for any $z\in\cD_{\heartsuit}$ we have $\Hom_{\cD}(z,G(x))\in\Sp_{\leq 0}.$ But $\cD^{\heartsuit}$ generates $\cD_{\geq 0}$ via colimits and extensions, since $\hat{\cY}(\cD_{\geq 0})\subset\Ind(\cD^b_{\geq 0}).$ We conclude that $G(x)\in \cD_{\leq 0},$ as required.
\end{proof}

\subsection{Filtered colimits of dualizable $t$-categories}

The following result is elementary, and again it is a special case of a more general fact.

\begin{prop}\label{prop:filtered_colimit_of_dualizable_t_categories}
Let $(\cC_i)_{i\in I}$ be a directed system of dualizable $t$-categories, where we assume the transition functors $F_{ij}:\cC_i\to\cC_j$ to be strongly continuous, exact and $t$-exact. Then the colimit $\cC=\indlim[i]\cC_i$ is naturally a dualizable $t$-category and all the functors $\cC_i\to\cC$ are $t$-exact.
\end{prop}

\begin{proof}
If all $\cC_i$ are compactly generated, then the assertion follows directly from Proposition \ref{prop:comp_ass_on_compactly_generated}: we have a diagram of small $t$-categories $(\cC_i^{\omega})_{i\in I},$ and the colimit $\indlim[i]\cC_i^{\omega}$ (taken in $\Cat^{\perf}$) has a natural bounded $t$-structure such that all the functors $\cC_i^{\omega}\to\indlim[i]\cC_i^{\omega}$ are $t$-exact. We have $\cC=\Ind(\indlim[i]\cC_i^{\omega}),$ which proves the proposition in this case.

The general case follows: take the colimit $\cB=\indlim[i]\cC_i^{\omega_1,b}$ in $\Cat^{\perf},$ and let $F:\cC\to\Ind(\cB)$ be the colimit (in $\Cat_{\st}^{\dual}$) of functors $\hat{\cY}:\cC_i\to\Ind(\cC_i^{\omega_1,b}).$ Then $F$ is fully faithful and the composition $F\circ F^R:\Ind(\cB)\to\Ind(\cB)$ is $t$-exact, hence by Corollary \ref{cor:induced_t_structure_is_compactly_assembled} we obtain a compactly assembled continuously bounded $t$-structure on $\cC.$ The functors $\cC_i\to\cC$ are $t$-exact since they are compositions of $t$-exact functors $\cC_i\to\Ind(\cC_i^{\omega_1,b})\to\Ind(\cB)\xto{F^R}\cC.$  
\end{proof}

We will need the following elementary application.

\begin{prop}\label{prop:functors_from_poset_to_dualizable_t_category}
Let $\cC$ be a dualizable $t$-category. Let $S$ be a poset, and consider the $t$-structure on the (dualizable) category $\cD=\Fun(S^{op},\cC),$ given by
\begin{equation*}
\cD_{\geq 0}=\Fun(S^{op},\cC_{\geq 0}),\quad \cD_{\leq 0}=\Fun(S^{op},\cC_{\leq 0}).
\end{equation*}
Then this $t$-structure is accessible, compatible with filtered colimits and is continuously bounded. Moreover, this $t$-structure is compactly assembled in the following two cases: 

\begin{enumerate}[label=(\roman*),ref=(\roman*)]
	\item $S$ is finite. \label{finite_poset}
    \item $S$ is a meet-semi-lattice, i.e. $S$ has the smallest element and for any $i,j\in S$ there exists $i\wedge j=\inf(i,j)\in S$ (in other words, $S$ considered as a category has finite products). \label{meet_semilattice}
\end{enumerate}
\end{prop}

\begin{proof}
It is clear that the $t$-structure is well-defined, accessible and compatible with filtered colimits. Denote by $\Phi_i:\cC\to\cD$ the left adjoint to $F\mapsto F(i),$ $i\in I.$ Then the functors $\Phi_i^R,$ $i\in I,$ form a conservative family, hence $\cD$ is generated by $\Phi_i(\cC^{\heartsuit})\subset\cD^{\heartsuit},$ so the $t$-structure is continuously bounded. Now consider the cases.

\ref{finite_poset} By Proposition \ref{prop:comp_ass_t_structure_via_retracts} we may and will assume that $\cC$ is compactly generated. Then $\cD$ is also compactly generated and we have $\cD^{\omega}=\Fun(S^{op},\cC^{\omega})$ since $S$ is finite. It follows that the $t$-structure on $\cD$ induces a $t$-structure on $\cD^{\omega},$ which proves \ref{finite_poset}. 

\ref{meet_semilattice} Let $I$ be the poset of (full) finite subposets $T\subset S$ closed under finite meets (in particular, we require that $T$ contains the smallest element of $S$). Then $I$ is directed, and we have $\cD\simeq \indlim[T\in I]\Fun(T^{op},\cC),$ where the colimit is taken in $\Cat_{\st}^{\dual}$ (equivalently, in $\Pr^L_{\st}$). Here the transition functors are left Kan extensions. Now, each inclusion of posets $T\subset T'$ in $I$ has a left adjoint by assumption. Hence, the functor $\Fun(T^{op},\cC)\to \Fun(T^{'op},\cC)$ is $t$-exact. It remains to apply \ref{finite_poset} and Proposition \ref{prop:filtered_colimit_of_dualizable_t_categories}.
\end{proof}

\subsection{Dual $t$-structures}

This subsection can be skipped on the first reading, we include it for completeness. Recall that if $\cT$ is a small stable category with a $t$-structure $(\cT_{\geq 0},\cT_{\leq 0}),$ then the category $\cT^{op}$ has a natural opposite $t$-structure given by $(\cT^{op})_{\geq 0}=(\cT_{\leq 0})^{op},$ $(\cT^{op})_{\leq 0}=(\cT_{\geq 0})^{op}.$ By taking ind-objects we obtain the notion of a dual of a compactly generated $t$-structure. We explain a natural generalization for compactly assembled $t$-structures.

\begin{prop}\label{prop:dual_t_structure}
Let $\cC$ be a dualizable category with a compactly assembled $t$-structure $(\cC_{\geq 0},\cC_{\leq 0}).$ Then the category $\cC^{\vee}$ has a natural compactly assembled $t$-structure $((\cC^{\vee})_{\geq 0},(\cC^{\vee})_{\leq 0})$ which is described as follows. Consider the (strongly continuous fully faithful) functor $\Phi=(\colim)^{\vee}:\cC^{\vee}\to\Ind((\cC^{\omega_1})^{op}).$ Then 
%\begin{equation*}
%(\cC^{\vee})_{\geq 0}=\{x\in\cC^{\vee}\mid \forall y\in\cC^{\omega_1} \forall \alpha\in\pi_0\ev_{\cC}(y,x) \exists z\in \cC^{\omega_1}_{\leq 0} \exists \beta\in\pi_0\ev_{\cC}(z,x)\exists f:z\to y\text{ such that }\ev_{\cC}(f,x)(\beta)=\alpha\}
%\end{equation*}
\begin{equation*}\label{eq:C_vee_geq_0}
(\cC^{\vee})_{\geq 0}=\{x\in\cC^{\vee}\mid \Phi(x)\in \Ind((\cC^{\omega_1}_{\leq 0})^{op})\subset\Ind((\cC^{\omega_1})^{op})\},
\end{equation*}
\begin{equation*}\label{eq:C_vee_leq_0}
(\cC^{\vee})_{\leq 0}=\{x\in\cC^{\vee}\mid \forall y\in\cC^{\omega_1}_{\leq 0}\text{ we have }\ev_{\cC}(y,x)\in\Sp_{\leq 0}\}.
\end{equation*}
\end{prop}

\begin{proof}
Consider the compactly generated $t$-structure on $\Ind((\cC^{\omega_1})^{op}),$ coming from the opposite $t$-structure on $(\cC^{\omega_1})^{op}.$ It suffices to prove that  it induces the $t$-structure on $\cC^{\vee}$ via $\Phi,$ which would be automatically compactly assembled by Corollary \ref{cor:induced_t_structure_is_compactly_assembled}. Indeed, by definition $(\cC^{\vee})_{\geq 0}$ is the preimage of $\Ind((\cC^{\omega_1})^{op})_{\geq 0},$ and the condition $x\in(\cC^{\vee})_{\leq 0}$ means that the mapping space $\Map(y,\Phi(x))$ is discrete for $y\in \Ind((\cC^{\omega_1})^{op})_{\geq 0},$ or equivalently $\Phi(x)\in \Ind((\cC^{\omega_1})^{op})_{\leq 0}.$ 

% It is clear that $(\cC^{\vee})_{\geq 0}\subset\cC^{\vee}$ is closed under colimits and $(\cC^{\vee})_{\leq 0}\subset\cC^{\vee}$ is closed under finite limits and filtered colimits. The orthogonality also follows from the description. Namely, if we identify $\Ind((\cC^{\omega_1})^{op})$ with $\Fun^{\ex}(\cC^{\omega_1},\Sp),$ then $\Phi(x)=\ev_{\cC}(-,x)$ for $x\in\cC^{\vee}.$ Let $x\in (\cC^{\vee})_{\geq 0},$ then the functor $\ev_{\cC}(-,x):\cC^{\omega_1}\to\Sp$ is pro-representable by the objects of $\cC^{\omega_1}_{\leq 0}.$ Identifying $\cC^{\vee}$ with $\Fun^{\omega_1\hy\rex}(\cC^{\omega_1},\Sp)$ and using the Yoneda lemma, we conclude that for $y\in(\cC^{\vee})_{\leq 0}$ we have $\Hom_{\cC^{\vee}}(x,y)\in\Sp_{\leq 0},$ as required.

Let $x\in\cC^{\vee}.$ It suffices to prove that the object $\tau_{\geq 0}\Phi(x)$ is contained in the essential image of $\Phi.$ If we identify $\Ind((\cC^{\omega_1})^{op})$ with $\Fun^{\ex}(\cC^{\omega_1},\Sp),$ then $\Phi(x)=\ev_{\cC}(-,x).$
%  we need to construct a cofiber sequence $\tau_{\geq 0}x\to x\to \tau_{\leq -1}x$ with $\tau_{\geq 0}x\in(\cC^{\vee})_{\geq 0}$ and $\tau_{\leq -1}x\in (\cC^{\vee})_{\leq 0}[-1].$ 
Consider the functor
\begin{equation}\label{eq:restriction_to_C_leq_0}
F:\cC^{\omega_1}_{\leq 0}\to\cS,\quad F(y)=\Omega^{\infty}\ev_{\cC}(y,x).
\end{equation}
Then $F$ commutes with finite limits (since finite limits in $\cC^{\omega_1}_{\leq 0}$ are the same as in $\cC^{\omega_1}$). Hence, $F$ is pro-representable, i.e. there exists a directed poset $I$ and a functor $I^{op}\to \cC^{\omega_1}_{\leq 0},$ $i\mapsto y_i,$ such that $F\cong\indlim[i]\Map(y_i,-).$
Then \begin{equation*}
\tau_{\geq 0}\Phi(x)\cong \indlim[i]\Hom_{\cC}(y_i,-)\in\Fun^{\ex}(\cC^{\omega_1},\Sp).
\end{equation*}
The essential image of $\Phi$ in $\Fun^{\ex}(\cC^{\omega_1},\Sp)$ is identified with the full subcategory $\Fun^{\omega_1\hy\rex}(\cC^{\omega_1},\Sp).$ Hence, we need to prove that the functor $\indlim[i]\Hom_{\cC}(y_i,-):\cC^{\omega_1}\to\Sp$ commutes with countable colimits. It suffices to show that for any $i\in I$ there exists $j\geq i$ such that the map $y_j\to y_i$ is compact in $\cC_{\leq 0},$ hence also compact in $\cC.$ This follows directly from the fact that the functor \eqref{eq:restriction_to_C_leq_0} commutes with filtered colimits. Namely, it suffices to show that for $z\in \cC^{\omega_1}_{\leq 0}$ and for $\alpha\in\pi_0 F(z)$ there exists $z'\in\cC^{\omega_1}_{\leq 0},$ a compact map $f:z'\to z$ and an element $\beta\in\pi_0 F(z')$ such that $\pi_0 F(f)(\beta)=\alpha.$ It suffices to take $z'=z_n$ for large $n,$ where $\hat{\cY}(z)\cong\inddlim[n\in\N]z_n.$ This proves that $\tau_{\geq 0}\Phi(x)\in\Phi(\cC^{\vee}),$ as required.
\end{proof}

The above construction is functorial in the following sense.

\begin{prop}\label{prop:functorizlity_of_dual_t_structures}
Let $\cC$ and $\cD$ be dualizable categories with compactly assembled continuously bounded $t$-structures $(\cC_{\geq 0},\cC_{\leq 0})$ resp. $(\cD_{\geq 0},\cD_{\leq 0}).$ Let $F:\cC\to\cD$ be a strongly continuous, exact and $t$-exact functor. We equip $\cC^{\vee}$ and $\cD^{\vee}$ with $t$-structures as in Proposition \ref{prop:dual_t_structure}. Then the functor $(F^R)^{\vee}:\cC^{\vee}\to\cD^{\vee}$ is $t$-exact.
\end{prop}

\begin{proof}
We have a commutative square
\begin{equation*}
\begin{tikzcd}
\cC^{\vee} \ar{r}{(\colim)^{\vee}}\ar[d, "(F^R)^{\vee}"] & \Ind((\cC^{\omega_1})^{op})\ar{d}{\Ind((F^{\omega_1})^{op})}\\
\cD^{\vee} \ar{r}{(\colim)^{\vee}} & \Ind((\cD^{\omega_1})^{op}),
\end{tikzcd}
\end{equation*}
in which the functors are strongly continuous and the horizontal arrows are fully faithful. Clearly, the right vertical arrow is $t$-exact. By the proof of Proposition \ref{prop:dual_t_structure} the horizontal arrows are $t$-exact. Hence, the left vertical arrow is also $t$-exact.
\end{proof}

\begin{remark}
It is easy to show that in the situation of Proposition \ref{prop:dual_t_structure} the double dual $t$-structure on $\cC^{\vee\vee}\simeq\cC$ is identified with the original $t$-structure $(\cC_{\geq 0},\cC_{\leq 0}).$ Furthermore, the $t$-structure $(\cC_{\geq 0},\cC_{\leq 0})$ is continuously bounded if and only if so is $((\cC^{\vee})_{\geq 0},(\cC^{\vee})_{\leq 0}).$ This is also equivalent to the right completeness of both $t$-structures.
\end{remark}

\section{Refined $K$-theoretic theorem of the heart for dualizable categories}
\label{sec:theorem_of_the_heart_for_K}

In this section we prove the following result.

\begin{theo}\label{th:theorem_of_the_heart}
Let $\cC$ be a dualizable $t$-category. Let $m\geq 0$ be an integer. Then the map $K_j^{\cont}(\cC_{[0,m]})\to K_j^{\cont}(\cC)$ is an isomorphism for $j\geq -m-2,$ and a monomorphism for $j = -m-3.$
 \end{theo}
 
 We first mention an immediate corollary.
 
 \begin{cor}\label{cor:theorem_of_the_heart_under_condition_on_Exts}
 Let $\cC$ be a dualizable $t$-category. Suppose that for some $n\geq 1$ we have isomorphisms $\Ext^i_{\cC^{\heartsuit}}(x,y)\to \Ext^i_{\cC}(x,y)$ for $i\leq n$ and $x,y\in\cC^{\omega_1,\heartsuit}$ (this assumption always holds for $n=1$). Then the map $K_j^{\cont}(\cC^{\heartsuit})\to K_j^{\cont}(\cC)$ is an isomorphism for $j\geq -n-1,$ and a monomorphism for $j = -n-2.$
 \end{cor}
 
 \begin{proof}
 By Proposition \ref{prop:when_partial_realization_is_an_equivalence} our assumption implies that we have an equivalence $\check{D}(\cC^{\heartsuit})\xto{\sim}\check{\St}(\cC_{[0,n-1]}).$ Hence, the assertion follows from Theorem \ref{th:theorem_of_the_heart}.
 \end{proof}

%If in the above theorem $\cC$ is compactly generated then we recover Barwick's theorem of the heart, and we have $K_{-1}^{\cont}(\cC^{\heartsuit})=K_{-1}^{\cont}(\cC)=0$ by Antieau-Gepner-Heller. 

%In the general case we no longer have the vanishing of $K_{-1}^{\cont},$ as shown by the example of sheaves of vector spaces on the real line. The proof of isomorphisms in non-negative degrees is quite different from Barwick's since the original argument is not available. 

Our arguments are completely different from \cite{Bar15}. Even if we are only interested in $K_{\geq 0}^{\cont},$ still the proof of \cite[Theorem 6.1]{Bar15} does not generalize to dualizable (not compactly generated) categories. We also note that an example below (Corollary \ref{cor:K_-1_not_vanishing}) shows that the group $K_{-1}^{\cont}$ might be non-zero for a dualizable $t$-category. In other words, the statement of \cite[Theorem 2.35]{AGH19} does not hold for dualizable categories. However, a different version of a construction from loc. cit. will be very useful for us, see Propositions \ref{prop:left_right_completion_dualizable} and \ref{prop:Eilenberg_swindle_E^-_E^+} below.

One of the key statements for our proof is the following special case of Theorem \ref{th:theorem_of_the_heart}.

\begin{lemma}\label{lem:key_surjectivity_on_K_-m-2}
Let $\cC$ be a dualizable $t$-category. Then for any $m\geq 0$ the map $K_{-m-2}^{\cont}(\cC_{[0,m]})\to K_{-m-2}^{\cont}(\cC)$ is surjective.
\end{lemma}

The proof of this lemma will require some preparations. The general idea is to find ``sufficiently bounded'' representatives of classes in the negative continuous $K$-theory of $\cC.$ The key ingredient is the following construction.
%, which is a dualizable version of the combination of left and right completion with respect to a $t$-structure. 
We consider the set $\Q_{\geq 0}$ as a totally ordered set with the usual order.

\begin{prop}\label{prop:left_right_completion_dualizable}
Let $\cC$ be a dualizable $t$-category. Consider the small stable category $\cC^{\omega_1,b}$ with the bounded $t$-structure induced from $\cC,$ and let $\cD$ be its simultaneous left and right completion, i.e. $\cD=\prolim[n]\cC^{\omega_1}_{[-n,n]}$. Equivalently, $\cD$ is the left completion of the category $\cC^{\omega_1,+},$ i.e. $\cD\simeq\prolim[n]\cC^{\omega_1}_{\leq n}.$ 

Consider the localizing subcategory $\cE\subset\Ind(\cD)$ generated by objects of the form $\inddlim[a\in\Q_{\geq 0}]F(a),$ where $F:\Q_{\geq 0}\to \cD$ is a functor such that for any rational $0\leq a<b$ and for any $n\in\N$ the morphism $\tau_{[-n,n]}F(a)\to \tau_{[-n,n]}F(b)$ is compact in $\cC.$ Here we identify $\cD_{[-n,n]}$ with $\cC^{\omega_1}_{[-n,n]}\subset\cC.$

\begin{enumerate}[label=(\roman*),ref=(\roman*)]
	\item The category $\cE$ is dualizable and the inclusion functor $\cE\to\Ind(\cD)$ is strongly continuous. Moreover, the above generating collection of objects of $\cE$ is exactly the collection of all $\omega_1$-compact objects of $\cE,$ up to isomorphism. \label{E_is_dualizable_etc}
	\item The (compactly generated) $t$-structure on $\Ind(\cD)$ induces a $t$-structure on $\cE.$ The latter $t$-structure is compactly assembled. The category $\cE^{\omega_1,\heartsuit}$ consists of objects isomorphic to $\inddlim[a\in\Q_{\geq 0}]F(a),$ where $F:\Q_{\geq 0}\to\cD^{\heartsuit}\simeq\cC^{\omega_1,\heartsuit}$ is a functor such that for any rational $0\leq a<b$ the map $F(a)\to F(b)$ is compact in $\cC$ (equivalently, compact in $\cC^{\heartsuit}$). \label{t_structure_on_E}
	\item We have a $t$-exact equivalence $\cC\xto{\sim}\la \cE^b\ra,$ so that the following square commutes:
	\begin{equation*}
	\begin{tikzcd}
	\cC \ar[r, "\sim"]\ar{d}{\hat{\cY}} & \la \cE^b\ra\ar[d]\\
	\Ind(\cC^{\omega_1,b})\ar[r] & \Ind(\cD).
	\end{tikzcd}
	\end{equation*}
	Here the lower horizontal functor is induced by the inclusion $\cC^{\omega_1,b}\simeq \cD^b\to \cD.$ \label{E^b}
\end{enumerate}
\end{prop} 

\begin{proof}
\ref{E_is_dualizable_etc} The argument is the same as for similar constructions with dualizable categories, but we spell it out for completeness. Denote by $\cT\subset\cE$ the full subcategory formed by objects isomorphic to an object in the generating collection described above. Take some $x=\inddlim[a\in \Q_{\geq 0}]F(a)\in\cT,$ where $F:\Q_{\geq 0}\to\cD$ is a functor as above. Then for any $n\in\N_{>0}$ the object $x_n:=\inddlim[0\leq a<n]F(a)$ is also in $\cT,$ and we have $x\cong\indlim[n]x_n.$ Moreover, the map $x_n\to x_{n+1}$ in $\Ind(\cD)$ factors through the constant ind-object $F(n),$ hence this map is compact in $\Ind(\cD).$ This implies that $\cE$ is dualizable and the inclusion functor $\cE\to \Ind(\cD)$ is strongly continuous, see also \cite[Lemma1.83 and its proof]{E24}.

To prove the equality $\cE^{\omega_1}=\cT$ (of strictly full subcategories of $\cE$), we only need to show that $\cT$ is closed under countable colimits and shifts. Clearly, for $x\in\cT$ we have $x[n]\in\cT$ for $n\in\Z.$ Next, let $(F_n:\Q_{\geq 0}\to\cD)_{n\in \N}$ be a collection of functors as above. For $n\geq 0$ define $G_n:\Q_{\geq 0}\to \cD$ to be the left Kan extension of $(F_n)_{\mid \Q_{\geq n}},$ i.e. 
\begin{equation*}
G_n(a)=\begin{cases}F_n(a) & \text{for }a\geq n;\\
0 & \text{for }0\leq a<n.\end{cases} 
\end{equation*}
Then for any $a\in\Q_{\geq 0}$ there are only finitely many $n\in\N$ such that $G_n(a)\ne 0,$ hence the direct sum $G=\biggplus[n]G_n:\Q_{\geq 0}\to\cD$ is well-defined. Moreover, for any rational $0\leq a<b$ and for $k\in\N$ the map $\tau_{[-k,k]}G(a)\to\tau_{[-k,k]}G(b)$ is compact in $\cC.$ We have
\begin{equation*}
\inddlim[a\in\Q_{\geq 0}]G(a)\cong\biggplus[n\in\N]\inddlim[a\in\Q_{\geq 0}]F_n(a),
\end{equation*}
which shows that $\cT\subset\cE$ is closed under countable coproducts. 

It remains to show that $\cT\subset\cE$ is closed under cones (cofibers). Let $\varphi:x\to x'$ be a morphism in $\cT,$ and $x\cong \inddlim[a\in\Q_{\geq 0}]F(a),$ $x'\cong \inddlim[a\in\Q_{\geq 0}]F'(a),$ where the functors $F,F':\Q_{\geq 0}\to \cD$ satisfy the above conditions. We may assume that $\varphi$ is induced by a morphism of functors $F\to F'.$ Indeed, identifying $x$ resp. $x'$ with $\inddlim[n\in\N]F(n)$ resp. $\inddlim[n\in\N]F'(n),$ we first see that there exists a strictly increasing function $f:\N\to\N$ such that $\varphi$ is induced by a morphism of functors $F_{\mid\N}\to (F'_{\mid \N})\circ f.$ Take any strictly increasing function $g:\Q_{\geq 0}\to \Q_{\geq 0}$ such that $g(n)=f(n+1)$ for $n\in\N.$ We obtain a morphism of functors $F\to F'\circ g$ (which factors through the right Kan extension of $(F'_{\mid \N})\circ f$), which induces $\varphi:x\to x'.$ Replacing $F'$ by $F'\circ g,$ we may assume that we have a morphism $u:F\to F'$ which induces $\varphi.$

Now we have $\Cone(\varphi:x\to x')\cong\inddlim[a\in\Q_{\geq 0}]\Cone(u_a:F(a)\to F'(a)).$ It suffices to check that for $0\leq a<b$ and for $k\in\N$ the map $\tau_{[-k,k]}\Cone(u_a)\to\tau_{[-k,k]}\Cone(u_b)$ is compact in $\cC.$ By assumptions on $F$ and $F'$ the maps
\begin{equation*}
\tau_{[-k-1,k+1]}F(a)\to \tau_{[-k-1,k+1]}F\left(\frac{a+b}{2}\right),\quad \tau_{[-k-1,k+1]}F'\left(\frac{a+b}{2}\right)\to \tau_{[-k-1,k+1]}F'(b)
\end{equation*}
%\begin{equation*}
%\tau_{[-k-1,k+1]}F'(a)\to \tau_{[-k-1,k+1]}F'\left(\frac{a+b}{2}\right),\quad \tau_{[-k-1,k+1]}F'\left(\frac{a+b}{2}\right)\to \tau_{[-k-1,k+1]}F'(b)
%\end{equation*}
are compact in $\cC.$ It follows that the map
\begin{equation}\label{eq:compact_map_intermediate}
\Cone(\tau_{[-k-1,k+1]}F(a)\to \tau_{[-k-1,k+1]}F'(a))\to \Cone(\tau_{[-k-1,k+1]}F(b)\to \tau_{[-k-1,k+1]}F'(b))
\end{equation}
is also compact in $\cC.$ Applying $\tau_{[-k,k]}$ to \eqref{eq:compact_map_intermediate} we deduce that the map $\tau_{[-k,k]}\Cone(u_a)\to \tau_{[-k,k]}\Cone(u_b)$ is compact in $\cC.$ This finishes the proof of \ref{E_is_dualizable_etc}.

Now \ref{t_structure_on_E} follows directly from \ref{E_is_dualizable_etc} and Corollary \ref{cor:induced_t_structure_is_compactly_assembled}. Indeed, it suffices to show that for $x\in\cE^{\omega_1}$ we have $\tau_{\geq 0}x\in\cE,$ where $\tau_{\geq 0}$ is the truncation endofunctor of $\Ind(\cD).$ By \ref{E_is_dualizable_etc} we have $x\cong\inddlim[a\in\Q_{\geq 0}]F(a),$ where $F:\Q_{\geq 0}\to \cD$ satisfies the above conditions. Then the functor $\tau_{\geq 0}F:\Q_{\geq 0}\to\cD$ also satisfies the above conditions, hence
\begin{equation*}
\tau_{\geq 0}x\cong \inddlim[a\in\Q_{\geq 0}]\tau_{\geq 0}F(a)\in\cE.
\end{equation*}
We also have
\begin{equation*}
\pi_0 x\cong \inddlim[a\in \Q_{\geq 0}]\pi_0 F(a),
\end{equation*}
which gives the stated description of $\cE^{\omega_1,\heartsuit}.$ This proves \ref{t_structure_on_E}.

Finally, we prove \ref{E^b}. Consider the composition $\Phi:\cC\xto{\hat{\cY}}\Ind(\cC^{\omega_1,b})\to \Ind(\cD).$ Clearly, $\Phi$ is strongly continuous and fully faithful. It suffices to show that the essential image of $\Phi_{\mid \cC^{\omega_1,\heartsuit}}$ coincides with $\cE^{\omega_1,\heartsuit}.$ Take some object $x\in\cC^{\omega_1,\heartsuit},$ and let $\hat{\cY}(x)\cong\inddlim[n\in\N] x_n,$ where $x_n\in\cC^{\omega_1,\heartsuit}$ and each map $x_n\to x_{n+1}$ is compact in $\cC^{\heartsuit}.$ Arguing as in \cite[Proof of Proposition 1.82]{E24} we can extend the functor $\N\to \cC^{\omega_1,\heartsuit},$ $n\mapsto x_n,$ to a functor $F:\Q_{\geq 0}\to \cC^{\omega_1,\heartsuit}$ such that for any rational $0\leq a<b$ the map $F(a)\to F(b)$ is compact in $\cC^{\heartsuit}.$ Then we have $\Phi(x)\cong \inddlim[a\in\Q_{\geq 0}]F(a),$ where we identify $\cC^{\omega_1,\heartsuit}$ with $\cD^{\heartsuit}.$ In particular, we have $\Phi(x)\in\cE^{\omega_1,\heartsuit}$ and by \ref{t_structure_on_E} any isomorphism class in $\cE^{\omega_1,\heartsuit}$ can be obtained in this way. This proves \ref{E^b}.
\end{proof}

The following is a slightly more sophisticated version of Eilenberg swindle.

\begin{prop}\label{prop:Eilenberg_swindle_E^-_E^+}
Let $\cC,\cD,\cE$ be as in Proposition \ref{prop:left_right_completion_dualizable}. Then any localizing invariant vanishes on $\la \cE^-\ra$ and $\la \cE^+\ra.$ In particular, we have $K^{\cont}(\la \cE^-\ra) = K^{\cont}(\la \cE^+\ra) = 0.$ We obtain an isomorphism
\begin{equation}\label{eq:suspension_for_K_theory_of_t_categories}
K^{\cont}(\cC)\simeq \Omega K^{\cont}(\cE).
\end{equation}
\end{prop}

\begin{proof}
Assuming the vanishing of $K^{\cont}(\la \cE^-\ra)$ and $K^{\cont}(\la \cE^+\ra),$ the isomorphism \eqref{eq:suspension_for_K_theory_of_t_categories} follows from Corollary \ref{cor:cartesian_square_C^-_C^+_C^b}.

It suffices to show that we have the vanishings
\begin{equation}\label{eq:vanishings_for_E^-_and_E^+}
[\Id_{\la \cE^-\ra}]=0\in K_0(\Fun^{LL}(\la \cE^-\ra,\la \cE^-\ra)),\quad [\Id_{\la \cE^+\ra}]=0\in K_0(\Fun^{LL}(\la \cE^+\ra,\la \cE^+\ra)).
\end{equation}
Note that $\la\cE^-\ra$ is contained in $\Ind(\cD^-)$ as a subcategory of $\Ind(\cD),$ and the inclusion functor is strongly continuous. For an object $x\in\cD^-$ the coproduct $\biggplus[n\geq 0]x[2n]$ exists in $\cD^-.$ Hence, we have a strongly continuous endofunctor
\begin{equation*}
\Phi=\Ind(\biggplus[n\in\N]\Sigma^{2n}):\Ind(\cD^-)\to\Ind(\cD^-).\end{equation*}
We claim that $\Phi$ preserves the subcategory $\la \cE^-\ra.$ It suffices to prove that for $x\in\cE^{\omega_1}_{\geq 0}$ we have $\Phi(x)\in \la\cE^-\ra.$ By Proposition \ref{prop:left_right_completion_dualizable} we have $x\cong\inddlim[a\in\Q_{\geq 0}]F(a),$ where $F:\Q_{\geq 0}\to \cD_{\geq 0}$ is a functor such that for any rational $0\leq a<b$ and for any $k\in\N$ the morphism $\tau_{[0,k]}F(a)\to\tau_{[0,k]}F(b)$ is compact in $\cC.$ Define $G:\Q_{\geq 0}\to\cD$ by the formula \begin{equation*}
G(a) = \biggplus[n\in\N]F(a)[2n].
\end{equation*}
Then for $a\in\Q_{\geq 0}$ and $k\in\N$ we have 
\begin{equation*}
\tau_{[-k,k]}G(a)\cong\biggplus[0\leq n\leq \lfloor \frac{k}{2}\rfloor] \tau_{[0,k-2n]}(F(a))[2n].
\end{equation*} 
Hence, for rational $0\leq a<b$ the map $\tau_{[-k,k]}G(a)\to \tau_{[-k,k]}G(b)$
is compact. Therefore, $\Phi(x)\cong\inddlim[a\in\Q_{\geq 0}]G(a)\in\la\cE^-\ra,$ as required.
 
Denote by $\Psi:\la \cE^-\ra\to \la \cE^-\ra$ the restriction of $\Phi.$ Then $\Psi$ is strongly continuous and by construction we have $\Psi\cong\Psi[2]\oplus\Id_{\la \cE^-\ra}.$ This proves the first vanishing in \eqref{eq:vanishings_for_E^-_and_E^+}. The other vanishing is proved by a similar argument.\end{proof}

We will need to consider the natural $t$-structures on the iterated Calkin categories. Recall from Subsection \ref{ssec:reminder_on_compass} that for a dualizable category $\cC$ we denote by $\Calk_{\omega_1}(\cC)$ the compactly generated category $\Ind(\cC^{\omega_1})/\hat{\cY}(\cC).$ We will denote by $\Calk_{\omega_1}^n(\cC)$ the $n$-th iteration of this construction, $n\geq 1.$ It is convenient to deal with them inductively, using the following observation.

\begin{prop}\label{prop:t_structure_on_Calk_properties}
Let $\cC$ be a dualizable category with an $\omega_1$-accessible $t$-structure $(\cC_{\leq 0},\cC_{\geq 0}),$ compatible with filtered colimits. Suppose that for some $n\geq 0$ the functor $\hat{\cY}_{\cC}[-n]:\cC\to \Ind(\cC^{\omega_1})$ is left $t$-exact. Consider the functor 
\begin{equation*}
\Phi:\cC\simeq\Ind_{\omega_1}(\cC^{\omega_1})\hto\Ind(\cC^{\omega_1})\to\Calk_{\omega_1}(\cC).
\end{equation*}
\begin{enumerate}[label=(\roman*),ref=(\roman*)]
	\item The $t$-structure on $\Ind(\cC^{\omega_1})$ induces an $\omega_1$-accessible $t$-structure on $\Calk_{\omega_1}(\cC)$ via the inclusion
	\begin{equation*}
	\Calk_{\omega_1}(\cC)\simeq \ker(\colim:\Ind(\cC^{\omega_1})\to\cC)\subset\Ind(\cC^{\omega_1}).
	\end{equation*}
	Moreover, this $t$-structure is compatible with filtered colimits. \label{t_structure_on_Calk}
	\item The functor $\Phi:\cC\to \Calk_{\omega_1}(\cC)$ is right $t$-exact, and the functor $\Phi[-n-1]$ is left $t$-exact. \label{functor_to_Calk_amplitude}
	\item For $\cD=\Calk_{\omega_1}(\cC),$ the functor $\hat{\cY}_{\cD}[-n-1]:\cD\to\Ind(\cD^{\omega_1})$ is left $t$-exact. \label{amplitude_for_Calk_jumps} 
\end{enumerate}
\end{prop}

\begin{proof} \ref{t_structure_on_Calk} is straightforward since the colimit functor $\Ind(\cC^{\omega_1})\to\cC$ is $t$-exact. Note that the $\omega_1$-accessibility of the induced $t$-structure on $\Calk_{\omega_1}(\cC)$ follows from the identification of $\Calk_{\omega_1}(\cC)^{\omega_1}$ with the kernel of the $t$-exact functor $\colim:\Ind(\cC^{\omega_1})^{\omega_1}\to\cC^{\omega_1}.$

To prove \ref{functor_to_Calk_amplitude}, it suffices to consider the restriction of $\Phi$ to $\cC^{\omega_1}.$ Considering $\Calk_{\omega_1}(\cC)$ as a full subcategory of $\Ind(\cC^{\omega_1}),$ we have
\begin{equation*}
\Phi(x)[-n-1]\cong \Fiber(\hat{\cY}_{\cC}(x)[-n]\to\cY_{\cC}(x)[-n])\in \Calk_{\omega_1}(\cC)_{\leq 0},\quad x\in\cC^{\omega_1}_{\leq 0},
\end{equation*}
since the functor $\hat{\cY}_{\cC}[-n]$ is left $t$-exact by assumption. This proves that the functor $\Phi[-n-1]$ is left $t$-exact, as stated. It is also clear that $\Phi$ is right $t$-exact.

We prove \ref{amplitude_for_Calk_jumps}. Consider again $\cD = \Calk_{\omega_1}(\cC)$ as a full subcategory of $\Ind(\cC^{\omega_1}),$ and recall that for $x\in\cC^{\omega_1}$ the object $\Cone(\hat{\cY}_{\cC}(x)\to\cY_{\cC}(x))\in\cD$ is compact. Now take any object $\inddlim[i]x_i\in\cD_{\leq 0},$ where $x_i\in\cC^{\omega_1}_{\leq 0}$ and $\indlim[i]x_i = 0\in\cC.$ Consider the category $\Ind(\cD^{\omega_1})$ as a full subcategory of $\Ind(\Ind(\cC^{\omega_1})^{\omega_1}).$ Then we have
\begin{multline*}
\hat{\cY}_{\cD}(\inddlim[i]x_i)[-n-1]=\hat{\cY}_{\cD}(\indlim[i]\Cone(\hat{\cY}_{\cC}(x_i)\to\cY_{\cC}(x_i)))[-n-1]\\
\cong\Fiber(\inddlim[i]\hat{\cY}_{\cC}(x_i)[-n]\to \inddlim[i]\cY_{\cC}(x_i)[-n])\in \Ind(\Ind(\cC^{\omega_1})^{\omega_1})_{\leq 0},
\end{multline*}
because again by assumption we have $\hat{\cY}_{\cC}(x_i)[-n]\in\Ind(\cC^{\omega_1})^{\omega_1}_{\leq 0}.$
\end{proof}

In the following proof we use the notation $\Ind^{\omega_1}(\cA)$ for $\Ind(\cA)^{\omega_1},$ where $\cA$ is a small category. We write $(\Ind^{\omega_1})^n(\cA)$ for the $n$-th iteration of this construction, and for convenience we put $(\Ind^{\omega_1})^0(\cA)=\cA.$ 

\begin{proof}[Proof of Lemma \ref{lem:key_surjectivity_on_K_-m-2}]
Let $\cD$ and $\cE$ be as in Proposition \ref{prop:left_right_completion_dualizable}. By loc. cit. we can identify $\cC$ with $\la \cE^b\ra,$ and $\cC_{[0,m]}$ with $\cE_{[0,m]}.$  By Proposition \ref{prop:Eilenberg_swindle_E^-_E^+} we have the following isomorphisms:
\begin{equation}\label{eq:from_K_0_of_E_to_K_0_of_quotient}
\varphi:K_{-m-1}^{\cont}(\cE)\xto{\sim}K_{-m-1}^{\cont}(\cE/\la\cE^+\ra)\xto{\sim}K_{-m-1}^{\cont}(\la\cE^-\ra/\la\cE^b\ra),
\end{equation}
\begin{equation*}
\psi:K_{-m-1}^{\cont}(\la\cE^-\ra/\la\cE^b\ra)\xto{\sim} K_{-m-2}^{\cont}(\la\cE^b\ra).
\end{equation*}
It suffices to prove that the image of (the isomorphism) $\psi\circ\varphi$ is contained in the image of the map $K_{-m-2}^{\cont}(\cE_{[0,m]})\to K_{-m-2}^{\cont}(\la \cE^b\ra).$

 We identify $K_{-m-1}^{\cont}(\cE)$ with $K_0(\Calk_{\omega_1}^{m+1}(\cE)^{\omega}).$ Below for any $n\geq 1$ we identify the category $\Calk_{\omega_1}^n(\cE)^{\omega_1}$ with its essential image in $(\Ind^{\omega_1})^n(\cE^{\omega_1}),$ and similarly for $\check{\St}(\cE_{[0,m]})$ and $\la \cE^b\ra.$ More precisely, $\Calk_{\omega_1}^n(\cE)^{\omega_1}\subset (\Ind^{\omega_1})^n(\cE^{\omega_1})$ is the intersection of kernels of the functors
\begin{equation*}
(\Ind^{\omega_1})^k(\colim):(\Ind^{\omega_1})^n(\cE^{\omega_1})\to (\Ind^{\omega_1})^{n-1}(\cE^{\omega_1}),\quad 0\leq k\leq n-1.
\end{equation*}

Take some class $\alpha\in K_{-m-1}^{\cont}(\cE),$ and choose an object $x\in \Calk_{\omega_1}^{m+1}(\cE)^{\omega}$ such that $[x]=\alpha.$ Consider the natural $t$-structure on $\Calk_{\omega_1}^{m+1}(\cE)^{\omega_1},$ which is obtained by an iterated application of Proposition \ref{prop:t_structure_on_Calk_properties}. Then the object $\tau_{\geq 0}x\in\Calk_{\omega_1}^{m+1}(\la\cE^-\ra)^{\omega_1}$ is a lift of the image of $x$ in $\Calk_{\omega_1}^{m+1}(\cE/\la \cE^+\ra)^{\omega}\simeq\Calk_{\omega_1}^{m+1}(\la \cE^-\ra/\la \cE^b\ra)^{\omega}.$ The latter object represents the element $\varphi(\alpha)\in K_{-m-1}^{\cont}(\la \cE^-\ra/\la \cE^b\ra).$

We can now describe a representative for $\psi(\varphi(\alpha))\in K_{-m-2}(\la \cE^b\ra).$ Namely, consider the composition 
\begin{equation*}
\Phi:\Calk_{\omega_1}^{m+1}(\cE)^{\omega_1}\to (\Calk_{\omega_1}^{m+1}(\cE)^{\omega_1}/\Calk_{\omega_1}^{m+1}(\cE)^{\omega})^{\Kar}\simeq\Calk_{\omega_1}^{m+2}(\cE)^{\omega}\hto \Calk_{\omega_1}^{m+2}(\cE)^{\omega_1}.
\end{equation*}
Then the object $y=\Phi(\tau_{\geq 0}x)$ is contained in the full subcategory $\Calk_{\omega_1}^{m+2}(\la \cE^b \ra)^{\omega}\subset \Calk_{\omega_1}^{m+2}(\cE).$ We then have $[y]=\psi(\varphi(\alpha)).$

Now, by (an iterated application of) Proposition \ref{prop:t_structure_on_Calk_properties} the functor $\Phi$ is right $t$-exact and the functor $\Phi[-m-2]$ is left $t$-exact. In particular, we have $y\in \Calk_{\omega_1}^{m+2}(\la \cE^b \ra)^{\omega_1}_{\geq 0}.$ On the other hand, since $\Phi(x) = 0,$ we obtain
\begin{equation*}
y = \Phi(\tau_{\geq 0}x)\cong \Phi(\tau_{\leq -1}x)[-1]\in \Calk_{\omega_1}^{m+2}(\la \cE^b \ra)^{\omega_1}_{\leq m}.
\end{equation*}
Therefore, we have
\begin{equation*}
y\in \Calk_{\omega_1}^{m+2}(\la \cE^b \ra)^{\omega_1}_{[0,m]}.
\end{equation*}
It follows formally that the image of $y$ in $(\Ind^{\omega_1})^{m+2}(\la\cE^b\ra^{\omega_1})$ is contained in $(\Ind^{\omega_1})^{m+2}(\cE_{[0,m]}^{\omega_1}).$ Denote by $z\in(\Ind^{\omega_1})^{m+2}(\check{\St}(\cE_{[0,m]})^{\omega_1})$ the image of $y$ under the functor
\begin{equation*}
(\Ind^{\omega_1})^{m+2}(\cE_{[0,m]}^{\omega_1})\to (\Ind^{\omega_1})^{m+2}(\check{\St}(\cE_{[0,m]})^{\omega_1}).
\end{equation*}
By construction, we have $z\in\Calk_{\omega_1}^{m+2}(\check{\St}(\cE_{[0,m]}))^{\omega_1}.$ 

We claim that $z$ is compact. Equivalently, this means that $z$ is annihilated by the functor
\begin{equation*}
\Psi:\Calk_{\omega_1}^{m+2}(\check{\St}(\cE_{[0,m]}))^{\omega_1}\to \Calk_{\omega_1}^{m+3}(\check{\St}(\cE_{[0,m]}))^{\omega_1}
\end{equation*}
To see this, consider the similar functor
\begin{equation*}
\Theta:\Calk_{\omega_1}^{m+2}(\la \cE^b\ra)^{\omega_1}\to \Calk_{\omega_1}^{m+3}(\la \cE^b\ra)^{\omega_1}.
\end{equation*}
By Proposition \ref{prop:t_structure_on_Calk_properties}, both $\Psi$ and $\Theta$ preserve bounded objects. We obtain a commutative square
\begin{equation*}
\begin{tikzcd}
\Calk_{\omega_1}^{m+2}(\check{\St}(\cE_{[0,m]}))^{\omega_1,b} \ar[d, "\Psi"]\ar[r, "F"] & \Calk_{\omega_1}^{m+2}(\la\cE^b\ra)^{\omega_1,b}\ar[d, "\Theta"]\\
\Calk_{\omega_1}^{m+3}(\check{\St}(\cE_{[0,m]}))^{\omega_1,b}\ar[r, "G"] & \Calk_{\omega_1}^{m+3}(\la\cE^b\ra)^{\omega_1,b}.
\end{tikzcd}
\end{equation*}
Note that $G$ is $t$-exact and it induces an equivalence between the hearts. Indeed, the hearts of both the source and the target are identified with the full subcategory of $(\Ind^{\omega_1})^{m+3}(\cE^{\omega_1,\heartsuit}),$ which is the intersection of the kernels of the functors
\begin{equation*}
(\Ind^{\omega_1})^k(\colim):(\Ind^{\omega_1})^{m+3}(\cE^{\omega_1,\heartsuit})\to (\Ind^{\omega_1})^{m+2}(\cE^{\omega_1,\heartsuit}),\quad 0\leq k\leq m+2.
\end{equation*}
It follows that $G$ is conservative. Now we have $F(z)\cong y$ and $\Theta(y)=0,$ hence $G(\Psi(z))\cong\Theta(F(z))=0.$ Therefore, $\Psi(z)=0,$ i.e. $z\in\Calk_{\omega_1}^{m+2}(\check{\St}(\cE_{[0,m]}))^{\omega},$ as stated.

We conclude that the class $[y]=\psi(\varphi(\alpha))\in K_{-m-2}^{\cont}(\la \cE^b\ra)$ is the image of $[z]\in K_{-m-2}^{\cont}(\cE_{[0,m]}).$ Since $\alpha$ was arbitrary, this proves the lemma.
\end{proof}

To deduce Theorem \ref{th:theorem_of_the_heart} from Lemma \ref{lem:key_surjectivity_on_K_-m-2} we need to analyze the coherently assembled exact categories $\cC_{[0,m]}.$ 
%For $m>0$ we define the following exact category $\wt{\cC}_{[0,m]}:$ its underlying category is $\cC_{[0,m]},$ and a morphism $f:x\to y$ is an inclusion in $\wt{\cC}_{[0,m]}$ if both morphisms $\pi_{m-1}(f):\pi_{m-1}(x)\to \pi_{m-1}(y)$ and $\pi_m(f):\pi_m(x)\to \pi_m(y)$ are monomorphisms in $\cC^{\heartsuit}.$ 

\begin{prop}\label{prop:properties_of_tilde_C}
Let $\cC$ be a dualizable $t$-category and let $m\geq 1$ be an integer. Consider the exact category $\wt{\cC}_{[0,m]}$ whose underlying category is $\cC_{[0,m]}$ and the exact structure is the following. A morphism $f:x\to y$ is an inclusion in $\wt{\cC}_{[0,m]}$ if and only if both morphisms $\pi_{m-1}(f):\pi_{m-1}(x)\to \pi_{m-1}(y)$ and $\pi_m(f):\pi_m(x)\to \pi_m(y)$ are monomorphisms in $\cC^{\heartsuit}.$ Further, $f:x\to y$ is a projection in $\wt{\cC}_{[0,m]}$ if and only if both morphisms $\pi_0(f):\pi_0(x)\to\pi_0(y)$ and $\pi_m(f):\pi_m(x)\to \pi_m(y)$ are epimorphisms in $\cC^{\heartsuit}.$
\begin{enumerate}[label=(\roman*),ref=(\roman*)]
\item $\wt{\cC}_{[0,m]}$ is a coherently assembled exact category. We denote by $i:\wt{\cC}_{[0,m]}\to \check{\St}(\wt{\cC}_{[0,m]})$ the universal continuous exact functor. \label{tilde_C_is_well_defined}
\item We have strongly continuous fully faithful exact functors $j_1:\cC_{[0,m-1]}\to\wt{\cC}_{[0,m]},$ $j_2:\cC^{\heartsuit}\to\wt{\cC}_{[0,m]},$ given by $j_1(x)=x,$ $j_2(y)=y[m].$ They induce fully faithful functors on presentable stable envelopes, which we denote by the same symbols. They give a semi-orthogonal decomposition in $\Cat_{\st}^{\dual}:$
\begin{equation}\label{eq:SOD_for_tilde_C}
\check{\St}(\wt{\cC}_{[0,m]})=\la j_1(\check{\St}(\cC_{[0,m-1]})),j_2(\check{D}(\cC^{\heartsuit}))\ra.	
\end{equation} \label{SOD_for_tilde_C}
\item We have a short exact sequence in $\Cat_{\st}^{\dual}:$
\begin{equation*}
0\to \cT_m\to \check{\St}(\wt{\cC}_{[0,m]})\xto{q_m}\check{\St}(\cC_{[0,m]})\to 0.
\end{equation*}
Here $q_m$ is induced by the identity functor $\wt{\cC}_{[0,m]}\to\cC_{[0,m]},$ and $\cT_m$ is the kernel of $q_m.$ The functor
\begin{equation*}
\Phi_m:\cC^{\heartsuit}\to \cT_m,\quad \Phi_m(x)=\Cone(i(x[m-1])[1]\to i(x[m])),
\end{equation*}
is strongly continuous and fully faithful, and its essential image is the heart of a compactly assembled continuously bounded $t$-structure. Moreover, the maps $\Ext^s_{\cC^{\heartsuit}}(x,y)\xto{\sim}\Ext^s_{\cT_m}(\Phi_m(x),\Phi_m(y))$ are isomorphisms for $s\leq m+1,$ where $x,y\in\cC^{\heartsuit}.$ \label{short_exact_sequence_for_tilde_C}
\end{enumerate} 
\end{prop}

\begin{proof}
The proof of \ref{tilde_C_is_well_defined} is straightforward: it is clear that the exact structure is well-defined, it induces an exact structure on $\wt{\cC}^{\omega_1}_{[0,m]},$ and the functors $\colim:\Ind(\wt{\cC}^{\omega_1}_{[0,m]})\to \wt{\cC}_{[0,m]}$ and $\hat{\cY}:\wt{\cC}_{[0,m]}\to \Ind(\wt{\cC}^{\omega_1}_{[0,m]})$ are exact.

We prove \ref{SOD_for_tilde_C}. Both functors $j_1:\cC_{[0,m-1]}\to \wt{\cC}_{[0,m]}$ and $j_2:\cC^{\heartsuit}\to \wt{\cC}_{[0,m]}$ are exact. The left adjoint to $j_1$ is given by $j_1^L(x)=\tau_{\leq m-1} x,$ hence it is exact. The right adjoint to $j_2$ is given by $j_2^R(x)=\pi_m(x),$ hence it is also exact. We have the isomorphisms $j_1^L j_1\xto{\sim}\Id_{\cC_{[0,m-1]}},$ $\Id_{\cC^{\heartsuit}}\xto{\sim} j_2^R j_2,$ and the vanishing $j_2^Rj_1=0.$ Moreover, we have a functorial short exact sequence
\begin{equation*}
j_2 j_2^R(x)\to x\to j_1 j_1^L(x),\quad x\in\wt{\cC}_{[0,m]}.
\end{equation*} 
Passing to presentable stable envelopes we deduce the fully faithfulness of $j_1$ and $j_2,$ and obtain the semi-orthogonal decomposition \eqref{eq:SOD_for_tilde_C}. This proves \ref{SOD_for_tilde_C}.

Finally we prove \ref{short_exact_sequence_for_tilde_C}. It is clear that $q_m$ is a strongly continuous quotient functor. We first show that $\cT_m$ (i.e. the kernel of $q_m$) is generated by the essential image of $\Phi_m$ as a localizing subcategory. Denote by $\cC_{[0,m]}^{\omega_1,\add}$ the additive category $\cC_{[0,m]}^{\omega_1}$ with a split exact structure. Clearly, $\cT_m$ is generated by the images of objects of the form $E(f)\in\St(\cC_{[0,m]}^{\omega_1,\add}),$ where  $f:x\to y$ is an exact inclusion in $\cC^{\omega_1}_{[0,m]},$ which means that $\pi_m(f)$ is a monomorphism in $\cC^{\heartsuit}$ (the notation $E(f)$ is introduced in \eqref{eq:object_E_of_f}). Denote by $\bbar{E}(f)$ the image of $E(f)$ in $\check{\St}(\wt{\cC}_{[0,m]}).$ It suffices to show that we have an isomorphism $\bbar{E}(f)\cong \Phi_m(z),$ where $z=\ker(\pi_{m-1}(f))\in\cC^{\omega_1,\heartsuit}.$ First we observe that $\tau_{\leq m-1}(f)$ is also an exact inclusion in $\cC^{\omega_1}_{[0,m]},$ and $\bbar{E}(f)\cong \bbar{E}(\tau_{\leq m-1}(f)).$ Hence, we may and will assume that $\pi_m(x)=\pi_m(y)=0.$ Put $x'=\Cone(z[m-1]\to x),$ then $f$ factors uniquely through $x',$ and the induced morphism $f':x'\to y$ is an exact inclusion in $\wt{\cC}_{[0,m]}.$ Hence $\bbar{E}(f)\cong\bbar{E}(z[m-1]\to 0)\cong \Phi_m(z),$ as stated.

It is clear that $\Phi_m$ is a strongly continuous exact functor (by definition of the exact category $\wt{\cC}_{[0,m]}$). Hence, by Propositions \ref{prop:St_of_Grothendieck_abelian} and \ref{prop:St_of_coherently_assembled} $\Phi_m$ induces a strongly continuous exact functor $\check{D}(\cC^{\heartsuit})\to \cT_m.$ By Proposition \ref{prop:t_structure_on_dualizable_via_functor_from_abelian} it remains to prove that the induced maps $\Ext^s_{\cC^{\heartsuit}}(x,y)\to \Ext^s_{\cT_m}(\Phi_m(x),\Phi_m(y))$ are isomorphisms for $s\leq m+1$ and $x,y\in\cC^{\heartsuit}$ (where $s$ can be strictly negative). We take such $x,y$ and denote by $i':C_{[0,m-1]}\to \check{\St}(\cC_{[0,m-1]})$ the universal functor. By \ref{SOD_for_tilde_C} we have a cartesian square of spectra:
\begin{equation}\label{eq:pullback_square_for_Hom_in_T_m}
\begin{tikzcd}
\Hom_{\cT_m}(\Phi_m(x),\Phi_m(y)) \ar[r]\ar[d] & \Hom_{\check{\St}(\cC_{[0,m-1]})}(i'(x[m-1]),i'(y[m-1]))\ar[d]\\
\Hom_{\check{D}(\cC^{\heartsuit})}(x,y)\ar[r] & \Hom_{\check{\St}(\wt{\cC}_{[0,m]})}(i(x[m-1])[1],i(y[m])).
\end{tikzcd}
\end{equation}
We first deal with the relevant homotopy groups of the bottom right spectrum. The fully faithfulness of the functor $i$ gives the isomorphisms
\begin{equation*}
\pi_{-s}\Hom_{\check{\St}(\wt{\cC}_{[0,m]})}(i(x[m-1])[1],i(y[m]))\cong \Ext^s_{\cC}(x,y)\quad\text{for }s\leq 1.
\end{equation*}
Next, for $2\leq s\leq m+1$ we compute:
\begin{multline*}
\pi_{-s}\Hom_{\check{\St}(\wt{\cC}_{[0,m]})}(i(x[m-1])[1],i(y[m]))\\
\cong \pi_{-s}\Hom_{\check{\St}(\wt{\cC}_{[0,m]})}(i(x[m-s+1])[s-1],i(y[m]))\cong \Ext^1_{\wt{\cC}_{[0,m]}}(x[m-s+1],y[m])\\
\cong \Ext^s_{\cC}(x,y).
\end{multline*}
%Similarly, we have isomorphisms
%\begin{equation*}
%\pi_{-s}\Hom_{\check{\St}(\cC_{[0,m-1]})}(i'(x[m-1]),i'(y[m-1]))\cong
%\begin{cases} \Ext^s_{\cC}(x,y) & \text{for }s\leq m;\\
%\Ext^2_{\cC_{[0,m-1]}}(x,y[m-1]).
%\end{cases} 
%\end{equation*}
Therefore, by Proposition \ref{prop:properties_of_functor_St_of_C_0_m_to_C} the right vertical arrow in \eqref{eq:pullback_square_for_Hom_in_T_m} induces isomorphisms on $\pi_{-s}$ for $s\leq m.$ and a monomorphism for $s=m+1.$ Hence, the same holds for the left vertical arrow in \eqref{eq:pullback_square_for_Hom_in_T_m}. But it has a right inverse, induced by the functor $\Phi_m.$ We conclude that the map $\Ext^s_{\cC^{\heartsuit}}(x,y)\to\Ext^s_{\cT_m}(\Phi_m(x),\Phi_m(y))$ is an isomorphism for $s\leq m+1,$ as required. This proves \ref{short_exact_sequence_for_tilde_C}. 
\end{proof}

\begin{proof}[Proof of Theorem \ref{th:theorem_of_the_heart}] We start with the key cofiber sequence of spectra, which will be used repeatedly. Let $m\geq 1$ and consider the exact category $\wt{\cC}_{[0,m]}$ and the dualizable category $\cT_m$ from Proposition \ref{prop:properties_of_tilde_C}. By loc. cit. we have a cofiber sequence
\begin{equation}\label{eq:exact_triangle_K_theory_C_0m}
K^{\cont}(\cT_m)\xto{(\alpha,\beta)} K^{\cont}(\cC_{[0,m-1]})\oplus K^{\cont}(\cC^{\heartsuit})\to K^{\cont}(\cC_{[0,m]}).
\end{equation}
	
 Next, we prove that for $m\geq 1$ the map $K_{j}^{\cont}(\cC_{[0,m-1]})\to K_{j}^{\cont}(\cC_{[0,m]})$ is an isomorphism for each $j\in [-m-1,-2],$ and a monomorphism for $j=-m-2.$ By Proposition \ref{prop:properties_of_tilde_C} the category $\cT_m$ from \eqref{eq:exact_triangle_K_theory_C_0m} has a compactly assembled continuously bounded $t$-structure, and we have an equivalence $\cC^{\heartsuit}\simeq \cT_m^{\heartsuit}.$ By the assertion on $\Ext^{\leq m+1}$ from loc. cit. and by Proposition \ref{prop:when_partial_realization_is_an_equivalence} we have an equivalence $\check{D}(\cC^{\heartsuit})\xto{\sim}\check{\St}((\cT_m)_{[0,l]})$ for $0\leq l\leq m.$ Hence, by Lemma \ref{lem:key_surjectivity_on_K_-m-2} we have a surjection $K_{j}^{\cont}(\cC^{\heartsuit})\onto K_{j}^{\cont}(\cT_m)$ for $j\in [-m-2,-2].$ On the other hand, the composition $\cT_m\to\check{\St}(\wt{\cC}_{[0,m]})\xto{j_2^R}\check{D}(\cC^{\heartsuit})$ is left inverse to the realization functor. Therefore the map $\beta$ from \eqref{eq:exact_triangle_K_theory_C_0m} induces an isomorphism $K_{j}^{\cont}(\cT_m)\xto{\sim} K_{j}^{\cont}(\cC^{\heartsuit})$ for $j\in [-m-2,-2]$ The long exact sequence of homotopy groups for \eqref{eq:exact_triangle_K_theory_C_0m} shows that the map $K_{j}^{\cont}(\cC_{[0,m-1]})\to K_{j}^{\cont}(\cC_{[0,m]})$ is an isomorphism for $j\in [-m-1,-2],$ and a monomorphism for $j=-m-2,$ as stated. 
 
 Passing to the colimit over $m$ and using Proposition \ref{prop:colimit_of_St_of_C_0_n} we deduce that for $m\geq 0$ the map $K_{j}^{\cont}(\cC_{[0,m]})\to K_{j}^{\cont}(\cC)$ is an isomorphism for $j\in [-m-2,2],$ and a monomorphism for $j=-m-3.$ 
 
 It remains to prove that we also have isomorphisms for $j\geq -1.$ For this we may and will assume that $m=0,$ i.e. we will prove that the map $K_j^{\cont}(\cC^{\heartsuit})\to K_j^{\cont}(\cC)$ is an isomorphism for $j\geq -1.$ We will use the already established case $j=-2$ as a base of induction on $j.$
 
 Suppose that the statement is proven for $j\in [-2,n],$ where $n\in\Z_{\geq -2}.$ We need to prove that the map $K_{n+1}^{\cont}(\cC^{\heartsuit})\to K_{n+1}^{\cont}(\cC)$ is an isomorphism. Take some $m\geq 1.$ It follows from the induction hypothesis that in \eqref{eq:exact_triangle_K_theory_C_0m} the map $\pi_n\beta:K_n^{\cont}(\cT_m)\to K_n^{\cont}(\cC^{\heartsuit})$ is an isomorphism (it has a right inverse, which is an isomorphism). It follows from the long exact sequence of homotopy groups that the images of the maps $K_{n+1}^{\cont}(\cC_{[0,m-1]})\to K_{n+1}^{\cont}(\cC_{[0,m]})$ and $K_{n+1}^{\cont}(\cC^{\heartsuit})\to K_{n+1}^{\cont}(\cC_{[0,m]})$ generate the target. But the image of the latter map is contained in the image of the former map, hence the former map is surjective. Passing to the colimit over $m$ and using Proposition \ref{prop:colimit_of_St_of_C_0_n} again, we conclude that the map $K_{n+1}^{\cont}(\cC^{\heartsuit})\to K_{n+1}^{\cont}(\cC)$ is surjective. Now let again $m\geq 1$ be a positive integer. Applying the  above surjectivity to $\cT_m,$ we see that the realization functor induces a surjection $K_{n+1}^{\cont}(\cC^{\heartsuit})\onto K_{n+1}^{\cont}(\cT_m).$ Therefore, its left inverse $\pi_{n+1}\beta$ is an isomorphism. Applying again the long exact sequence for \eqref{eq:exact_triangle_K_theory_C_0m} we deduce that the map $K_{n+1}^{\cont}(\cC_{[0,m-1]})\to K_{n+1}^{\cont}(\cC_{[0,m]})$ is an isomorphism. Passing to the colimit over $m$ we conclude that the map $K_{n+1}^{\cont}(\cC^{\heartsuit})\to K_{n+1}^{\cont}(\cC)$ is an isomorphism. This proves the induction step and the theorem.
\end{proof}

\section{Coconnectivity estimates}
\label{sec:coconnectivity_estimates}

\subsection{Estimates for $K$-groups}

Using Theorem \ref{th:theorem_of_the_heart} we prove the following key result, which is a dual analogue of Waldhausen's connectivity estimates.

\begin{theo}\label{th:coconnectivity_estimates}
Let $\cC$ and $\cD$ be dualizable $t$-categories, and let $n\geq 1$ be an integer. Suppose that $F:\cC\to\cD$ is a functor satisfying the following conditions:
\begin{itemize}
	\item $F$ is strongly continuous, exact and $t$-exact;
	\item $F(\cC)$ generates $\cD$ as a localizing subcategory;
	\item for $x,y\in\cC^{\omega_1,\heartsuit}$ the map $\Ext^i_{\cC}(x,y)\to \Ext^i_{\cD}(F(x),F(y))$ is an isomorphism for $i\leq n-1,$ and a monomorphism for $i=n.$ 
\end{itemize}
Then the induced map $K_j^{\cont}(\cC)\to K_j^{\cont}(\cD)$ is an isomorphism for $j\geq -n,$ and a monomorphism for $j=-n-1.$
\end{theo}

It is convenient to introduce the following notion.

\begin{defi}\label{def:-n_coconnective_functor} Let $\cC,\cD$ and $n\geq 1$ be as in Theorem \ref{th:coconnectivity_estimates}. We say that a functor $F:\cC\to\cD$ is $(-n)$-coconnective if it satisfies the conditions of Theorem \ref{th:coconnectivity_estimates}.\end{defi}

%We will deduce Theorem \ref{th:coconnectivity_estimates} from Theorem \ref{th:theorem_of_the_heart} using certain constructions with monads on dualizable $t$-categories. For this we need to recall some preliminary notions.

%First we observe that under a suitable condition on a monad the category of modules has a natural $t$-structure.

\begin{remark}It is easy to see that for a $(-n)$-coconnective functor $F:\cC\to\cD$ the condition on maps between $\Ext$ groups automatically holds for all $x,y\in\cC^{\heartsuit}.$ On the other hand, if $\cC$ is compactly generated, then it suffices to check this condition for $x,y\in\cC^{\omega,\heartsuit}.$ These assertions follow for example from Proposition \ref{prop:-n_coconnectivity_via_monads} below.\end{remark}

We first interpret the $(-n)$-coconnectivity condition via monads.

\begin{prop}\label{prop:-n_coconnectivity_via_monads}
	Let $\cC$ be a dualizable $t$-category. Let $A\in\Alg_{\bE_1}(\Fun^L(\cC,\cC))$ be a continuous exact monad. Suppose that the functor $\Cone(\Id_{\cC}\to A)[1]$ is left $t$-exact, so by Proposition \ref{cor:t_structures_on_modules} there is a unique compactly assembled continuously bounded $t$-structure on $\Mod_A(\cC)$ such that the natural functor $\cC\to\Mod_A(\cC)$ is $t$-exact. For $n\geq 1$ the following are equivalent.
	
	\begin{enumerate}[label=(\roman*),ref=(\roman*)]
		\item The functor $\cC\to\Mod_A(\cC)$ is $(-n)$-coconnective. \label{-n_coconnectivity}
		\item The functor $\Cone(\Id_{\cC}\to A)[n]:\cC\to\cC$ is left $t$-exact. \label{shift_by_n_left_t_exact} 
	\end{enumerate}
\end{prop}

\begin{proof}
The implication \Implies{shift_by_n_left_t_exact}{-n_coconnectivity} is immediate by adjunction.

\Implies{-n_coconnectivity}{shift_by_n_left_t_exact}. Again, by adjunction the functor $\Cone(\Id_{\cC}\to A)[n]$ takes $\cC^{\omega_1,\heartsuit}$ to $\cC_{\leq 0}.$ Here we of course used that $\cC_{\geq 0}$ is generated by $\cC^{\omega_1,\heartsuit}$ via extensions and colimits. By Proposition \ref{prop:left_t_exactness_suffices_for_the_heart} we conclude that this functor is left $t$-exact.
\end{proof}

\begin{cor}\label{cor:equiv_on_C_0_n}
	Let $F:\cC\to\cD$ be a $(-n)$-coconnective functor between dualizable $t$-categories, where $n\geq 2.$ Then $F$ induces an equivalence of exact categories $\cC_{[0,n-2]}\xto{\sim} \cD_{[0,n-2]}.$ 
\end{cor}

\begin{proof}
	By Proposition \ref{prop:t_structure_on_dualizable_via_functor_from_abelian} \ref{when_A_equivalent_to_heart_of_C} the functor $F$ induces an equivalence $\cC^{\heartsuit}\xto{\sim}\cD^{\heartsuit}.$ By Proposition \ref{prop:-n_coconnectivity_via_monads} the functor $\Cone(\Id_{\cC}\to F^R\circ F)$ takes $\cC_{\leq 0}$ to $\cC_{\leq -n},$ hence it takes $\cC_{[0,n-2]}$ to $\cC_{\leq -2}.$ It follows by adjunction that the functor $\cC_{[0,n-2]}\to \cD_{[0,n-2]}$ is fully faithful, and it induces isomorphisms on $\Ext^1.$ To see that it is essentially surjective, note that its essential image is closed under extensions, and it contains $\cD^{\heartsuit}[i]$ for $0\leq i\leq n-2.$ Therefore, this essential image coincides with $\cD_{[0,n-2]}$ and we have an equivalence of categories $\cC_{[0,n-2]}\xto{\sim}\cD_{[0,n-2]}.$
	
	It remains to observe that this equivalence automatically reflects exactness: in both categories the short exact sequences are precisely all fiber-cofiber sequences.
\end{proof}

We now recall the machinery of deformed tensor algebras, which will be crucial for our arguments. We refer to \cite[Section 1.9]{E25c} for this notion in an abstract setting. Here we need the following special case. Let $\cC$ be a dualizable category. We consider the functor
\begin{equation*}
T^{\deff}(-):\Fun^L(\cC,\cC)_{/\Sigma_{\cC}}\to \Alg_{\bE_1}(\Fun^L(\cC,\cC)),
\end{equation*}
which is left adjoint to the composition
\begin{equation*}
\Alg_{\bE_1}(\Fun^L(\cC,\cC))\to \Alg_{\bE_0}(\Fun^L(\cC,\cC))\simeq \Fun^L(\cC,\cC)_{\Id_{\cC}/}\simeq \Fun^L(\cC,\cC)_{/\Sigma_{\cC}}.
\end{equation*}

We will typically write $T^{\deff}(F),$ assuming that $F:\cC\to\cC$ is a continuous exact functor with a chosen morphism $F\to\Sigma_{\cC}.$ Note that by adjunction we have a natural morphism of functors
\begin{equation*}
\Fiber(F\to\Sigma_{\cC})\to T^{\deff}(F).
\end{equation*} We consider $T^{\deff}(F)$ as a continuous exact monad on $\cC.$ 

We will use the natural non-negative multiplicative filtration $\Fil_{\bullet}$ on $T^{\deff}(F)$ as explained in \cite[Section 1.9]{E25c}. Its associated graded is the (non-deformed) tensor algebra $T(F).$ Note that we have $\Fil_0 T^{\deff}(F)\cong \Id_{\cC}$ and $\Fil_1 T^{\deff}(F)\cong \Fiber(F\to\Sigma_{\cC}).$ 

To deal with monads which are deformed tensor algebras we will use the key short exact sequence given by the following proposition. We recall the following notation: for an accessible exact functor $F:\cD\to\cC$ between presentable stable categories we denote by $\cC\oright_F \cD$ the category of triples $(x,y,\varphi),$ where $x\in\cC,$ $y\in\cD$ and $\varphi:x\to F(y).$ By \cite[Proposition 1.2]{E24} this category is also presentable stable. Moreover, if $\cC$ and $\cD$ are dualizable and $F$ is continuous, then the category $\cC\oright_F\cD$ is also dualizable by \cite[Proposition 1.78]{E24}.

\begin{prop}\label{prop:ses_for_deformed}
Let $\cC$ be a dualizable category, and let $F\in\Fun^L(\cC,\cC)$ be a continuous exact functor with a morphism $f:\Id_{\cC}\to F.$ 
\begin{enumerate}[label=(\roman*),ref=(\roman*)]
	\item We have a short exact sequence in $\Cat_{\st}^{\dual}:$
	\begin{equation*}
		0\to \cD\xto{\Psi} \cC\oright_F \cC\xto{\Phi} \Mod_{T^{\deff}(\Cone(f))}(\cC)\to 0.
	\end{equation*}
	Here $\cD\subset \cC\oright_F \cC$ is generated as a localizing subcategory by the image of the strongly continuous functor
	\begin{equation}\label{eq:Theta_generating_the_kernel}
		\Theta:\cC\to \cC\oright_F \cC,\quad x\mapsto (x,x,f_x).
\end{equation}
We consider $\Theta$ as a (strongly continuous) functor from $\cC$ to $\cD.$ \label{ses_for_deformed}
\item We denote by $i_1,i_2:\cC\to \cC\oright_F \cC$ the (strongly continuous) inclusion functors
\begin{equation*}
i_1(x)=(x[-1],0,0),\quad i_2(x)=(0,x,0),
\end{equation*}
and denote by $i_1^L$ resp. $i_2^R$ the (strongly continuous) left adjoint to $i_1$ resp. right adjoint to $i_2$ Then the functors $\Phi\circ i_1,\Phi\circ i_2:\cC\to \Mod_{T^{\deff}(\Cone(f))}(\cC)$ are isomorphic, and they are left adjoint to the forgetful functor.
The functors $i_1^L[-1]\circ\Psi$ and $i_2^R\circ\Psi$ are left inverses to the functor $\Theta:\cC\to\cD.$ \label{left_inverse}
\item The adjunction counit $\Id_{\cC}\to \Theta^R\circ\Theta$ is a split monomorphism, its left inverse is given by
\begin{equation}\label{eq:left_inverse_to_counit}
\Theta^R\circ\Theta\xto{\Theta^R\circ\eta\circ\Theta} \Theta^R\circ(i_2^R\circ\Psi)^R\circ i_2^R\circ\Psi\circ\Theta\cong \Id_{\cC}.
\end{equation}
This gives a direct sum decomposition in $\Fun^L(\cC,\cC):$
\begin{equation*} \Theta^R\circ \Theta\cong \Id_{\cC}\oplus \Fiber(f).
\end{equation*} \label{counit_splits}
\end{enumerate}
\end{prop}

\begin{proof}
We prove \ref{ses_for_deformed} using the universal property of the monad $T^{\deff}(\Cone(f)).$ First, by \cite[Proposition 1.78]{E24} the functors $i_1$ and $i_2$ give a semi-orthogonal decomposition in $\Cat_{\st}^{\dual},$ i.e. $\cC\oright_F \cC=\la i_1(\cC), i_2(\cC)\ra.$ By \cite[Proposition 1.79]{E24} the functor $\Theta$ from \eqref{eq:Theta_generating_the_kernel} is strongly continuous. Hence, by \cite[Corollary 1.57]{E24} the category $\cD$ (as defined in \ref{ses_for_deformed}) is dualizable and the inclusion functor $\Psi:\cD\to\cC\oright_F \cC$ is strongly continuous. We put $\cE=(\cC\oright_F \cC)/\Psi(\cD),$ and consider the (strongly continuous) composition
\begin{equation*}
G:\cC\xto{i_2}\cC\oright_F \cC \to\cE.
\end{equation*}    
By construction, $\cE$ is dualizable, the functor $G$ is strongly continuous and $G(\cC)$ generates $\cE$ as a localizing subcategory (since $\Psi(\cD)$ and $i_2(\cC)$ generate $\cC\oright_F \cC$). Hence, $(\cE,G)\in(\Cat_{\st}^{\dual})_{\cC/}$ is contained in the essential image of \eqref{eq:from_monads_to_undercategory}. For any other pair $(\cE',G')\in (\Cat_{\st}^{\dual})_{\cC/}$ using \cite[Proposition 1.79]{E24} we obtain a natural equivalence of spaces:
\begin{multline*}
\Map_{(\Cat_{\st}^{\dual})_{\cC/}} ((\cE,G),(\cE',G'))\simeq \Map_{\Fun^L(\cC,\cC)_{\Id_{\cC}/}}(F,G'^R\circ G')\\
\simeq \Map_{\Alg_{\bE_0}(\Fun^L(\cC,\cC))}(F,G'^R\circ G').
\end{multline*}
Hence, the fully faithfulness of \eqref{eq:from_monads_to_undercategory} implies an equivalence
\begin{equation*}
\cE\simeq\Mod_{T^{\deff}(\Cone(f))}(\cC).
\end{equation*}
This proves \ref{ses_for_deformed} and also the first assertion of \ref{left_inverse}. Namely, we see that $\Phi\circ i_2$ is left adjoint to the forgetful functor, and the compositions $\Phi\circ i_1$ and $\Phi\circ i_2$ are isomorphic by the construction of $\cD\subset\cC\oright_F \cC.$ The second assertion of \ref{left_inverse} follows from the definitions.

It remains to prove \ref{counit_splits}. Since we have an isomorphism $i_2^R\circ \Psi\circ\Theta\cong \Id_{\cC},$ the map \eqref{eq:left_inverse_to_counit} is indeed a left inverse to the adjunction counit.  The direct sum decomposition is obtained by a direct computation:
\begin{equation*}
\Theta^R\circ\Theta \cong \Fiber(\Id_{\cC}\oplus \Id_{\cC}\xto{(f,f)}F)\cong \Id_{\cC}\oplus \Fiber(f).\qedhere
\end{equation*}
\end{proof}

We apply this general construction to dualizable $t$-categories.

\begin{prop}\label{prop:t_structures_for_deformed_tensor_algebra}
	Let $\cC,$ $F$ and $f:\Id_{\cC}\to F$ be as in Proposition \ref{prop:ses_for_deformed}, and we keep the notation from therein. Suppose that $(\cC_{\geq 0},\cC_{\leq 0})$ is a compactly assembled continuously bounded $t$-structure on $\cC.$ Suppose that for some $n\geq 1$ the functor $\Cone(f)[n]$ is left $t$-exact. Then the categories $\Mod_{T^{\deff}(\Cone(f))}(\cC)$ and $\cD$ have natural compactly assembled continuously bounded $t$-structures such that the functor $\cC\to \Mod_{T^{\deff}(\Cone(f))}(\cC)$ is $(-n)$-coconnective and the functor $\Theta:\cC\to\cD$ is $(-n-1)$-coconnective. 
\end{prop}

\begin{proof}
	This follows almost directly from Propositions \ref{prop:-n_coconnectivity_via_monads} and \ref{prop:ses_for_deformed}. Namely, we need to see that the corresponding monads on $\cC$ satisfy the required assumptions. This is clear for $\Theta,$ since we have $\Cone(\Id_{\cC}\to\Theta^R\circ\Theta)[n+1]\cong \Fiber(f)[n+1]\cong \Cone(f)[n].$ Next, the endofunctor $\Cone(\Id_{\cC}\to T^{\deff}(\Cone(f)))$ has an exhaustive increasing filtration with subquotients $\Cone(f)^{\circ m},$ $m\geq 1.$ By assumption, $\Cone(f)^{\circ m}[mn]$ is left $t$-exact, hence $\Cone(f)^{\circ m}[n]$ is left $t$-exact for $m\geq 1.$ We conclude that $\Cone(\Id_{\cC}\to T^{\deff}(\Cone(f)))[n]$ is left $t$-exact. This proves the proposition.
\end{proof}

The following general construction gives a way to reduce certain questions about (continuous exact) monads to the case of a deformed tensor algebra.

\begin{prop}\label{prop:approx_by_deformed_tensor_algebras}
Let $\Phi:\cC\to\cD$ be a strongly continuous exact functor between dualizable categories, and suppose that $\Phi(\cC)$ generates $\cD$ as a localizing subcategory. We define a functor $\N\to (\Cat_{\st}^{\dual})_{/\cD},$ $n\mapsto (\cC_n,\Phi_n),$ inductively as follows. We put $(\cC_0,\Phi_0)=(\cC,\Phi),$ and for $n\geq 0$ we define
\begin{equation*}
\cC_{n+1}=\Mod_{T^{\deff}(\Cone(\Id_{\cC_n}\to \Phi_n^R\circ \Phi_n))}(\cC_n).
\end{equation*}
The functor $\Phi_{n+1}:\cC_{n+1}\to\cC$ is induced by the natural map of monads on $\cC_n,$ namely
\begin{equation*}
T^{\deff}(\Cone(\Id_{\cC_n}\to \Phi_n^R\circ \Phi_n))\to \Phi_n^R\circ \Phi_n.
\end{equation*}
Then we have an equivalence
\begin{equation*}
\indlim[n]\cC_n\simeq \cD.
\end{equation*}
\end{prop}

\begin{proof}
This is in fact the same construction as in \cite[Proof of Proposition 3.15]{E25c} if we spell it out in terms of monads on $\cC.$ We give the details for completeness.

Denote the transition functors by $F_{nm}:\cC_n\to\cC_m,$ $n\leq m.$ By construction, $F_{0n}(\cC_0)$ generates $\cC_n$ for all $n\geq 0.$ Next, for $n\geq 0$ the morphism of functors $\Id_{\cC_n}\to F_{n,n+1}^R\circ F_{n,n+1}$ factors through $\Phi_n^R\circ\Phi_n.$ Indeed, by construction the monad $F_{n,n+1}^R\circ F_{n,n+1}$ is the deformed tensor algebra of $\Cone(\Id_{\cC_n}\to \Phi_n^R\circ \Phi_n).$ The term $\Fil_1$ of the standard filtration is identified with $\Phi_n^R\circ \Phi_n.$ Composing with $F_{0n}^R$ and precomposing with $F_{0n}$ we get a factorization in the category $\Fun^L(\cC_0,\cC_0):$
\begin{equation*}
F_{0n}^R\circ F_{0n}\to \Phi_0^R\circ \Phi_0\to F_{0,n+1}^R\circ F_{0,n+1}.
\end{equation*}
Moreover, the composition
\begin{equation*}
\Phi_0^R\circ \Phi_0\to F_{0,n+1}^R\circ F_{0,n+1}\to \Phi_0^R\circ\Phi_0
\end{equation*}
is homotopic to the identity. Hence, we get an isomorphism of monads
\begin{equation*}
\indlim[n] F_{0n}^R\circ F_{0n}\xto{\sim} \Phi_0^R\circ \Phi_0.
\end{equation*}
By assumption, $\Phi_0(\cC_0)$ generates $\cD,$ hence we obtain an equivalence $\indlim[n]\cC_n\simeq \cD.$ This proves the proposition.
\end{proof}

Again, we apply this abstract result to dualizable $t$-categories.

\begin{prop}\label{prop:approx_by_deformed_for_t_categories} We keep the notation and assumption from Proposition \ref{prop:approx_by_deformed_tensor_algebras}. Suppose that $\cC$ and $\cD$ are moreover dualizable $t$-categories and the functor $\Phi:\cC\to\cD$ is $(-n)$-coconnective for some $n\geq 1.$ There are compactly assembled continuously bounded $t$-structures on the categories $\cC_k,$ $k\geq 1,$ which are uniquely determined by the requirement that all the transition functors are $t$-exact. Moreover, for each $k\geq 0$ the functor $\Cone(\Id_{\cC_k}\to\Phi_k^R\circ\Phi_k)[2^k(n-1)+1]:\cC_k\to\cC_k$ is left $t$-exact.
\end{prop}

\begin{proof}
We construct the $t$-structures inductively, simultaneously proving the left $t$-exactness assertion. We already have the $t$-structure on $\cC_0=\cC.$ By assumption on $\Phi_0=\Phi$ and by Proposition \ref{prop:-n_coconnectivity_via_monads}, the functor $\Cone(\Id_{\cC_0}\to \Phi_0^R\circ\Phi_0)[n]$ is left $t$-exact. 

Suppose that for some $l\geq 0$ we have the claimed $t$-structures on $\cC_k$ for $0\leq k\leq l,$ such that for these $k$ the functor $\Cone(\Id_{\cC_k}\to\Phi_k^R\circ\Phi_k)[2^k(n-1)+1]$ is left $t$-exact. Applying Proposition \ref{prop:t_structures_for_deformed_tensor_algebra}, we obtain a unique compactly assembled continuously bounded $t$-structure on $\cC_{l+1}$ such that the transition functor $F_{l,l+1}:\cC_l\to\cC_{l+1}$ is $t$-exact. We need to prove that the functor 
\begin{equation*}
\Cone(\Id_{\cC_{l+1}}\to \Phi_{l+1}^R\circ \Phi_{l+1})[2^{l+1}(n-1)+1]
\end{equation*} is left $t$-exact. By Proposition \ref{prop:left_t_exactness_suffices_for_composition} this is equivalent to the left $t$-exactness of the functor $\Cone(F_{l,l+1}^R\circ F_{l,l+1}\to \Phi_l^R\circ\Phi_l)[2^{l+1}(n-1)+1].$ By construction, the monad $F_{l,l+1}^R\circ F_{l,l+1}$ has a non-negative increasing filtration $\Fil_{\bullet},$ such that $\Fil_1\cong \Phi_l^R\circ \Phi_l.$ The inclusion of $\Fil_1$ is a right inverse to the morphism $F_{l,l+1}^R\circ F_{l,l+1}\to \Phi_l^R\circ\Phi_l.$ Hence, it suffices to show that for $m\geq 2$ the functor $\gr_m \Fil_{\bullet}[2^{l+1}(n-1)+2]$ is left $t$-exact. This follows from the induction hypothesis: we have
\begin{multline*}
\gr_m \Fil_{\bullet}[2^{l+1}(n-1)+2]\cong \Cone(\Id_{\cC_l}\to \Phi_l^R\circ\Phi_l)^{\circ m}[2^{l+1}(n-1)+2]\\
\cong (\Cone(\Id_{\cC_l}\to \Phi_l^R\circ\Phi_l)[2^l(n-1)+1])^{\circ m}[(2-m)(2^l(n-1)+1)],
\end{multline*} and the latter functor is left $t$-exact since $\Cone(\Id_{\cC_l}\to \Phi_l^R\circ\Phi_l)[2^l(n-1)+1]$ is left $t$-exact and $m\geq 2.$ This proves the induction step and the proposition. 
\end{proof}

\begin{proof}[Proof of Theorem \ref{th:coconnectivity_estimates}]
We first consider the following special case. Let $F:\cC\to\cD$ be as in the theorem, and suppose that moreover the monad $F^R\circ F$ is isomorphic to $T^{\deff}(G),$ where $G\in\Fun^L(\cC,\cC)$ is equipped with a morphism $G\to \Sigma_{\cC}$ and the functor $G[n]$ is left $t$-exact. Put $H = \Fiber(G\to\Sigma_{\cC}).$ Applying the construction from Proposition \ref{prop:ses_for_deformed} (and changing the notation for the categories), we obtain a short exact sequence in $\Cat_{\st}^{\dual}$
\begin{equation}\label{eq:ses_for_deformed_tensor_algebra}
0\to \cE\xto{\Psi}\cC\oright_H \cC\xto{\Phi} \cD\to 0. 
\end{equation}
As above, we denote by $i_1,i_2:\cC\to\cC\oright_H \cC.$ the natural inclusions, and consider the adjoints $i_1^L$ and $i_2^R.$ Applying $K^{\cont}(-)$ to \eqref{eq:ses_for_deformed_tensor_algebra}, we obtain the cofiber sequence
\begin{equation*}
K^{\cont}(\cE)\to K^{\cont}(\cC)\oplus K^{\cont}(\cC)\to K^{\cont}(\cD).
\end{equation*}
This gives an equivalence of spectra
\begin{equation}\label{eq:equivalence_of_cones}
\Cone(K^{\cont}(\cC)\xto{K^{\cont}(F)}  K^{\cont}(\cD))\cong \Cone(K^{\cont}(\cE)\xto{K^{\cont}(i_2^R\circ \Psi)} K^{\cont}(\cC)).
\end{equation}
Now, recall that by Proposition \ref{prop:ses_for_deformed} we have a functor $\Theta:\cC\to\cE,$ right inverse to $i_2^R\circ \Psi.$ Moreover, by Proposition \ref{prop:t_structures_for_deformed_tensor_algebra} $\cE$ is naturally a $t$-category and the functor $\Theta$ is $(-n-1)$-coconnective. By Corollary \ref{cor:equiv_on_C_0_n} $\Theta$ induces an equivalence $\cC_{[0,n-1]}\xto{\sim} \cE_{[0,n-1]}.$ By Theorem \ref{th:theorem_of_the_heart} the latter equivalence of exact categories implies that $\Theta$ induces an equivalence $K_{\geq -n-1}^{\cont}(\cC)\xto{\sim} K_{\geq -n-1}^{\cont}(\cE).$ Hence the left inverse functor $i_2^R\circ \Psi$ also induces an equivalence on $K^{\cont}_{\geq -n-1}.$ We conclude from \eqref{eq:equivalence_of_cones} that $\Cone(K^{\cont}(\cC)\to K^{\cont}(\cD))\in \Sp_{\leq -n-1}.$

Now consider the general situation of the theorem. By Proposition \ref{prop:approx_by_deformed_for_t_categories} we have a sequence $\cC=\cC_0\to\cC_1\to\dots$ of dualizable $t$-categories such that $\cD\simeq\indlim[k]\cC_k,$ each transition functor $F_k:\cC_k\to \cC_{k+1}$ is $t$-exact, $\cC_{k+1}$ is generated by $F_k(\cC_k)$ and the monad $F_k^R\circ F_k$ is the deformed tensor algebra of a functor $G_k:\cC_k\to \cC_k$ such that $G_k[n]$ is left $t$-exact. By the above special case, we know that the spectrum $\Cone(K^{\cont}(F_k))$ is $(-n-1)$-coconnective for $k\geq 0.$ Hence, also the spectra $\Cone(K^{\cont}(\cC_0)\to K^{\cont}(\cC_k))$ are $(-n-1)$-coconnective. Taking the colimit over $k,$ we conclude that $\Cone(K^{\cont}(\cC)\to K^{\cont}(\cD))$ is $(-n-1)$-coconnective. This proves the theorem.  
\end{proof}

\subsection{Estimates for higher nil groups}

For the proof of the theorem of the heart for $KH$ it will be important to have a version of Theorem \ref{th:coconnectivity_estimates} for higher nil groups $N^s K_j^{\cont}.$ We will need the following immediate application of Proposition \ref{prop:t_structure_on_nilpotent_endomorphisms}.
%In the following corollary for convenience we put $N^0 K = K$ for convenience.

\begin{prop}\label{prop:t_structure_on_nilpotent_endomorphisms_coconnectivity}
	Let $\cC$ and $\cD$ be a dualizable $t$-categories, and let $F:\cC\to\cD$ be a $(-n)$-coconnective functor for some $n\geq 1.$ Let $y$ be a formal variable of degree $0.$ Then the functor
	\begin{equation}\label{eq:id_otimes_F}
		\Id\otimes F:(\Mod_{y\hy\tors}\hy \bS[y])\otimes \cC\to (\Mod_{y\hy\tors}\hy \bS[y])\otimes \cD
	\end{equation} 
	is also $(-n)$-coconnective, where the source and the target are equipped with $t$-structures from Proposition \ref{prop:t_structure_on_nilpotent_endomorphisms}.
		% and consider the (compactly generated) category $\Mod_{y\hy\tors}\hy \bS[y]$ of $\bS[y]$-modules such that $y$ acts locally nilpotently on the homotopy groups. Denote by $i:\Sp\to \Mod_{y\hy\tors}\hy \bS[y]$ the restriction of scalars functor for $\bS[y]\to\bS,$ $y\mapsto 0.$ 
	%\begin{enumerate}[label=(\roman*),ref=(\roman*)]
%		\item The category $(\Mod_{y\hy\tors}\hy \bS[y])\otimes \cC$ has a unique compactly assembled continuously bounded $t$-structure such that the functor $i\otimes\cC:\cC\to (\Mod_{y\hy\tors}\hy \bS[y])\otimes \cC$ is $t$-exact. \label{t_structure_for_nilpotent_endomorphisms}
%		\item Let $\cD$ be another dualizable $t$-category and let $F:\cC\to\cD$ be a $(-n)$-coconnective functor for some $n\geq 1.$ Then the functor 
%		\begin{equation}\label{eq:id_otimes_F}
%%		\end{equation} 
%		is also $(-n)$-coconnective. \label{preservation_of_coconnectivity_for_nil_endomorphsisms}
%	\end{enumerate}
\end{prop}

\begin{proof}
	By Proposition \ref{prop:-n_coconnectivity_via_monads}, the $(-n)$-coconnectivity of $F$ means that the functor $\Cone(\Id_{\cC}\to F^R\circ F)[n]$ is left $t$-exact, and we need to prove that the same holds for the adjunction counit of the functor \eqref{eq:id_otimes_F}. As in Proposition \ref{prop:t_structure_on_nilpotent_endomorphisms} we consider the functor $i:\Sp\to\Mod_{y\hy\tors}\bS[y]$ (restriction of scalars for $y\mapsto 0$). Consider the composition 
\begin{equation*}
		\Phi:\cC\xto{i\otimes\cC}(\Mod_{y\hy\tors}\bS[y])\otimes\cC\xto{\Id\otimes F} (\Mod_{y\hy\tors}\bS[y])\otimes\cD.
\end{equation*}
	By Proposition \ref{prop:left_t_exactness_suffices_for_composition} it suffices to show that the functor
	$\Cone((i\otimes\cC)^R\circ (i\otimes\cC)\to\Phi^R\circ\Phi)[n]$ is left $t$-exact. This follows from the isomorphism
	\begin{multline*}
		\Cone((i\otimes\cC)^R\circ (i\otimes\cC)\to\Phi^R\circ\Phi)[n]\\
		\cong \Cone(\Id_{\cC}\to F^R\circ F)[n]\oplus \Cone(\Id_{\cC}\to F^R\circ F)[n-1].\qedhere
	\end{multline*}
\end{proof}

We recall that for $\cC\in\Cat^{\perf}$ and for $s\geq 1$ the higher nil spectrum $N^s K(\cC)$ is defined as $K(\cC\otimes\Perf(\bS[x_1,\dots,x_s]))^{\red},$ where we take the reduced part in the simplicial sense and we consider the assignment $[s]\mapsto \bS[x_1,\dots,x_s]$ as a simplicial $\bE_{\infty}$-ring as explained in \cite[Section 8]{E25c}. On the level of homotopy groups we have
\begin{multline*}
N^s K_j(\cC)\cong \coker(\biggplus[1\leq i\leq s]K_j(\cC\otimes\Perf(\bS[x_1,\dots,x_{i-1},x_{i+1},\dots,x_s]))\\
\to K_j(\cC\otimes\Perf(\bS[x_1,\dots,x_s]))),
\end{multline*}
see \cite{Wei89} for the $\Z$-linear version.

As usual, we denote by $N^s K^{\cont}$ the associated localizing invariant of dualizable categories.

\begin{cor}\label{cor:coconnectivity_estimates_for_N^s K}
	Let $F:\cC\to\cD$ be a $(-n)$-coconnective functor between dualizable $t$-categories, where $n\geq 1.$ Let $s\geq 1$ be an integer. Then the induced map $N^s K_j^{\cont}(\cC)\to N^s K_j^{\cont}(\cD)$ is an isomorphism for $j\geq -n+s,$ and a monomorphism for $j=-n+s-1.$ %Hence, the spectrum $\Cone(U_k^{\cont}(F))$ is also $(-n+2k-1)$-coconnective for $k\geq 0.$ 
\end{cor}

\begin{proof}
	Recall that we have the identification of reduced localizing motives $\wt{\cU}_{\loc}(\bS[x])\cong \Sigma \wt{\cU}_{\loc}(\Perf_{\infty}(\PP_{\bS}^{1,\flat}))$ which follows from the Beilinson's full exceptional collection $\la\cO(-1),\cO \ra$ on $\PP_{\bS}^{1,\flat}$ \cite{Bei78}, see for example \cite[Proof of Theorem 9.1]{E25c}. This isomorphism implies that $N^s K^{\cont}(\cC)$ is naturally a direct summand of $K^{\cont}((\Mod_{x^{-1}\hy\tors}\hy \bS[x^{-1}])^{\otimes s}\otimes \cC)[s],$ and similarly for $\cD.$ Hence, the spectrum $\Cone(N^s K^{\cont}(F))$ is a direct summand of the spectrum
	\begin{equation}\label{eq:cone_after_tensoring_by_s_th_power}
		\Cone(K^{\cont}((\Mod_{x^{-1}\hy\tors}\hy \bS[x^{-1}])^{\otimes s}\otimes \cC)\to K^{\cont}((\Mod_{x^{-1}\hy\tors}\hy \bS[x^{-1}])^{\otimes s}\otimes \cD))[s].
	\end{equation}
	By the inductive application of Propositions \ref{prop:t_structure_on_nilpotent_endomorphisms} and \ref{prop:t_structure_on_nilpotent_endomorphisms_coconnectivity}, we see that the tensor products $(\Mod_{x^{-1}\hy\tors}\hy \bS[x^{-1}])^{\otimes s}\otimes \cC$ and $(\Mod_{x^{-1}\hy\tors}\hy \bS[x^{-1}])^{\otimes s}\otimes \cD$ are naturally $t$-categories, and the functor
	\begin{equation*}
		(\Mod_{x^{-1}\hy\tors}\hy \bS[x^{-1}])^{\otimes s}\otimes \cC\to (\Mod_{x^{-1}\hy\tors}\hy \bS[x^{-1}])^{\otimes s}\otimes \cD
	\end{equation*}
	is $(-n)$-coconnective. Applying Theorem \ref{th:coconnectivity_estimates} again, we see that the spectrum \eqref{eq:cone_after_tensoring_by_s_th_power} is $(-n+s-1)$-coconnective. 
\end{proof}

We record another immediate application of the above results.

\begin{cor}\label{cor:coconnectivity_for_N^sK_of_t_categories}
Let $\cC$ be a dualizable $t$-category. Then we have
\begin{equation*}
N^s K_j^{\cont}(\cC) = 0\quad\text{for }s\geq 1,\,j\geq s-1.
\end{equation*}
\end{cor}

\begin{proof}
Put $\cD = (\Mod_{x^{-1}\hy\tors}\hy\bS[x^{-1}])\otimes \cC.$ As in Proposition \ref{prop:t_structure_on_nilpotent_endomorphisms} consider the strongly continuous functor $i:\Sp\to\Mod_{x^{-1}\hy\tors}\hy\bS[x^{-1}]$ and put $F=i\otimes\cC:\cC\to\cD.$ By loc. cit. $\cD$ is naturally a dualizable $t$-category and the functor $F$ is $(-1)$-coconnective. Applying Corollary \ref{cor:coconnectivity_estimates_for_N^s K} and using again the isomorphism $\wt{\cU}_{\loc}(\bS[x])\cong \Sigma \wt{\cU}_{\loc}(\Perf_{\infty}(\PP_{\bS}^{1,\flat})),$ we obtain that for $s\geq 1$ the spectrum
\begin{equation*}
N^s K^{\cont}(\cC)\cong \Sigma\Cone(N^{s-1}K^{\cont}(F))
\end{equation*}
is $(s-2)$-coconnective. This proves the corollary.
\end{proof}

\section{Theorem of the heart for $KH$}
\label{sec:th_of_the_heart_for_KH}

We prove the following result.

\begin{theo}\label{th:theorem_of_the_heart_for_KH}
Let $\cC$ be a dualizable $t$-category. Then the realization functor induces an equivalence of spectra $KH^{\cont}(\cC^{\heartsuit})\xto{\sim} KH^{\cont}(\cC).$
\end{theo}

To prove Theorem \ref{th:theorem_of_the_heart_for_KH} we will deal with the natural filtration on $KH.$ For $k\geq 0$ we consider the localizing invariant $U_k=\Fil_k KH:\Cat^{\perf}\to \Sp,$ i.e.
\begin{equation}\label{eq:definition_of_U_k}
U_k(\cC) = \indlim[\substack{[n]\in \Delta_{\leq k}^{op}}] K(\cC\otimes \Perf(\bS[x_1,\dots,x_n])).
\end{equation}
These functors form a direct sequence $U_0\to U_1\to\dots.$ We have 
\begin{equation*}
U_0(\cC)\cong K(\cC),\quad \indlim[k]U_k(\cC)\cong KH(\cC),
\end{equation*}
and
\begin{equation*}
\Cone(U_k(\cC)\to U_{k+1}(\cC))\cong N^{k+1} K(\cC)[k+1],\quad k\geq 0,\,\cC\in\Cat^{\perf}. 
\end{equation*}
For convenience we put $N^0 K = K.$

Throughout this section we will systematically use the notion of a $(-n)$-coconnective functor from Definition \ref{def:-n_coconnective_functor}. Our first step is to deduce the coconnectivity estimates for $\Cone(U_k^{\cont}(F)),$ where $F$ is $(-n)$-coconnective for some $n\geq 1.$

\begin{cor}\label{cor:coconnectivity_estimates_for_U_k}
Let $F:\cC\to\cD$ be a $(-n)$-coconnective functor between dualizable $t$-categories, where $n\geq 1.$ Then for $k\geq 0$ the spectrum $\Cone(U_k^{\cont}(F))$ is $(-n+2k-1)$-coconnective.
\end{cor}

\begin{proof}
This follows from Theorem \ref{th:coconnectivity_estimates} and Corollary \ref{cor:coconnectivity_estimates_for_N^s K} since the spectrum $\Cone(U_k^{\cont}(F))$ has a finite filtration with subquotients $\Cone(N^s K^{\cont}(F))[s],$ $0\leq s\leq k.$
\end{proof}

To prove Theorem \ref{th:theorem_of_the_heart_for_KH} we will consider certain statements depending on four integer parameters (we do not claim that they hold for all values of the parameters, this is certainly not the case). Namely, for $n\geq 1,$ $c\in\Z$ and $0\leq k\leq l$ consider the statements

\begin{itemize}
	\item $A(n,c,k,l).$ If $F:\cC\to\cD$ is a $(-n)$-connective functor between dualizable $t$-categories, then the map of spectra
	\begin{equation*}
	\tau_{\geq c}\Cone(U_k^{\cont}(F))\to \tau_{\geq c}\Cone(U_l^{\cont}(F))
	\end{equation*}
	is null-homotopic.
	\item $B(n,c,k,l).$ Let $\cC$ be a dualizable $t$-category, and suppose that $G:\cC\to\cC$ is a continuous exact endofunctor with a map $G\to\Sigma_{\cC},$ such that $G[n]$ is left $t$-exact. Consider the functor $F:\cC\to\Mod_{T^{\deff}(G)}(\cC).$ Then the map of spectra
	\begin{equation*}
		\tau_{\geq c}\Cone(U_k^{\cont}(F))\to \tau_{\geq c}\Cone(U_l^{\cont}(F))
	\end{equation*} 
	is null-homotopic.
\end{itemize}

Note that $B(n,c,k,l)$ is a special case of $A(n,c,k,l)$ by Proposition \ref{prop:t_structures_for_deformed_tensor_algebra}. The following lemma is crucial for the proof of Theorem \ref{th:theorem_of_the_heart_for_KH}.

\begin{lemma}\label{lem:A_implies_B}
Let $n\geq 1,$ $c\in\Z$ and $0\leq k<l.$ Then the statement $A(2n+1,c-1,k,l-1)$ implies the statement $B(n,c,k,l).$
\end{lemma}

To prove this implication we use a general observation on monads which are (non-deformed) tensor algebras. 

\begin{prop}\label{prop:A^1-action_on_tensor_algebra}
Let $\cC$ be a dualizable category and let $G:\cC\to\cC$ be a continuous exact endofunctor. Consider the monad $T(G),$ put $\cD=\Mod_{T(G)}(\cC),$ and denote by $F:\cC\to\cD$ the left adjoint to the forgetful functor. Then for any $k\geq 0$ the map $\Cone(U_k^{\cont}(F))\to\Cone(U_{k+1}^{\cont}(F))$ is null-homotopic.   
\end{prop}

\begin{proof}
It is convenient to prove a stronger statement on localizing motives. Namely, for $k\geq 0$ we put $M_k=\indlim[\substack{[n]\in\Delta_{\leq k}^{op}}] \cU_{\loc}(\Delta^n)\in\Mot^{\loc}.$ It suffices to prove that the map
\begin{equation*}
\Cone(\cU_{\loc}^{\cont}(F))\otimes M_k\to \Cone(\cU_{\loc}^{\cont}(F))\otimes M_{k+1}
\end{equation*}
is null-homotopic in $\Mot^{\loc}.$ To see this, we first recall that by \cite[Proofs of Theorem 8.1 and Proposition 8.2]{E25c} the map
\begin{equation*}
\wt{\cU}_{\loc}(\bS[x])\otimes M_k\to \wt{\cU}_{\loc}(\bS[x])\otimes M_{k+1}
\end{equation*}
is null-homotopic. Thus, it suffices to show that the object $\Cone(\cU_{\loc}^{\cont}(F))$ is a retract of the object $\wt{\cU}_{\loc}(\bS[x])\otimes \Cone(\cU_{\loc}^{\cont}(F)).$ This follows from the standard argument: informally speaking, the non-negative grading on the tensor algebra $T(G)$ corresponds to the multiplicative $\A^{1,\flat}_{\bS}$-action. More precisely, consider the map of monads on $\cC,$
\begin{equation}\label{eq:coaction_on_tensor_algebra}
T(G)\to \bS[x]\otimes T(G)\quad\text{in }\Alg_{\bE_1}(\Fun^L(\cC,\cC)),
\end{equation}
which by adjunction corresponds to the composition
\begin{equation*}
G\xto{x\otimes G}\bS[x]\otimes G\to \bS[x]\otimes T(G).
\end{equation*}
The map \eqref{eq:coaction_on_tensor_algebra} has a left inverse induced by the map of $\bE_1$-rings $\bS[x]\to\bS,$ $x\mapsto 1.$ Passing to the categories of modules over the monads, we get the retraction
\begin{equation*}
\Mod_{T(G)}(\cC)\to (\Mod\hy\bS[x])\otimes \Mod_{T(G)}(\cC)\to \Mod_{T(G)}(\cC). 
\end{equation*}
Passing to motives, we get a retraction
\begin{equation*}
\cU_{\loc}^{\cont}(\Mod_{T(G)}(\cC))\to \cU_{\loc}(\bS[x])\otimes \cU_{\loc}^{\cont}(\Mod_{T(G)}(\cC))\to \cU_{\loc}^{\cont}(\Mod_{T(G)}(\cC)).
\end{equation*}
Now, the object $\Cone(\cU_{\loc}^{\cont}(F))$ is naturally a direct summand of $\cU_{\loc}^{\cont}(\Mod_{T(G)}(\cC)).$ Passing to reduced motives, we obtain the desired retraction
\begin{equation*}
\Cone(\cU_{\loc}^{\cont}(F))\to \wt{\cU}_{\loc}(\bS[x])\otimes \Cone(\cU_{\loc}^{\cont}(F))\to \Cone(\cU_{\loc}^{\cont}(F)).
\end{equation*}
This proves the proposition.
\end{proof}

\begin{proof}[Proof of Lemma \ref{lem:A_implies_B}]
Suppose that $A(2n+1,c-1,k,l-1)$ holds. Let $G:\cC\to\cC$ and $F:\cC\to\cD$ be as in $B(n,c,k,l).$ Put $H=\Fiber(G\to\Sigma_{\cC}).$ Consider the short exact sequence from Proposition \ref{prop:ses_for_deformed}, with modified notation:
\begin{equation*}
0\to \cE\xto{\Psi} \cC\oright_H \cC\to\cD\to 0.
\end{equation*}
We denote by $i_1.i_2:\cC\to\cC\oright_H \cC$ the standard inclusions as in loc. cit., and denote by $i_2^R$ the right adjoint to $i_2.$ Arguing as in the proof of Theorem \ref{th:coconnectivity_estimates}, we obtain the isomorphisms
\begin{equation*}\Cone(U_j^{\cont}(F))\cong \Cone(U_j^{\cont}(i_2^R\circ\Psi)),\quad j\geq 0,
\end{equation*}
compatible with the transition maps. As in Proposition \ref{prop:ses_for_deformed}, we consider the functor $\Theta:\cC\to\cE,$ right inverse to $i_2^R\circ\Psi.$ Recall that $\Theta(\cC)$ generates $\cE$ by construction, and by part \ref{counit_splits} of loc.cit., the adjunction counit $\Id_{\cC}\to \Theta^R\circ\Theta$ is a split monomorphism, giving a direct sum decomposition $\Theta^R\circ\Theta\cong \Id_{\cC}\oplus G[-1].$ Consider the morphism of monads $T(G[-1])\to \Theta^R\circ\Theta,$ corresponding by adjunction to the inclusion of $G[-1]$ as a direct summand. Put $\cB=\Mod_{T(G[-1])}(\cC).$ Then the maps of monads $\Id_{\cC}\to T(G[-1])\to\Theta^R\circ\Theta\to\Id_{\cC}$ correspond to the chain of (strongly continuous exact) functors
\begin{equation*}
\cC\xto{\Lambda}\cB\xto{\Phi}\cE\xto{i_2^R\circ\Psi}\cC.
\end{equation*}
By Proposition \ref{prop:t_structures_for_deformed_tensor_algebra} and Proposition \ref{prop:t_structure_on_dualizable_via_functor_from_abelian} \ref{sufficient_for_t_exactness}, $\cB$ and $\cE$ are naturally dualizable $t$-categories and the functors $\Lambda,$ $\Phi$ and $i_2^R\circ\Psi$ are $t$-exact.

We claim that the functor $\Phi:\cB\to\cE$ is $(-2n-1)$-coconnective. By Proposition \ref{prop:-n_coconnectivity_via_monads} we need to check that the functor $\Cone(\Id_{\cB}\to \Phi^R\circ\Phi)[2n+1]$ is left $t$-exact. By Proposition \ref{prop:left_t_exactness_suffices_for_composition} it suffices to show that the functor $\Cone(\Lambda^R\circ\Lambda\to\Lambda^R\circ\Phi^R\circ\Phi\circ\Lambda)[2n+1]$ is left $t$-exact. This is clear by construction: the latter functor is isomorphic to
\begin{equation*}
(\biggplus[m\geq 2]G[-1]^{\circ m})[2n+2]\cong\biggplus[m\geq 2] G[n]^{\circ m}[(2-m)(n+1)], 
\end{equation*}  
and $G[n]$ is left $t$-exact by assumption.

For $j\geq 0$ consider the cofiber sequence
\begin{equation*}
\Cone(U_j^{\cont}(i_2^R\circ\Psi\circ\Phi))\to \Cone(U_j^{\cont}(i_2^R\circ\Psi))\to \Cone(U_j^{\cont}(\Phi))[1].
\end{equation*}
Applying $A(2n+1,c-1,k,l-1)$ to $\Phi,$ we see that the map
\begin{equation*}
\tau_{\geq c}(\Cone(U_k^{\cont}(\Phi))[1])\to \tau_{\geq c}(\Cone(U_{l-1}^{\cont}(\Phi))[1])
\end{equation*}
is null-homotopic. Hence, the map 
\begin{equation*}
\tau_{\geq c}\Cone(U_k^{\cont}(i_2^R\circ\Psi))\to \tau_{\geq c}\Cone(U_{l-1}^{\cont}(i_2^R\circ\Psi))
\end{equation*} factors through $\tau_{\geq c}\Cone(U_{l-1}^{\cont}(i_2^R\circ\Psi\circ\Phi)).$ Now the functor $i_2^R\circ\Psi\circ\Phi$ is left inverse to $\Lambda.$ Applying Proposition \ref{prop:A^1-action_on_tensor_algebra} we see that the map
\begin{equation*}
\Cone(U_{l-1}^{\cont}(i_2^R\circ\Psi\circ\Phi))\to \Cone(U_l^{\cont}(i_2^R\circ\Psi\circ\Phi))
\end{equation*}
is null-homotopic. We conclude that the map
\begin{equation*}
\tau_{\geq c}\Cone(U_k^{\cont}(i_2^R\circ\Psi))\to \tau_{\geq c}\Cone(U_l^{\cont}(i_2^R\circ\Psi))
\end{equation*}
is null-homotopic. This proves the implication $A(2n+1,c-1,k,l-1)\Rightarrow B(n,c,k,l).$
\end{proof}

Before proving Theorem \ref{th:theorem_of_the_heart_for_KH} we spell out a trivial sufficient condition for vanishing of a composition of maps between spectra equipped with finite filtrations. As usual, we denote by $[m]$ the totally ordered set $\{0,1,\dots,m\}$ for $m\geq 0.$

\begin{prop}\label{prop:composition_of_maps_between_filtered_spectra_vanishes}
Let $m\geq 1$ and $c$ be integers, and let $\Phi:[m]\times [m]\to\Sp$ be a functor. Suppose that the following conditions hold.
\begin{enumerate}[label=(\roman*),ref=(\roman*)]
	\item $\Phi(0,i)=0$ for $0\leq i\leq m.$ \label{first_column_vanishes}
	\item The map 
	\begin{equation*}
		\tau_{\geq c}\Cone(\Phi(m-i-1,i)\to \Phi(m-i,i))\to \tau_{\geq c}\Cone(\Phi(m-i-1,i+1)\to \Phi(m-i,i+1)) 
	\end{equation*} 
	is null-homotopic for $0\leq i\leq m-1.$ \label{from_row_i_to_row_i+1}
\end{enumerate}
Then the map $\tau_{\geq c}\Phi(m,0)\to\tau_{\geq c}\Phi(m,m)$ is null-homotopic. 
\end{prop}

\begin{proof}
We apply induction on $m.$ The base $m=1$ is tautological.

Let $m\geq 2$ and suppose that the statement holds for $m-1.$ Let $\Phi:[m]\times [m]\to \Sp$ be a functor satisfying the conditions \ref{first_column_vanishes} and \ref{from_row_i_to_row_i+1}. Applying \ref{from_row_i_to_row_i+1} to $i=0$ we see that the map $\tau_{\geq c}\Phi(m,0)\to \tau_{\geq c}\Phi(m,1)$ factors through $\tau_{\geq c}\Phi(m-1,1).$ Hence, it suffices to show that the map $\tau_{\geq c}\Phi(m-1,1)\to \tau_{\geq c}\Phi(m-1,m)$ is null-homotopic. This follows from the induction hypothesis applied to the functor
\begin{equation*}
\Psi:[m-1]\times [m-1]\to\Sp,\quad \Psi(i,j)=\Phi(i,j+1).\qedhere
\end{equation*}
\end{proof}

We now prove a vanishing result which is in fact stronger than Theorem \ref{th:theorem_of_the_heart_for_KH}.

\begin{prop}\label{prop:statements_S_m_and_T_m}
For $m\geq 0$ the following statements hold.
\begin{itemize}
	\item $(S_m):$ the statement $A(n,c,k,l)$ holds if 
	\begin{equation*}
	n\geq 1,\quad k\geq 0,\quad l-k\geq 2^m -1,\quad c\geq -2^m(n-1)+2k-1.
	\end{equation*}
	\item $(T_m):$ the statement $B(n,c,k,l)$ holds if
	\begin{equation*}
		n\geq 1,\quad k\geq 0,\quad l-k\geq 2^m,\quad c\geq -2^{m+1} n+2k.
	\end{equation*}
\end{itemize}
\end{prop}

\begin{proof}
We apply induction on $m.$ We already know that $(S_0)$ holds: this is exactly the assertion of Corollary \ref{cor:coconnectivity_estimates_for_U_k}. 

Next, for $m\geq 0$ we have the implication $(S_m)\Rightarrow (T_m).$ Indeed, suppose that $(S_m)$ holds and let $(n,c,k,l)$ be as in $(T_m).$ Then the quadruple $(2n+1,c-1,k,l-1)$ satisfies the assumptions of $(S_m),$ so $A(2n+1,c-1,k,l-1)$ holds. The latter statement implies $B(n,c,k,l)$ by Lemma \ref{lem:A_implies_B}. This proves that $(S_m)$ implies $(T_m).$ 

It remains to prove $(S_{m+1})$ assuming $(S_0)$ and $(T_i)$ for $0\leq i\leq m.$ Let $(n,c,k,l)$ be as in $(S_{m+1}).$ We need to prove that $A(n,c,k,l)$ holds. We may and will assume that $l=k+2^{m+1}-1$ and $c=-2^{m+1}(n-1)+2k-1.$ 

Let $F:\cC\to\cD$ be a $(-n)$-coconnective functor between dualizable $t$-categories. We apply the construction from Proposition \ref{prop:approx_by_deformed_tensor_algebras}. By Propositions \ref{prop:approx_by_deformed_for_t_categories} and \ref{prop:t_structures_for_deformed_tensor_algebra}, it gives a direct sequence of dualizable $t$-categories $\cC=\cC_0\to\cC_1\to\dots,$ such that $\indlim[i]\cC_i\simeq\cD$ and each functor $F_i:\cC_i\to\cD$ is $(-2^i(n-1)-1)$-coconnective. Moreover, if $F_{ij}:\cC_i\to\cC_j$ are the transition functors for $0\leq i\leq j,$ then $F_{0,i}(\cC_0)$ generates $\cC_i,$ and each monad $F_{i,i+1}^R\circ F_{i,i+1}$ is isomorphic to a deformed tensor algebra of some $G_i:\cC_i\to\cC_i$ such that $G_i[2^i(n-1)+1]$ is left $t$-exact.

Applying $(S_0)$ to the functor $F_{m+1}:\cC_{m+1}\to\cD,$ we see that $\tau_{\geq c}\Cone(U_k^{\cont}(F_{m+1}))=0.$ Hence, we have an equivalence
\begin{equation*}
\tau_{\geq c}\Cone(U_k^{\cont}(F_{0,m+1}))\xto{\sim} \tau_{\geq c}\Cone(U_k^{\cont}(F)).
\end{equation*}
Thus, it suffices to prove that the map
\begin{equation}\label{eq:null_homotopy_for_F_0_m+1}
\tau_{\geq c} \Cone(U_k^{\cont}(F_{0,m+1}))\to \tau_{\geq c}\Cone(U_l^{\cont}(F_{0,m+1}))
\end{equation}
is null-homotopic. We will apply Proposition \ref{prop:composition_of_maps_between_filtered_spectra_vanishes} to the following functor:
\begin{equation*}
\Phi:[m+1]\times [m+1]\to \Sp,\quad \Phi(i,j)=\Cone(U_{k+2^j-1}^{\cont}(F_{0,i})).
\end{equation*}
We need to check that the assumptions of loc. cit. hold. Clearly, we have $\Phi(0,i)=0$ for $0\leq i\leq m+1.$ It remains to see that for $0\leq i\leq m$ the map 
\begin{equation*}
\tau_{\geq c} \Cone(U_{k+2^i-1}^{\cont}(F_{m-i,m-i+1}))\to \tau_{\geq c}\Cone(U_{k+2^{i+1}-1}^{\cont}(F_{m-i,m-i+1}))
\end{equation*}
is null-homotopic. This follows from $(T_i).$ Indeed, the functor $G_{m-i}[2^{m-i}(n-1)+1]:\cC_{m-i}\to\cC_{m-i}$ is left $t$-exact, the monad $F_{m-i,m-i+1}^R\circ F_{m-i,m-i+1}$ is the deformed tensor algebra of $G_{m-i}$ and we have
\begin{equation*}
-2^{i+1}(2^{m-i}(n-1)+1)+2(k+2^i-1)=-2^{m+1}(n-1)+2k-2<c.
\end{equation*}
Applying Proposition \ref{prop:composition_of_maps_between_filtered_spectra_vanishes}, we obtain that the map \eqref{eq:null_homotopy_for_F_0_m+1} is null-homotopic, as required. This proves the proposition.
\end{proof}

\begin{proof}[Proof of Theorem \ref{th:theorem_of_the_heart_for_KH}]
The realization functor $F:\check{D}(\cC^{\heartsuit})\to\cC$ is $(-2)$-coconnective. It follows from Proposition \ref{prop:statements_S_m_and_T_m} that for any $k\geq 0$ and for any $c\in\Z$ there exists $l\geq k$ such that $A(2,c,k,l)$ holds. Namely, it suffices to take the smallest $m\geq 0$ such that $c\geq -2^m+2k-1,$ and put $l=k+2^m-1.$ Therefore, for any $c\in\Z$ we have
\begin{equation*}
\tau_{\geq c}\Cone(KH^{\cont}(F))\cong \indlim[k]\tau_{\geq c}\Cone(U_k^{\cont}(F))=0.
\end{equation*}
In other words, the map $KH^{\cont}(F)$ is an equivalence of spectra. This proves the theorem.
\end{proof}

\section{D\'evissage theorems for $K$-theory and homotopy $K$-theory}
\label{sec:devissage}

In this section we apply the results from Sections \ref{sec:coconnectivity_estimates} and \ref{sec:th_of_the_heart_for_KH} to deduce d\'evissage theorems for coherently assembled abelian categories. One way to approach this is by generalizing the construction from \cite[Theorem 0.3]{E25b}. However, we choose a slightly different but similar argument.

\begin{theo}\label{th:devissage_for_K_and_KH}
Let $\cA$ be a coherently assembled abelian category. Let $\cB\subset\cA$ be a coherently assembled abelian full subcategory, such that the inclusion functor is strongly continuous, and $\cB$ generates $\cA$ via extensions and filtered colimits. 

\begin{enumerate}[label=(\roman*),ref=(\roman*)]
\item The map $K_j^{\cont}(\cB)\to K_j^{\cont}(\cA)$ is an isomorphism for $j\geq -1,$ and a monomorphism for $j=-2.$ \label{devissage_for_K}
\item The map $KH^{\cont}(\cA)\to KH^{\cont}(\cB)$ is an equivalence of spectra. \label{devissage_for_KH}
\end{enumerate}
\end{theo}

\begin{proof}
\ref{devissage_for_K} The assumption implies that the induced functor $\check{D}(\cB)\to\check{D}(\cA)$ is $(-1)$-coconnective. Hence, the statement of \ref{devissage_for_K} is a special case of Theorem \ref{th:coconnectivity_estimates}.

\ref{devissage_for_KH} Consider more generally a $(-1)$-coconnective functor $F:\cC\to\cD$ between dualizable $t$-categories. We need to prove that the map $KH^{\cont}(F)$ is an equivalence. Arguing as in the proof of Theorem \ref{th:coconnectivity_estimates} we reduce to the special case when the monad $F^R\circ F$ is isomorphic to a deformed tensor algebra of some $G:\cC\to\cC$ such that $G[1]$ is left $t$-exact. Arguing as in loc. cit. we find a dualizable $t$-category $\cE$ and a $(-2)$-coconnective functor $\Theta:\cC\to\cE$ such that $\Cone(KH^{\cont}(F))\cong \Cone(KH^{\cont}(\Theta))[1].$ The latter spectrum is zero by Theorem \ref{th:theorem_of_the_heart_for_KH} since we have $\cC^{\heartsuit}\simeq\cE^{\heartsuit}$ by Proposition \ref{prop:t_structure_on_dualizable_via_functor_from_abelian}. 
\end{proof}

We will see in the next section that in the situation of Theorem \ref{th:devissage_for_K_and_KH} the map $K_{-2}^{\cont}(\cB)\to K_{-2}^{\cont}(\cA)$ does not have to be surjective, even when $\cA$ and $\cB$ are locally coherent and linear over a field.

\begin{remark}
One can obtain a more precise version of part \ref{devissage_for_KH} of Theorem \ref{th:devissage_for_K_and_KH} when every object of $\cA$ is an extension of two objects of $\cB.$ Under this assumption a generalization of \cite[Theorem 0.3]{E25b} shows that the functor $F:\check{D}(\cB)\to\check{D}(\cA)$ satisfies the following property: the monad $F^R\circ F$ is isomorphic to a deformed tensor algebra of some continuous exact functor $G:\check{D}(\cB)\to\check{D}(\cB)$ such that $G[1]$ is left $t$-exact. By Proposition \ref{prop:statements_S_m_and_T_m} for any $m\geq 0,$ $k\geq 0$ and $c=-2^{m+1}+2k$ the map
\begin{equation*}
\tau_{\geq c}\Cone(U_k^{\cont}(F))\to\tau_{\geq c}\Cone(U_{k+2^m}^{\cont}(F))
\end{equation*}
is null-homotopic. Here $U_k=\Fil_k KH$ ($k\geq 0$) is the localizing invariant defined in \eqref{eq:definition_of_U_k}. 
\end{remark}

\section{Sharpness of the estimates}
\label{sec:sharpness_of_estimates}

\subsection{Estimates for theorems of the heart and d\'evissage}

In this section we prove the sharpness of our estimates from Corollary \ref{cor:theorem_of_the_heart_under_condition_on_Exts} and Theorem \ref{th:devissage_for_K_and_KH} \ref{devissage_for_K}, even when we are dealing with small dg resp. abelian categories over a field.

\begin{theo}\label{th:sharpness_of_estimates} Let $\mk$ be a field.
\begin{enumerate}[label=(\roman*),ref=(\roman*)]
\item  There exist small $\mk$-linear abelian categories $\cB\subset \cA$ such that the (fully faithful) inclusion functor $\cB\to\cA$ is exact, each object of $\cA$ is an extension of two objects of $\cB,$ and the map $K_{-2}(\cB)\to K_{-2}(\cA)$ is not surjective. \label{sharp_for_devissage}
\item For any $n\geq 1$ there exists a $\mk$-linear small $t$-category $\cC_n$ such that the maps $\Ext_{\cC_n^{\heartsuit}}^i(x,y)\to\Ext_{\cC_n}^i(x,y)$ are isomorphisms for $i\leq n$ and $x,y\in\cC_n^{\heartsuit},$ and the map $K_{-n-2}(\cC_n^{\heartsuit})\to K_{-n-2}(\cC_n)$ is not surjective. \label{sharp_for_th_of_the_heart}
\end{enumerate}
\end{theo}

We will need the following lemma, which is certainly known to experts, but we do not know a reference.

\begin{lemma}\label{lem:non_positive_nil_groups_for_artinian}
Let $\cE$ be an additive $\infty$-category, such that for any object $x\in\cE$ the ring $\pi_0\End_{\cE}(x)$ is left or right artinian. Then we have $N^s K_j(\cE)=0$ for $s\geq 1,$ $j\leq 0.$
\end{lemma} 

\begin{proof}
Since the functors $N^s K_j$ commute with filtered colimits, it suffices to consider the case when $\cE$ is equivalent to $\Proj_{f.g.}\hy A,$ the category of finite projective right modules over a connective $\bE_1$-ring $A,$ such that $\pi_0(A)$ is left or right artinian. Let $s\geq 1$ and $j\leq 0.$ Then by \cite[Theorem 9.53]{BGT} and by the nilinvariance of $K_{\leq 0}$ we have $N^s K_j(A)\cong N^s K_j(\pi_0(A)/J)$ where $J\subset \pi_0(A)$ is the Jacobson radical. The ring $\pi_0(A)/J$ is a finite product of matrix algebras over skew-fields, hence we may assume that $A = D$ is a skew-field. In this case all the spectra $N^s K(D)$ vanish for $s\geq 1.$ 
\end{proof}

In the following proof we minimize the computations, focusing only on proving the sharpness results.

\begin{proof}[Proof of Theorem \ref{th:sharpness_of_estimates}]
\ref{sharp_for_devissage} Let $X=\{x_1^3 = x_0 x_2^2\}\subset \PP^2_{\mk}$ be the cuspidal cubic curve over $\mk.$ Denote by $\cE=\Vect_X$ the category of vector bundles of finite rank on $X.$ Denote by $\Ac^b(\cE)$ the dg category of acyclic bounded complexes of objects of $\cE.$ By \cite[Lemma 1.2]{Nee21} this category has a bounded $t$-structure whose heart $\cB_0=\Ac^b(\cE)^{\heartsuit}$ is equivalent to the category $\eff(\cE)$ of effaceable functors $\cE^{op}\to\Ab$ (this is a special case of Proposition \ref{prop:t_structure_on_Ac}).
% consists of short exact sequences placed in degrees $0,1,2$ (homologically):
%\begin{equation}\label{eq:short_exact_acyclic_complex}
%\dots\to 0\to\cF_2\to\cF_1\to\cF_0\to 0\to\dots
%\end{equation}
%We recall that we have a fully faithful exact inclusion $\cB_0\to\Fun^{\add}(\cE,\Ab)$ which sends the object \eqref{eq:short_exact_acyclic_complex} to $\coker(\h_{\cF_1}\to \h_{\cF_0}).$ Here we use the usual notation $\h_{\cF}=\Hom(-,\cF)$ for $\cF\in\cE.$ This identifies $\Ac^b(\cE)^{\heartsuit}$ with the category $\eff(\cE)$ of effaceable functors $\cE^{op}\to\Ab.$ 

Denote by $\Nil(\cB_0)$ the abelian category of pairs $(x,f),$ where $x\in\cB_0$ and $f:x\to x$ is a nilpotent endomorphism. We have a fully faithful inclusion $i:\cB_0\to\Nil(\cB_0),$ $i(x)=(x,0),$ and we identify $\cB_0$ with its essential image. More generally, for $k\geq 0$ denote by $\cB_k\subset\Nil(\cB_0)$ the (abelian) full subcategory of pairs $(x,f)$ such that $f^k=0.$ Then the inclusion functors are exact and each object of $\cB_{k+1}$ is an extension of two objects of $\cB_k.$

We claim that for some $k\geq 0$ the map $K_{-2}(\cB_k)\to K_{-2}(\cB_{k+1})$ is not surjective, which would prove \ref{sharp_for_devissage}. It suffices to prove that the map $K_{-2}(\cB_0)\to K_{-2}(\Nil(\cB_0))$ is not surjective. 

By Proposition \ref{prop:t_structure_on_nilpotent_endomorphisms} \ref{realization_is_equivalence_for_nil_endomorphisms} we have $\coker(K_{-2}(\cB_0)\to K_{-2}(\Nil(\cB_0)))\cong NK_{-1}(\cB_0),$ so we need to show that this nil group is non-zero. By \cite[Proposition 2.4]{Nee21} the realization functor $D^b(\cB_0)\to \Ac^b(\cE)$ is an equivalence. Denoting by $\cE^{\add}$ the category $\cE$ with a split exact structure, we obtain a short exact sequence
\begin{equation*}
0\to D^b(\cB_0)\to \St(\cE^{\add})\to \Perf(X)\to 0.
\end{equation*}
By Lemma \ref{lem:non_positive_nil_groups_for_artinian} we have $NK_{\leq 0}(\cE^{\add})=0.$ Hence, $NK_{-1}(\cB_0)\cong NK_0(X).$ The latter group is non-zero, since we have a surjection
\begin{equation*}
NK_0(X)\onto \coker(\Pic(X)\to \Pic(X\times \A^1))\cong x\mk[x].
\end{equation*} 
This proves \ref{sharp_for_devissage}.

\ref{sharp_for_th_of_the_heart} Take any pair $\cB\subset\cA$ as in \ref{sharp_for_devissage} such that the map $K_{-2}(\cB)\to K_{-2}(\cA)$ is not surjective. We will construct inductively the small $\mk$-linear $t$-categories $\cC_n$ with required properties, so that $\cC_n^{\heartsuit}\simeq \cB$ for all $n\geq 1.$

First we construct $\cC_1.$ We use the construction from \cite[Theorem 0.3]{E25b} and recall the assertions of loc. cit. We denote by $\cE_{\cA,\cB}$ the quasi-abelian category of pairs $(x,y),$ where $x\in\cA$ and $y\subset x$ is a subobject such that both $y$ and $x/y$ are in $\cB.$ The forgetful functor $\cE_{\cA,\cB}\to\cA,$ $(x,y)\mapsto x,$ induces a quotient functor on bounded derived categories. We define $\cC_1$ to be its kernel, so we have a short exact sequence
\begin{equation}\label{eq:ses_for_devissage}
0\to \cC_1\xto{\Phi} D^b(\cE_{\cA,\cB})\xto{q} D^b(\cA)\to 0.
\end{equation}
We have a semi-orthogonal decomposition $D^b(\cE_{\cA,\cB})=\la i_1(D^b(\cB)), i_2(D^b(\cB))\ra.$ Here the fully faithful functors $i_1,i_2:D^b(\cB)\to D^b(\cE_{\cA,\cB})$ are induced by the exact functors $\cB\to\cE_{\cA,\cB},$ which are given respectively by $x\mapsto (x,0),$ $x\mapsto (x,x).$ The category $\cC_1$ has a bounded $t$-structure, and we have an equivalence $\cB\xto{\sim}\cC_1^{\heartsuit},$ given by $x\mapsto\Cone((x,0)\to (x,x)).$

Now, the cofiber sequence of $K$-theory spectra for \eqref{eq:ses_for_devissage} is of the form
\begin{equation*}
K(\cC_1)\to K(\cB)\oplus K(\cB)\to K(\cA).
\end{equation*} 
Both functors $q\circ i_1$ and $q\circ i_2$ are isomorphic to the derived functor of the inclusion $\cB\to\cA.$ Hence, by assumption the map $K_{-2}(\cB)\oplus K_{-2}(\cB)\to K_{-2}(\cA)$ is not surjective. It follows that the map $K_{-3}(i_2^R\circ \Phi):K_{-3}(\cC_1)\to K_{-3}(\cB)$ is not a monomorphism, where $i_2^R$ is the right adjoint to $i_2.$ But the functor $i_2^R\circ \Phi:\cC_1\to D^b(\cB)$ is left inverse to the realization functor, hence the map $K_{-3}(\cC_1^{\heartsuit})\to K_{-3}(\cC_1)$ is not surjective, as required.

Now suppose that for some $n\geq 1$ we have constructed a $\mk$-linear small $t$-category $\cC_n$ with required properties such that the map $K_{-n-2}(\cC_n^{\heartsuit})\to K_{-n-2}(\cC_n)$ is not surjective. We will construct the $t$-category $\cC_{n+1}$ as follows. Put $\cD=\cC_n,$ and consider the exact 
category $\cD_{[0,n]}\subset\cD$ with the induced exact structure. Consider another exact category $\wt{\cD}_{[0,n]}$ from Proposition 
\ref{prop:properties_of_tilde_C} with the same underlying category as $\cD_{[0,n]}$ (we use the small category version of loc. cit.). Recall from loc. cit. that a morphism $f:x\to y$ in $\wt{\cD}_{[0,n]}$ is an inclusion if and only if both 
$\pi_{n-1}(f)$ and $\pi_n(f)$ are monomorphisms. We define $\cC_{n+1}$ via the short exact sequence
\begin{equation}\label{eq:ses_for_St_D_0_n}
0\to\cC_{n+1}\xto{\Phi} \St(\wt{\cD}_{[0,n]})\xto{q} \St(\cD_{[0,n]})\to 0.
\end{equation}

Our assumption on $\cD=\cC_n$ implies that the realization functor $D^b(\cD^{\heartsuit})\to\St(\cD_{[0,n-1]})$ is an equivalence: this is Proposition \ref{prop:when_partial_realization_is_an_equivalence} applied to $\Ind(\cD).$ Hence, by Proposition \ref{prop:properties_of_tilde_C} we have a semi-orthogonal decomposition $D^b(\wt{\cD}_{[0,n]})=\la j_1(D^b(\cD^{\heartsuit}), j_2(D^b(\cD^{\heartsuit})))\ra.$ Here the functors $j_1$ and $j_2$ are induced by the exact functors $\cD^{\heartsuit}\to \wt{\cD}_{[0,n]},$ given by $x\mapsto x$ resp. $x\mapsto x[n].$ The category $\cC_{n+1}$ has a bounded $t$-structure and we have an equivalence $\cD^{\heartsuit}\xto{\sim}\cC_{n+1}^{\heartsuit}.$ Moreover, by loc. cit. the maps $\Ext^i_{\cC_{n+1}^{\heartsuit}}(x,y)\to \Ext^i_{\cC_{n+1}}(x,y)$ are isomorphisms for $i\leq n+1,$ $x,y\in\cC_{n+1}^{\heartsuit}.$ 

Now the cofiber sequence of $K$-theory spectra for \eqref{eq:ses_for_St_D_0_n} is of the form
\begin{equation*}
K(\cC_{n+1})\to K(\cD^{\heartsuit})\oplus K(\cD^{\heartsuit})\to K(\cD_{[0,n]}).
\end{equation*}

Both $q\circ j_1$ and $q\circ j_2[-m]$ are isomorphic to the realization functor $D^b(\cD^{\heartsuit})\to \St(\cD_{[0,n]}).$ Now, by Theorem \ref{th:theorem_of_the_heart} we have an isomorphism $K_{-n-2}(\cD_{[0,n]})\xto{\sim} K_{-n-2}(\cD).$ Hence, our assumption on $\cD=\cC_n$ implies that the map $K_{-n-2}(\cD^{\heartsuit})\oplus K_{-n-2}(\cD^{\heartsuit})\to K_{-n-2}(\cD_{[0,n]})$ is not surjective. Hence, the map $K_{-n-3}(j_2^R\circ \Phi):K_{-n-3}(\cC_{n+1})\to K_{-n-3}(\cD^{\heartsuit})$ is not a monomorphism, where $j_2^R$ is right adjoint to $j_2.$ The functor $j_2^R\circ \Phi$ is left inverse to the realization functor, hence the map $K_{-n-3}(\cC_{n+1}^{\heartsuit})\to K_{-n-3}(\cC_{n+1})$ is not surjective, as required. This concludes the inductive step of the construction and proves the theorem. 
\end{proof}

\subsection{Estimates for higher nil groups of abelian categories}

We give a closer look at the abelian category $\cB_0$ from the proof of Theorem \ref{th:sharpness_of_estimates} in the characteristic zero case. We recall the localizing invariant $DK=\Fiber(K\to KH).$ For a small abelian category $\cA$ we put $DK(\cA)=DK(D^b(\cA)).$ We do the following computation, in particular proving the sharpness of the estimates from Corollary \ref{cor:coconnectivity_for_N^sK_of_t_categories}.

\begin{theo}\label{th:sharpness_for_N^sK}
Let $\mk$ be a field of characteristic zero. Let $X=\{x_1^3=x_0 x_2^2\}\subset\PP^2_{\mk}$ be the cuspidal cubic curve over $\mk,$ and let $\cE=\Vect_X$ be the category of vector bundles of finite rank on $X.$ Let $\cA=\eff(\cE)\simeq\Ac^b(\cE)^{\heartsuit}$ be the (small) abelian category of effaceable functors $\cE^{op}\to\Ab.$ 

\begin{enumerate}[label=(\roman*),ref=(\roman*)]
	\item For $s\geq 1$ we have
	\begin{equation}\label{eq:N^sK_computation}
		N^s K(\cA)\cong \HC^{-,\red}(\mk[x_1,\dots,x_s]/\mk)[-2].
	\end{equation}
	In particular, we have 
	\begin{equation*}
		N^s K_{s-2}(\cA)\cong \Omega^s_{\mk[x_1,\dots,x_s]/\mk}\ne 0.
	\end{equation*} \label{N^sK_of_effaceable}
\item We have
\begin{equation}\label{eq:DK_of_effaceable}
	DK_n(\cA)=\begin{cases}
		\mk & \text{for }n\geq -1\text{ odd;}\\
		0 & \text{else.}
	\end{cases}
\end{equation} \label{DK_of_effaceable}
\end{enumerate}
\end{theo}

We elaborate on the meaning of the right hand side of \eqref{eq:N^sK_computation}. First, by $\HC^-(-/\mk)$ we mean the negative cyclic homology over $\mk,$ considered as a spectrum \cite{Tsy83, FT85, Con94}. The superscript ``red'' means the reduced part in the simplicial sense, where we consider the assignment $[s]\mapsto \mk[x_1,\dots,x_s]$ as a functor $\Delta^{op}\to \CAlg(\Vect_{\mk}),$ see for example \cite[Section 8]{E25c}.

For the proof of Theorem \ref{th:sharpness_for_N^sK} we will need a certain short exact sequence, which might be known to some experts. We recall the notion of a lax pullback: if $F:\cA\to\cC$ and $G:\cB\to\cC$ are functors between small $\infty$-categories, then $\cA\laxtimes{\cB}\cC$ is the category of triples $(x,y,\varphi),$ where $x\in\cA,$ $y\in\cB$ and $\varphi:F(x)\to G(y).$ If $\cA,\cB,\cC$ are stable and $F, G$ are exact, then the category $\cA\laxtimes{\cC}\cB$ is also stable, and it has a semi-orthogonal decomposition $\cA\laxtimes{\cC}\cB=\la\cA,\cB\ra,$ where the inclusions are given by $x\mapsto (x,0,0)$ resp. $y\mapsto (0,y,0).$ The same applies to $\mk$-linear stable categories and $\mk$-linear exact functors. 

\begin{prop}\label{prop:ses_for_cat_res_of_cuspidal}
Let $X$ be as in Theorem \ref{th:sharpness_for_N^sK}. We identify the normalization $\wt{X}$ with $\PP^1_{\mk}$ with coordinate $x=\frac{x_2}{x_1}.$ We put $\mk[\veps]=\mk[x]/x^2,$ and consider the pullback functors $\Perf(\PP^1_{\mk})\to\Perf(\mk[\veps]),$ $\Perf(\mk)\to\Perf(\mk[\veps]).$ Consider the semi-free (associative) dg algebra $B=\mk\la x,y\ra,$ where $\deg(x)=0,$ $\deg(y)=1,$ $dx=0,$ $dy=x^2.$ Then we have a short exact sequence in $\Cat_{\mk}^{\perf}:$
\begin{equation}\label{eq:ses_cat_res_for_cuspidal}
0\to \Perf(X)\xto{F} \Perf(\mk)\laxtimes{\Perf(\mk[\veps])} \Perf(\PP^1_{\mk}) \to \Perf(B)\to 0.
\end{equation}
\end{prop}

\begin{proof}
The fully faithful functor $F$ in \eqref{eq:ses_cat_res_for_cuspidal} is the familiar categorical resolution of singularities from \cite{KL15}. It is induced by the pullback functors: to the singular point $\Perf(X)\to\Perf(\mk)$ and to the normalization $\Perf(X)\to\Perf(\PP^1_{\mk}).$ The essential image of $F$ is in fact the strict pullback.

Put $\cC=\Perf(\mk)\laxtimes{\Perf(\mk[\veps])} \Perf(\PP^1_{\mk}).$ We need to compute the quotient $\cC/F(\Perf(X)).$ First take the affine chart $U=X\setminus\{(0:0:1)\}.$ Then $\Perf(U)$ is a quotient of $\Perf(X),$ and $F$ induces a fully faithful functor
\begin{equation*}
\bbar{F}:\Perf(U)\hto \Perf(\mk)\laxtimes{\Perf(\mk[\veps])} \Perf(\mk[x])=:\bbar{\cC}.
\end{equation*}
We have an equivalence
\begin{equation*}
\cC/F(\Perf(X))\xto{\sim}\bbar{\cC}/\bbar{F}(\Perf(U)).
\end{equation*}
We consider the objects
\begin{equation*}
P_1=(\mk,\mk[x],\id)\in\bbar{\cC},\quad P_2=(0,\mk[x],0)\in\bbar{\cC}.
\end{equation*}
Clearly, $P_1\cong \bbar{F}(\cO_U),$ and together $P_1$ and $P_2$ generate $\bbar{\cC}$ (as an idempotent-complete stable subcategory). Hence, it suffices to compute the endomorphism dg algebra of the image of $P_2$ in the quotient of $\bbar{\cC}$ by (the stable subcategory generated by) $P_1$. 

First, the dg algebra $\End_{\bbar{\cC}}(P_1\oplus P_2)$ is discrete, i.e. its homology vanishes in non-zero degrees. Hence, we may identify it with its $H_0.$ Putting $R=\mk[x^2,x^3]\subset\mk[x],$ we have
\begin{equation*}
\Hom_{\bbar{\cC}}(P_1,P_1)=R,\quad \Hom_{\bbar{\cC}}(P_1,P_2)=x^2\mk[x],
\end{equation*}
\begin{equation*}
\Hom_{\bbar{\cC}}(P_2,P_1)=\mk[x],\quad \Hom_{\bbar{\cC}}(P_2,P_2)=\mk[x].
\end{equation*}
Denote by $\bbar{P_2}$ the image of $P_2$ in $\cC'=\bbar{\cC}/\la P_1\ra.$ By \cite{Dr04} the dg algebra $\End_{\cC'}(\bbar{P_2})$ is quasi-isomorphic to
\begin{equation*}
B'=\mk[x]\oplus\biggplus[n\geq 0](x^2\mk[x]h)\otimes (R h)^{\otimes n}\otimes\mk[x],
\end{equation*}
where $h$ has degree $1,$ the multiplication is straightforward and the differential is given by $d(x^k h)=x^k.$ Note that the cokernel of $\mk[x]\to B$ is identified with the shifted non-reduced bar complex computing the derived tensor product $x^2\mk[x]\Ltens{R}\mk[x].$ Since the non-reduced bar complex is quasi-isomorphic to the reduced bar complex, we have a quasi-isomorphism
\begin{equation*}
B\xto{\sim} B',\quad x\mapsto x,\quad y\mapsto x^2 h.
\end{equation*}
This proves the proposition.
\end{proof}

We recall from \cite{Ke64} that for an ordinary small additive category $\cB$ its Jacobson radical is a sub-bifunctor 
\begin{equation*}
\rad_{\cB}\subset\cB(-,-):\cB^{op}\times\cB\to\Ab,
\end{equation*} 
\begin{equation*}
\rad_{\cB}(x,y)=\{f\in\cB(x,y)\mid \forall g\in\cB(y,x)\text{ the morphism }1_x+gf\text{ is invertible}\}.
\end{equation*}
We denote by $\cB/\rad$ the category with the same objects as in $\cB,$ where the morphisms are given by
\begin{equation*}
(\cB/\rad)(x,y)=\cB(x,y)/\rad_{\cB}(x,y).
\end{equation*}
We will need a simple observation.

\begin{prop}\label{prop:vanishing_of_N^sK_of_quotient_by_radical}
Let $\mk$ be a field and suppose that $\cB$ is an ordinary idempotent-complete $\mk$-linear small category with finite-dimensional morphisms. Then we have $N^s K(\cB/\rad) = 0$ for $s\geq 1.$
\end{prop}

\begin{proof}
Indeed, the category $\cB/\red$ is equivalent to a direct sum of (a possibly infinite collection of) categories of finite-dimensional vector spaces over skew-fields. For any skew-field $D$ we have $N^s K(D) = 0$ for $s\geq 1.$
\end{proof}

Now Theorem \ref{th:sharpness_for_N^sK} is easily deduced from the following lemma. We first clarify the notation. For a $\Q$-linear stable category (or for a $\Q$-algebra, or for a scheme over $\Q$) we denote by $\HH$ the Hochschild homology over $\Q,$ and similarly for $\HC^-,$ $\HC$ and $\HP.$ We will also use the notation $N^s \HC^-$ for the higher nil versions, so for $\cC\in\Cat_{\Q}^{\perf}$ we have 
\begin{equation*}
N^s \HC^-(\cC)=\HC^{-,\red}(\cC\otimes\Perf(\Q[x_1,\dots,x_s])),\quad s\geq 0.
\end{equation*}
We also note that for $s\geq 1$ the $\Q S^1$-module $HH^{\red}(\Q[x_1,\dots,x_s])$ is isomorphic to a direct sum of shifts of $\Q S^1,$ which implies an isomorphism
\begin{equation*}
N^s \HC^-(\cC)\cong HH(\cC)\otimes \HC^{-,\red}(\Q[x_1,\dots,x_s]),\quad s\geq 1,\,\cC\in\Cat^{\perf}_{\Q}.
\end{equation*}

\begin{lemma}\label{lem:computing_N^sK_of_effaceable}
Let $X,$ $\cE=\Vect_X$ and $\cA=\eff(\cE)$ be as in Theorem \ref{th:sharpness_for_N^sK}. As in Proposition \ref{prop:ses_for_cat_res_of_cuspidal} we identify the normalization $\wt{X}$ with $\PP^1_{\mk},$ and denote by $\wt{\cE}=\Vect_{\PP^1_{\mk}}$ the category of vector bundles on $\PP^1_{\mk}$ of finite rank.
\begin{enumerate}[label=(\roman*),ref=(\roman*)]
	\item Consider the universal finitary localizing invariant of $\Q$-linear idempotent-complete stable categories $\cU_{\loc}:\Cat_{\Q}^{\perf}\to \Mot^{\loc}_{\Q}.$ Then we have an isomorphism $\cU_{\loc}(\wt{\cE}^{\add})\xto{\sim}\cU_{\loc}(\wt{\cE}^{\add}/\rad).$ This gives a natural commutative square in the category $D(\Q)^{B S^1}\simeq D(\Q S^1):$
	\begin{equation}\label{eq:square_of_HH}
		\begin{tikzcd}
			\HH(\cE^{\add})\ar[r]\ar[d] & \HH(X)\ar[d]\\
			\HH(\cE^{\add}/\rad) \ar[r] & \HH(\PP^1_{\mk}).
		\end{tikzcd}
	\end{equation}
	\label{square_of_HH}
\item Denote by $V$ the total cofiber of \eqref{eq:square_of_HH}, considered as a complex of $\Q$-vector spaces (forgetting the $S^1$-action). Then for $s\geq 1$ we have a natural isomorphism
\begin{equation}\label{eq:N^sK_of_A}
N^s K(\cA)\cong V[-2]\otimes \HC^{-,\red}(\Q[x_1,\dots,x_s]).
\end{equation}
\label{higher_nik_K_theory_via_nil_HC^-}
\end{enumerate}
\end{lemma}

%Again, in \eqref{eq:N^sK_of_A} the superscript ``red'' means the reduced part in the simplicial sense.

\begin{proof}
\ref{square_of_HH} By Grothendieck's theorem every vector bundle on $\PP^1_{\mk}$ is a finite direct sum of line bundles $\cO(n).$ Hence, as a $\mk$-linear dg category $\St(\wt{\cE}^{\add})$ has a $\Z$-indexed full exceptional collection $\{\cO(n)\}_{n\in\Z}.$ Now, the Jacobson radical of $\wt{\cE}^{\add}$ is generated by the morphisms $\cO(n)\to\cO(n')$ for $n<n'.$ Hence, the category $\St(\wt{\cE}^{\add}/\rad)$ has a (completely orthogonal) $\Z$-indexed full exceptional collection formed by the images of $\cO(n).$ It follows that the functor $\St(\wt{\cE})\to\St(\wt{\cE}^{\add}/\rad)$ induced an isomorphism on $\cU_{\loc}.$ In particular, the map $\HH(\wt{\cE}^{\add})\to\HH(\wt{\cE}^{\add}/\rad)$ is an isomorphism. The square \eqref{eq:square_of_HH} is obtained from the diagram
\begin{equation*}
\begin{tikzcd}
\HH(\cE^{\add}/\rad)\ar[d] & \HH(\cE^{\add})\ar[l]\ar[d]\ar[r] & \HH(X)\ar[d]\\
\HH(\wt{\cE}^{\add}/\rad) & \HH(\wt{\cE}^{\add})\ar[swap]{l}{\sim}\ar[r] & \HH(\PP^1_{\mk}).
\end{tikzcd}
\end{equation*}

\ref{higher_nik_K_theory_via_nil_HC^-} Let $s\geq 1$ be an integer. Recall that by \cite[Proposition 2.4]{Nee21} we have an equivalence $D^b(\cA)\simeq\Ac^b(\cE).$ Hence, we have a cofiber sequence of spectra
\begin{equation}\label{eq:cofiber_sequence_for_N^sK_of_A}
N^s K(\cA)\to N^s K(\cE^{\add})\to N^s K(X).
\end{equation}
By Proposition \ref{prop:vanishing_of_N^sK_of_quotient_by_radical} and by the Dundas-Goodwillie-McCarthy theorem \cite[Theorem 7.0.0.2]{DGM13} we have the isomorphisms
\begin{multline}\label{eq:N^sK_of_E^add}
N^s K(\cE^{\add})\cong \Fiber(N^s K(\cE^{\add})\to N^s K(\cE^{\add}/\rad))\\
\cong \Fiber(N^s \HC^-(\cE^{\add})\to N^s\HC^-(\cE^{\add}/\rad))\\
\cong \Fiber(\HH(\cE^{\add})\to\HH(\cE^{\add}/\rad))\otimes\HC^{-,\red}(\Q[x_1,\dots,x_s]).
\end{multline}
Next, recall the semi-free dg $\mk$-algebra $B$ from Proposition \ref{prop:ses_for_cat_res_of_cuspidal}: it is given by $B=\mk\la x, y\ra,$ where $\deg(x) = 0,$ $\deg(y) = 1,$ and $d x = 0,$ $d y = x^2.$ Applying loc. cit. and \cite[Theorem 7.0.0.2]{DGM13} again we obtain
\begin{multline}\label{eq:N^sK_of_X}
N^s K(X)\cong N^s K(B)[-1]\cong \Fiber(N^s K(\mk)\to N^s K(B))\\
\cong \Fiber(N^s \HC^-(\mk)\to N^s \HC^-(B))\cong\Fiber(N^s \HC^-(X)\to N^s \HC^-(\PP^1_{\mk}))\\
\cong \Fiber(\HH(X)\to\HH(\PP^1_{\mk}))\otimes\HC^{-,\red}(\Q[x_1,\dots,x_s]). 
\end{multline}
Combining \eqref{eq:N^sK_of_E^add} and \eqref{eq:N^sK_of_X} together with the cofiber sequence \eqref{eq:cofiber_sequence_for_N^sK_of_A} we obtain the isomorphism \eqref{eq:N^sK_of_A}.
\end{proof}

\begin{proof}[Proof of Theorem \ref{th:sharpness_of_estimates}]
\ref{N^sK_of_effaceable} Consider the complex of $\Q$-vector spaces $V$ from Lemma \ref{lem:computing_N^sK_of_effaceable}. By loc. cit. it suffices to prove that $V\cong \mk.$ Clearly, we have $\Cone(\HH(\cE^{\add})\to\HH(\cE^{\add}/\rad))\in D(\Q)_{\geq 1}.$ Next, by Proposition \ref{prop:ses_for_cat_res_of_cuspidal} for the dg $\mk$- algebra $B$ from loc. cit. we have
\begin{equation*}
\Cone(\HH(X)\to\HH(\PP^1_{\mk}))\cong \Cone(\HH(\mk)\to\HH(B))\in D(\Q)_{\geq 0}.
\end{equation*}
Moreover, $H_0$ of the latter complex is isomorphic to $\mk.$ Hence, we have $V\in D(\Q)_{\geq 0}$ and $H_0(V)\cong \mk.$ Applying Lemma \ref{lem:computing_N^sK_of_effaceable} for $s = 1$ we obtain
\begin{equation*}
NK(\cA)\cong V[-2]\otimes\HC^{-,\red}(\Q[x_1])\cong V[-1]\otimes\Omega^1_{\Q[x_1]/\Q}.
\end{equation*}
By Corollary \ref{cor:coconnectivity_for_N^sK_of_t_categories} we have $NK(\cA)\in\Sp_{\leq -1}.$ We conclude that $V\in D(\Q)_{\leq 0},$ and by the above we obtain $V\cong\mk.$ This proves \ref{N^sK_of_effaceable}.

\ref{DK_of_effaceable} now follows almost directly from \ref{N^sK_of_effaceable}. Namely, we first obtain an isomorphism
\begin{equation*}
DK(\cA)\cong \Fiber(\HC^-(\mk/\mk)[-2]\to \colim\limits_{[n]\in\Delta^{op}}\HC^-(\mk[x_1,\dots,x_n]/\mk)[-2]).
\end{equation*}
Since $\HP(\mk[x_1,\dots,x_n]/\mk)\cong HP(\mk/\mk)$ for $n\geq 0,$ we obtain
\begin{multline*}
\Fiber(\HC^-(\mk/\mk)[-2]\to \colim\limits_{[n]\in\Delta^{op}}\HC^-(\mk[x_1,\dots,x_n]/\mk)[-2])\\
\cong \Fiber(\HC(\mk/\mk)[-1]\to\colim\limits_{[n]\in\Delta^{op}}\HC(\mk[x_1,\dots,x_n]/\mk)[-1]))\cong\HC(\mk/\mk)[-1].
\end{multline*}
The latter isomorphism follows from the well-known vanishing of the colimit $\colim\limits_{[n]\in\Delta^{op}}\HH(\mk[x_1,\dots,x_n]/\mk):$ this is an $\bE_{\infty}$-algebra in which the unit element is homotopic to zero. Therefore we obtain
\begin{equation*}
DK(\cA)\cong\HC(\mk/\mk)[-1],
\end{equation*}
which gives \eqref{eq:DK_of_effaceable}.
\end{proof}

\begin{remark}
It follows directly from the \cite[Theorem 9.3]{E25c} that for any $\cC\in\Cat_{\Q}^{\perf}$ the spectrum $DK(\cC)$ is naturally a $\Q[[u]]$-module, where $u$ has cohomological degree $2$ (recall that $TC(\Q)\cong\Q[[u]]$). Moreover, it follows from \cite[Theorem 9.1]{E25c} that $u$ is acting locally nilpotently on $DK(\cC),$ since by loc. cit. we have $\pi_{-2}\End_{\Mot^{\loc}_{\Q}}(\wt{\cU}_{\loc}(\Q[x]))=0.$ The proof of Theorem \ref{th:sharpness_for_N^sK} shows that for the abelian category $\cA=\eff(\Vect_X)$ we have an isomorphism $DK(\cA)\cong (\mk((u))/u \mk[[u]])[-1]$ of $\Q[[u]]$-modules.
\end{remark}

\section{Examples of compactly assembled $t$-structures and coherently assembled abelian categories}
\label{sec:examples_of_comp_ass_and_coh_ass}

\subsection{Sheaves on locally compact Hausdorff spaces}

For simplicity we consider only the continuously bounded $t$-structures, and we only consider the categories of sheaves with values in a fixed category (not in a presheaf of categories like in \cite[Sections 6.2-6.4]{E24}).

\begin{prop}\label{prop:t_structure_on_sheaves}
Let $X$ be a compact Hausdorff space and let $\cC$ be a dualizable $t$-category. 

\begin{enumerate}[label=(\roman*),ref=(\roman*)]
	\item There is a compactly assembled continuously bounded $t$-structure on the dualizable category of presheaves $\PSh(X;\cC),$ such that $\PSh(X;\cC)_{\geq 0}=\PSh(X;\cC_{\geq 0}).$  \label{t_structure_PSh}
	\item There is a unique compactly assembled continuously bounded $t$-structure on the category $\Shv(X;\cC)$ such that the sheafification functor $\PSh(X;\cC)\to \Shv(X;\cC)$ is $t$-exact. Moreover, we have $\Shv(X;\cC)^{\heartsuit}\simeq\Shv(X;\cC^{\heartsuit}).$ \label{t_structure_Sh}
\item If $\cA$ is a coherently assembled abelian category, then the category $\Shv(X;\cA)$ is also coherently assembled. \label{Sh_is_coherently_assembled}
\end{enumerate}
\end{prop}

\begin{proof}
\ref{t_structure_PSh} follows directly from Proposition \ref{prop:functors_from_poset_to_dualizable_t_category}.

\ref{t_structure_Sh} It is well-known that there is a unique $t$-structure on $\Shv(X;\cC)$ such that the sheafification functor $\cF\mapsto\cF^{\sharp}$ is $t$-exact. Moreover, this $t$-structure is compatible with filtered colimits. To finish the proof, by Proposition \ref{prop:comp_ass_t_structure_via_retracts} it suffices to construct a functor $\Phi:\Shv(X;\cC)\to\PSh(X;\cC),$ right inverse to the sheafification functor, such that $\Phi$ is continuous, exact and left $t$-exact. For a closed subset $Z\subset X$ denote by $i_Z:Z\to X$ the inclusion. We put 
\begin{equation*}
	\Phi(\cF)(U)=\Gamma_c(X,i_{\bar{U},*}(i_{\bar{U}}^*(\cF))),\quad \cF\in\Shv(X;\cC),\quad U\subset X\text{ open}.
\end{equation*}
Then $\Phi$ commutes with filtered colimits since so do $i_{\bar{U}}^*,$ $i_{\bar{U},*}$ and $\Gamma_c(X,-).$ For the same reason $\Phi$ is left $t$-exact. It remains to show that $\Phi$ is right inverse to the sheafification.

To see this, consider another (accessible exact) functor $\Psi:\Shv(X;\cC)\to\PSh(X;\cC),$ given by
\begin{equation*}
\Psi(\cF)(U)=\Gamma(\bar{U},i_{\bar{U}}^*(\cF)),\quad \cF\in\Shv(X;\cC)\quad U\subset X\text{ open}.
\end{equation*}
We have natural transformations $\Phi\to \Psi$ and $\Psi\to\incl,$ where $\incl$ is the inclusion functor. By construction, for $\cF\in\Shv(X;\cC)$ and for any compact $Z\subset X$ we have
\begin{equation*}
\Gamma(Z,\Cone(\Phi(\cF)^{\sharp}\to \Psi(\cF)^{\sharp})_{\mid Z})=0,\quad \Gamma(Z,\Cone(\Psi(\cF)^{\sharp}\to \cF)_{\mid Z})=0.
\end{equation*}
Hence, the maps $\Phi(\cF)^{\sharp}\to\Psi(\cF)^{\sharp}$ and $\Psi(\cF)^{\sharp}\to\cF$ are isomorphisms of sheaves, as required.

Finally, \ref{Sh_is_coherently_assembled} follows from \ref{t_structure_Sh} applied to $\cC=\check{D}(\cA).$
\end{proof}

\begin{cor}\label{cor:K_-1_not_vanishing}
There exist coherently assembled abelian categories for which the group $K_{-1}^{\cont}$ is non-zero. For example, let $\mk$ be a field, and take $\cA=\Shv(\R;\Vect_{\mk}).$ Then $\cA$ is coherently assembled and $K_{-1}^{\cont}(\cA)=\Z.$ 
\end{cor}

\begin{proof}
By Proposition \ref{prop:t_structure_on_sheaves} the category $\cA$ is coherently assembled. The computation of $K_{-1}^{\cont}(\cA)$ is a special case of \cite[Proposition 4.18]{E24}: we have
\begin{equation*}
K_{-1}^{\cont}(\cA)\cong K_{-1}^{\cont}(\Shv(\R;D(\mk)))\cong K_0(\mk)\cong\Z.\qedhere
\end{equation*}
\end{proof}

\subsection{Nuclear modules over $\Z_p$}

Consider the abelian category $\Solid_{\Z_p}$ of light solid modules over $Z_p$ from \cite{CS20, CS24}. It is generated by a single compact projective object $P=\prodd[\N]\Z_p.$ We denote by $\Nuc^{CS}(\Z_p)\subset D(\Solid_{\Z_p})$ the original version of nuclear solid modules over $\Z_p,$ whose description will be recalled in the proof of the next proposition. This category is known to be dualizable (but not compactly generated). Recall from \cite{E25a} that we have a bigger dualizable category $\Nuc(\Z_p)\supset\Nuc^{CS}(\Z_p),$ such that the inclusion functor is strongly continuous and we have
\begin{equation*}
\Nuc(\Z_p)\simeq\prolim[n]^{\dual}D(\Z/p^n).
\end{equation*}
Here the inverse limit is taken in $\Cat_{\st}^{\dual}.$ Recall that we have
\begin{equation*}
\Nuc(\Z_p)^{\omega}\simeq\Nuc^{CS}(\Z_p)^{\omega}\simeq\Perf(\Z_p).
\end{equation*}
The latter category has a standard $t$-structure since $\Z_p$ is regular. Our goal in the next proposition is to show that this $t$-structure extends to a compactly assembled continuously bounded $t$-structure on $\Nuc^{CS}(\Z_p).$ One can show by a more difficult argument that it further extends to a compactly assembled $t$-structure on $\Nuc(\Z_p)$ with the same heart, but the latter $t$-structure is not continuously bounded. 

The first part of the following proposition is known but we include the proof for completeness. %It is known that $A$ is right coherent.  

\begin{prop}\label{prop:t_structure_on_Nuc_of_Z_p}
\begin{enumerate}[label=(\roman*),ref=(\roman*)]
\item The category $\Solid_{\Z_p}$ is locally coherent. The abelian category $\Solid_{\Z_p}^{\omega}$ has homological dimension $1,$ and we have $\check{D}(\Solid_{\Z_p})\simeq D(\Solid_{\Z_p}).$ \label{solid_is_locally_coherent}
\item The standard $t$-structure on $D(Solid_{\Z_p})$ induces a compactly assembled continuously bounded $t$-structure on $\Nuc^{CS}(\Z_p).$ Moreover, the realization functor $\check{D}(\Nuc^{CS}(\Z_p)^{\heartsuit})\to \Nuc^{CS}(Z_p)$ is an equivalence. \label{nuc_is_coherently_assembled}
\end{enumerate}
\end{prop}

\begin{proof}
\ref{solid_is_locally_coherent} Consider the compact projective generator $P\in\Solid_{\Z_p}$ as above. Putting $A=\End(P),$ we have an equivalence $D(\Solid_{\Z_p})\cong D(A).$ Recall that the morphisms $P\to P$ are simply the continuous maps of topological abelian groups $\prodd[\N]\Z_p\to \prodd[\N]\Z_p.$ It follows that we have an isomorphism $A^{op}\cong\End_{\Z}((\biggplus[\N]\Z)^{\wedge}_p).$ Recall that we have an equivalence $D_{p\hy\compl}(\Z)\xto{\sim}D_{p\hy\tors}(\Z),$ sending $(\biggplus[\N]\Z)^{\wedge}_p$ to $\biggplus[\N]\Q_p/\Z_p[-1].$ Hence, we have an isomorphism $A^{op}\cong \End(\biggplus[\N]\Q_p/\Z_p).$ Now, the object $I=\biggplus[\N]\Q_p/\Z_p$ is injective in the small abelian category $\cA=(\Mod_{p\hy\tors}\hy\Z)^{\omega_1}.$ Moreover, any object of $\cA$ can be embedded into $I.$ It follows that the ring $A$ is right coherent and we have equivalences $\Solid_{\Z_p}^{\omega}\simeq \Coh\hy A\cong \cA^{op}.$ In particular, these categories have homological dimension $1.$ Since $\Solid_{\Z_p}$ also has a compact projective generator, the equivalence $\check{D}(\Solid_{\Z_p})\simeq D(\Solid_{\Z_p})$ follows for example from \cite[Proposition C.5.8.12]{Lur18}.

\ref{nuc_is_coherently_assembled} Within the above notation, denote by $J\subset A$ the ideal of trace-class maps in the sense of \cite[Definition 13.11]{CS20}. Under the above identification $A^{op}\cong\End_{\Z}((\biggplus[\N]\Z)^{\wedge}_p),$ the ideal $J^{op}\subset A^{op}$ consists of compact operators, i.e. maps $X:(\biggplus[\N]\Z)^{\wedge}_p\to (\biggplus[\N]\Z)^{\wedge}_p$ such that for each $n>0$ the map $X\otimes\Z/p^n:\biggplus[\N]\Z/p^n\to \biggplus[\N]\Z/p^n$ has finite rank. We have a short exact sequence
\begin{equation*}
0\to \Nuc^{CS}(\Z_p)\xto{i} D(A)\to D(A/J)\to 0.
\end{equation*}
By \ref{solid_is_locally_coherent}, the standard $t$-structure on $D(A)$ is compactly assembled and continuously bounded. By Corollary \ref{cor:induced_t_structure_is_compactly_assembled} it suffices to prove that the functor $i\circ i^R:D(A)\to D(A)$ is $t$-exact, where $i^R$ is right adjoint to $i.$ %Assuming this for the moment, we see that the assertions about the $t$-structure on $\Nuc^{CS}(\Z_p)$ follow. The realization functor is automatically an equivalence, since we have $\Nuc^{CS}(\Z_p)^{\heartsuit}\simeq \Mod\hy A/\Mod\hy (A/J).$

%It remains to prove the $t$-exactness claim. 
We have $i\circ i^R\cong -\tens{A}J.$ Hence, we need to prove that $J$ is flat as a left $A$-module. This is known, but we give a proof for completeness. It is convenient to deal with the opposite algebra, i.e. with endomorphisms of $(\biggplus[\N]\Z)^{\wedge}_p.$ We consider the poset 
\begin{equation*}
S= \{f:\N\to\N \mid \lim\limits_{n\to\infty}f(n)=\infty\}
\end{equation*} with reverse order: $f\leq g$ if $f(n)\geq g(n)$ for $n\geq 0.$ Then $S$ is directed. For each such $f\in S$ consider the operator 
\begin{equation*}
X_f:(\biggplus[\N]\Z)^{\wedge}_p\to (\biggplus[\N]\Z)^{\wedge}_p,\quad X_f(e_n)=p^{f(n)}e_n\text{ for }n\geq 0.
\end{equation*}
Here $\{e_n\}_{n\geq 0}$ is the topological basis of $(\biggplus[\N]\Z)^{\wedge}_p.$ It is easy to see that $J^{op}$ as a right ideal in $A^{op}$ is given by the directed union:
\begin{equation*}
J^{op}=\bigcup\limits_{f:\N\to\N} X_f\cdot A^{op},
\end{equation*}
where $f$ runs through functions as above. Now, each $X_f$ (as an operator) has zero kernel, and for $f\leq g$ we have $X_f\in X_g\cdot A^{op}.$ It follows that $J^{op}$ is flat as a right $A^{op}$-module, i.e. $J$ is flat as a left $A$-module. This proves the proposition.
\end{proof}

\subsection{Chase criterion}
\label{ssec:chase_criterion}

In this subsection we consider abelian categories in the setting of the following proposition. In fact they are exactly the dualizable $\Ab$-modules in $\Pr^L,$ as shown in \cite{LLS}. They are also exactly the categories described by Roos' theorem \cite{Roo65}. Namely, these are categories of the form $\Mod\hy R/\Mod\hy(R/I),$ where $R$ is an associative unital ring and $I\subset R$ is an ideal such that $I^2=I.$ 

\begin{prop}\label{prop:dualizable_Ab_modules}
Let $\cC$ be a compactly assembled additive ordinary category. Consider the category $\cA=\Fun^{\cont,\add}(\cC,Ab)$ of continuous additive functors. Then $\cA$ is a Grothendieck abelian category, satisfying (AB6) and (AB4*). Moreover, we have an equivalence
\begin{equation}\label{eq:Lazard_generalized}
\cC\xto{\sim}\Fun^{\cont,\ex}(\cA,\Ab),\quad x\mapsto (F\mapsto F(x)).
\end{equation}
Here the target is the category of continuous exact functors.
\end{prop}

\begin{proof}
If $\cC$ is compactly generated, then $\cA\simeq\Fun^{\add}(\cC^{\omega},\Ab),$ and the stated axioms hold. The equivalence \eqref{eq:Lazard_generalized} is essentially a reformulation of Lazard's theorem. Indeed, we have an equivalence $\Fun^L(\cA,\Ab)\simeq\Fun^{\add}((\cC^{\omega})^{op},\Ab),$ and the full subcategory of flat functors $(\cC^{\omega})^{op}\to \Ab$ is identified with $\Ind(\cC^{\omega})\simeq\cC.$

Now consider the general case. Restriction to $\cC^{\omega_1}$ gives a fully faithful functor $\Phi:\cA\hto \Fun^{\add}(\cC^{\omega_1},\Ab).$ Its right adjoint $\Phi^R$ is induced by the functor $\hat{\cY}:\cC\to\Ind(\cC^{\omega_1}).$ Explicitly, for $x\in\cC$ with $\hat{\cY}(x)=\inddlim[i]x_i\in\Ind(\cC^{\omega_1}),$ we have
\begin{equation*}
\Phi^R(F)(x)=\indlim[i]F(x_i),\quad F\in\Fun^{\add}(\cC^{\omega_1},\Ab).
\end{equation*}
In particular, $\Phi^R$ is exact and commutes with filtered colimits. Since (AB5), (AB6) and (AB4*) hold in $\Fun^{\add}(\cC^{\omega_1},\Ab),$ they also hold in $\cA.$ It is also clear that $\cA$ has a generator, namely $\Phi^R(G),$ where $G$ is a generator in $\Fun^{\add}(\cC^{\omega_1},\Ab).$

Finally, note that by the above observations the functor \eqref{eq:Lazard_generalized} is a retract of the functor
\begin{equation*}
\Ind(\cC^{\omega_1})\to \Fun^{\cont,\ex}(\Fun(\cC^{\omega_1},\Ab),\Ab),
\end{equation*}
which is an equivalence by Lazard's theorem.	 
\end{proof}

The following result is a generalization of Chase criterion of left coherence of an associative unital ring $R.$ For completeness we give essentially a self-contained proof, which is purely categorical and covers the original theorem \cite[Theorem 2.1]{Ch60}. More precisely, the result in loc. cit. corresponds to the case when $\cC=\Flat\hy R$ is the (compactly generated) category of flat right $R$-modules and $\cA=R\hy\Mod$ is the abelian category of left $R$-modules.

\begin{theo}\label{th:Chase}
Let $\cC$ be a compactly assembled additive ordinary category, and let $\cA=\Fun^{\cont,\add}(\cC,\Ab).$ The following are equivalent.
\begin{enumerate}[label=(\roman*),ref=(\roman*)]
	\item $\cA$ is coherently assembled. \label{Chase_coherently_assembled}
	\item $\cC$ has infinite products. \label{Chase_products_of_flat}
\end{enumerate}
\end{theo}

\begin{proof} \Implies{Chase_coherently_assembled}{Chase_products_of_flat}. By Proposition \ref{prop:dualizable_Ab_modules} we need to show that the category $\Fun^{\cont,\ex}(\cA,\Ab)$ has infinite products, provided that $\cA$ is coherently assembled. Restriction to $\cA^{\omega_1}$ defines a fully faithful functor
\begin{equation*}
\Psi:\Fun^{\cont,\ex}(\cA,\Ab)\hto \Fun^{\ex}(\cA^{\omega_1},\Ab).
\end{equation*}
Now, exactness of $\hat{\cY}_{\cA}:\cA\to\Ind(\cA^{\omega_1})$ implies that $\Psi$ has a right adjoint, given by
\begin{equation*}
\Psi^R(F)(x)=\indlim[i]F(x_i),\quad F\in\Fun^{\ex}(\cA^{\omega_1},\Ab),\,x\in\cA,\,\hat{\cY}_{\cA}(x)=\inddlim[i]x_i.
\end{equation*}
Since the category of abelian groups satisfies (AB4*) (exactness of infinite products), the category $\Fun^{\ex}(\cA^{\omega_1},\Ab)$ has infinite products. Hence, so does the category $\Fun^{\cont,\ex}(\cA,\Ab),$ as required.

\Implies{Chase_products_of_flat}{Chase_coherently_assembled} We first consider the case when $\cC$ is compactly generated, and we identify $\cA\simeq\Fun^{\add}(\cC^{\omega},\Ab).$ Choose a sufficiently large regular cardinal $\kappa$ with the following properties: the set of isomorphism classes of objects in $\cC^{\omega}$ is $\kappa$-small, and for any $x,y\in\cC^{\omega}$ the group $\Hom_{\cC^{\omega}}(x,y)$ is $\kappa$-small. Then the category $\cA^{\kappa}$ is abelian. Moreover, this small category has a collection of projective generators. More precisely, let $\cB\subset\cA^{\kappa}$ be the full (additive) subcategory of objects of the form $\biggplus[j\in J]\h_{x_j}^{\vee},$ where $J$ is a $\kappa$-small set, $x_j\in\cC^{\omega}$ and $\h_{x_j}^{\vee}=\Hom(x_j,-)\in\cA^{\omega}$ for $j\in J.$ Then we have an equivalence $\Ind(\cA^{\kappa})\simeq\Fun^{\add}(\cB^{op},\Ab).$ We identify $\cC^{\omega}$ with a full subcategory of $\cB^{op}$ via $x\mapsto \h_x^{\vee}.$

Now, the colimit functor $\colim:\Ind(\cA^{\kappa})\to\cA$ is identified with the precomposition functor $\Theta:\Fun^{\add}(\cB^{op},\Ab)\to\Fun^{\add}(\cC^{\omega},\Ab),$ which is exact and commutes with all colimits. Thus, it suffices to prove that the left adjoint functor $\Theta^L:\Fun^{\add}(\cC^{\omega},\Ab)\to\Fun^{\add}(\cB^{op},\Ab)$ is exact. Indeed, this would imply that $\cA$ is coherently assembled since so is $\Ind(\cA^{\kappa}).$

Now, exactness of $\Theta^L$ means that for any $P\in\cB$ the functor $F_P:(\cC^{\omega})^{op}\to\Ab,$ $F_P(x)=\Hom_{\cB}(P,\h_x^{\vee}),$ is flat. But this holds by assumption: if $P=\biggplus[j\in J]\h_{x_j}^{\vee},$ then the product $\prodd[j\in J]x_j$ exists in $\cC,$ which exactly means that the functor $F_P\cong \prodd[j]\Hom(-,x_j)$ is flat. This proves that $\cA$ is coherently assembled, or equivalently locally coherent.

Now consider the general case. We first observe that the category $\Ind(\cC^{\omega_1})$ has infinite products. Indeed, it suffices to show that for a collection of objects $\{x_j\in\cC^{\omega_1}\}_{j\in J}$ their product exists in $\Ind(\cC^{\omega_1}).$ This product is computed as follows: we first take the product of $x_j$ in $\cC$ (which exists by assumption), and then take its image under the functor $\cC\simeq\Ind_{\omega_1}(\cC^{\omega_1})\subset\Ind(\cC^{\omega_1}).$

By the above special case we know that the category $\Fun^{\add}(\cC^{\omega_1},\Ab)$ is coherently assembled. By the proof of Proposition \ref{prop:dualizable_Ab_modules} we have a retraction $\cA\to\Fun^{\add}(\cC^{\omega_1},\Ab)\to\cA,$ where both functors are exact and commute with all colimits. Therefore, by Proposition \ref{prop:coherently_assembled_as_retract} the category $\cA$ is coherently assembled, as required.
\end{proof}

\subsection{Concluding remarks}
\label{ssec:concluding_remarks}

We covered only some very basic non-trivial examples of compactly assembled $t$-structures and coherently assembled abelian categories. We briefly mention the following directions.

\begin{itemize}
	\item One can consider abelian categories of almost modules $\cA=\Mod\hy R/\Mod\hy (R/I),$ where $R$ is a discrete associative unital ring, and $I\subset R$ is a two-sided ideal such that $I^2 = I.$ It is interesting to figure out some basic sufficient conditions (for example, in the commutative case) which imply that $\cA$ is coherently assembled. If this is the case, the functor $\check{D}(\cA)\to D(\cA)$ does not have to be an equivalence, and one can consider the continuous $K$-theory $K^{\cont}(\cA)$ as $G$-theory of the pair $(R,I).$ When $I=R,$ this means that $R$ is right coherent and we obtain the usual $G$-theory of $R.$
	
	\item One can consider the categories of nuclear modules for more general adic spaces, in particular for formal schemes. For example, if $R$ is a regular noetherian commutative ring with an ideal $I\subset R,$ we expect that the category $\Nuc^{CS}(R^{\wedge}_I)$ has a compactly assembled continuously bounded $t$-structure. If $R$ is not regular, there should still be a continuously bounded $t$-structure with a coherently assembled heart, so one can consider $G$-theory in this framework. We also expect nice $t$-structures in the archimedean setting, say for nuclear gaseous modules over various gaseous rings.
	
    \item One can construct a dualizable $t$-category starting from a coherently assembled exact category $\cE.$ Namely, (slightly abusing the notation) denote by $\cE^{\add}$ the same category with an exact structure from Proposition \ref{prop:minimal_exact_structure}. Then we have a short exact sequence in $\Cat_{\st}^{\dual}:$
    \begin{equation*}\label{eq:ses_check_Ac_of_E}
    	0\to \check{\Ac}(\cE)\to\check{\St}(\cE^{\add})\to\check{\St}(\cE)\to 0.
    \end{equation*}
    One can show that the category $\check{\Ac}(\cE)$ has a natural compactly assembled continuously bounded $t$-structure, induced by the natural $t$-structure on $\check{\St}(\cE^{\add}).$ If $\cE$ is locally coherent, then $\check{\Ac}(\cE)\simeq\Ind(\Ac(\cE^{\omega})),$ and the induced $t$-structure on $\Ac(\cE^{\omega})$ coincides with the (bounded) $t$-structure from Proposition \ref{prop:t_structure_on_Ac}. In general the heart $\check{\Ac}(\cE)^{\heartsuit}$ is equivalent to the category of sequential limit-preserving additive functors $F:(\cE^{\omega_1})^{op}\to\Ab$ such that for any $x\in\cE^{\omega_1},$ for any $\alpha\in F(x)$ and for any compact morphism $f:y\to x$ with $y\in\cE^{\omega_1},$ there exists an exact projection $g:z\to y$ in $\cE^{\omega_1}$ such that $F(fg)(\alpha)=0.$
    
   	\item Let $R$ be an associative unital ring, and consider the abelian category $\cA=\Mod\hy R$ of right $R$-modules. By Proposition \ref{prop:coh_ass_comp_gen}, $\cA$ is coherently assembled if and only if $R$ is right coherent. In general, one can consider the left derived functors $L_k\hat{\cY}:\cA\to\Ind(\cA),$ and define the non-coherence rank of $R$ to be 
   	\begin{equation*}
   	n=\sup\{k\geq 0\mid L_k\hat{\cY}\ne 0\}\in\N\cup\{\infty\}.
   	\end{equation*}
   	Suppose that  $0<n<\infty.$ One can show that the category $\check{D}(\cA)$ is dualizable and we are in the situation of Question \ref{ques:non_coherence_rank_positive}. One can define $G$-theory of $R,$ putting $G(R)=K^{\cont}(\check{D}(\cA)).$ We are not aware of examples when $\check{D}(\cA)$ is not compactly generated (under the assumption $0<n<\infty$).
\end{itemize}

\end{document}